\RequirePackage{fix-cm}
\documentclass{amsart} 
\usepackage[utf8]{inputenc}
\usepackage{mathptmx}
\usepackage{latexsym}
\usepackage{amsmath}
\usepackage{amssymb}
\usepackage{amsfonts} 
\usepackage{eucal} 
\usepackage{amsbsy}
\usepackage{amsthm}
\usepackage[all]{xy}
\usepackage{hyperref}
\usepackage{graphicx}
\usepackage{xcolor}

\newtheorem{thm}{Theorem}[section]
\newtheorem*{thm*}{Theorem}
\newtheorem{lemma}[thm]{Lemma}
\newtheorem{prop}[thm]{Proposition}
\newtheorem{cor}[thm]{Corollary}
\newtheorem*{cor*}{Corollary}

\theoremstyle{definition}
\newtheorem{defn}[thm]{Definition}
\newtheorem{example}[thm]{Example}

\theoremstyle{remark}
\newtheorem{remark}[thm]{Remark}

\newcommand {\Aa}    {\ensuremath{\mbox{$\mathcal{A}$}}}
\newcommand {\Ba}    {\ensuremath{\mbox{$\mathcal{B}$}}}
\newcommand {\Ca}    {\ensuremath{\mbox{$\mathcal{C}$}}}

\newcommand {\Ea}    {\ensuremath{\mbox{$\mathcal{E}$}}}
\newcommand {\Fa}    {\ensuremath{\mbox{$\mathcal{F}$}}}

\newcommand {\Ha}    {\ensuremath{\mbox{$\mathcal{H}$}}}

\newcommand {\Ja}    {\ensuremath{\mbox{$\mathcal{J}$}}}

\newcommand {\Sa}    {\ensuremath{\mbox{$\mathcal{S}$}}}
\newcommand {\Ta}    {\ensuremath{\mbox{$\mathcal{T}$}}}

\newcommand {\Xa}    {\ensuremath{\mbox{$\mathcal{X}$}}}
\newcommand {\Ya}    {\ensuremath{\mbox{$\mathcal{Y}$}}}

\newcommand {\mbl}   {\ensuremath{\mathbb{L}}}

\newcommand {\real}  {\ensuremath{\mathbb{R}}}
\newcommand {\intg}  {\ensuremath{\mathbb{Z}}}

\newcommand {\cplx}  {\ensuremath{\mathbb{C}}}
\newcommand {\rat}   {\ensuremath{\mathbb{Q}}}

\newcommand {\Hom}   {\ensuremath{\operatorname{Hom}}}
\newcommand {\Tor}   {\ensuremath{\operatorname{Tor}}}

\newcommand {\im}    {\operatorname{im}}
\newcommand {\rk}    {\operatorname{rk}}

\newcommand {\omwt}  {\ensuremath{\Omega^{\operatorname{Witt}}}}

\newcommand {\ch}    {\ensuremath{\operatorname{ch}}}

\newcommand {\smlhf} {\ensuremath{\mbox{$\frac{1}{2}$}}}

\newcommand {\pro}   {\ensuremath{\operatorname{pr}}}

\newcommand {\redh}  {\ensuremath{\widetilde{H}}}

\newcommand {\syml}  {\ensuremath{\mathbb{L}^{\bullet}}}
\newcommand {\redsyml}  {\ensuremath{\widetilde{\mathbb{L}^\bullet}}}

\newcommand {\cone}  {\ensuremath{\operatorname{cone}}}

\newcommand {\pt}    {\ensuremath{\operatorname{pt}}}
\newcommand {\id}    {\ensuremath{\operatorname{id}}}

\newcommand {\Ob}   {\ensuremath{\operatorname{Ob}}}

\newcommand {\BF}   {\ensuremath{\operatorname{BF}}}

\newcommand {\BG}   {\ensuremath{\operatorname{BG}}}

\newcommand {\BSO}   {\ensuremath{\operatorname{BSO}}}

\newcommand {\BSPL}   {\ensuremath{\operatorname{BSPL}}}

\newcommand {\BBSPL}   {\ensuremath{\operatorname{B}\widetilde{\operatorname{SPL}}}}
\newcommand {\BBPL}   {\ensuremath{\operatorname{B}\widetilde{\operatorname{PL}}}}
\newcommand {\BSTOP}   {\ensuremath{\operatorname{BSTOP}}}

\newcommand {\MSO}   {\ensuremath{\operatorname{MSO}}}

\newcommand {\RMSO}   {\ensuremath{\widetilde{\operatorname{MSO}}}}
\newcommand {\MSPL}   {\ensuremath{\operatorname{MSPL}}}

\newcommand {\RMSPL}   {\ensuremath{\widetilde{\operatorname{MSPL}}}}
\newcommand {\MSTOP}   {\ensuremath{\operatorname{MSTOP}}}

\newcommand {\RMSTOP}   {\ensuremath{\widetilde{\operatorname{MSTOP}}}}
\newcommand {\SO}   {{\ensuremath{\operatorname{SO}}}}
\newcommand {\SPL}   {{\ensuremath{\operatorname{SPL}}}}
\newcommand {\PL}   {{\ensuremath{\operatorname{PL}}}}
\newcommand {\TPL}   {{\ensuremath{\widetilde{\operatorname{PL}}}}}
\newcommand {\STOP}   {{\ensuremath{\operatorname{STOP}}}}

\newcommand {\TOP}   {{\ensuremath{\operatorname{TOP}}}}

\newcommand {\PLB}   {{\ensuremath{\operatorname{PLB}}}}

\newcommand {\Witt}   {\ensuremath{\operatorname{Witt}}}
\newcommand {\MWITT}   {\ensuremath{\operatorname{MWITT}}}

\newcommand {\Th}   {\ensuremath{\operatorname{Th}}}

\newcommand {\Oo}    {\ensuremath{\mathcal{O}}}

\newcommand {\pr}  {\ensuremath{\mathbb{P}}}

\newcommand {\alg}    {\ensuremath{{\operatorname{alg}}}}
\newcommand {\BM}    {\ensuremath{{\operatorname{BM}}}}

\newcommand {\td}    {\ensuremath{{\operatorname{td}}}}
\newcommand {\MHM}    {\ensuremath{{\operatorname{MHM}}}}
\newcommand {\Sp}    {\ensuremath{\operatorname{Sp}}}
\newcommand {\cl}    {\ensuremath{\operatorname{cl}}}
\newcommand {\Per}    {\ensuremath{\operatorname{Per}}}
\newcommand {\qrat}    {\ensuremath{\operatorname{rat}}}
\newcommand {\Bl}    {\ensuremath{\operatorname{Bl}}}

\newcommand {\Sing}    {\ensuremath{\operatorname{Sing}}}

\begin{document}


\title[Gysin Restriction of Topological and Hodge-Theoretic
Characteristic Classes]{Gysin Restriction of Topological and Hodge-Theoretic
Characteristic Classes for Singular Spaces}

\author{Markus Banagl}

\address{Mathematisches Institut, Universit\"at Heidelberg,
  Im Neuenheimer Feld 205, 69120 Heidelberg, Germany}

\email{banagl@mathi.uni-heidelberg.de}

\thanks{This work is supported by Deutsche Forschungsgemeinschaft (DFG) 
 under Germany’s Excellence Strategy EXC-2181/1 - 390900948 
 (the Heidelberg STRUCTURES Cluster of Excellence).}

\date{October, 2019}




\begin{abstract}
We establish formulae
that explain how the topological Goresky-MacPherson 
characteristic L-classes as well as the Hodge-theoretic 
Hirzebruch characteristic classes
defined by Brasselet, Sch\"urmann and Yokura
transform
under Gysin restrictions associated to normally nonsingular
embeddings of singular spaces. We find that both types of classes
transform in the same manner. We give a first application of these
formulae in obtaining algebraic rigidity results for topologically
homeomorphic projective varieties.
\end{abstract}

\maketitle


\tableofcontents


\section{Introduction}

We establish Verdier-Riemann-Roch type formulae
that describe the behavior of both topological 
characteristic $L$-classes and Hodge-theoretic characteristic classes
under Gysin restrictions associated to normally nonsingular
embeddings of singular spaces.

In \cite{hirzebruch}, Hirzebruch introduced cohomological $L$-classes $L^*$ for smooth
manifolds as certain polynomials with rational coefficients in the
tangential Pontrjagin classes. In view of the signature theorem,
Thom described the Poincar\'e duals of these $L$-classes
by organizing the signatures of submanifolds with trivial normal
bundle into a homology class, using global transversality and
bordism invariance of the signature.
For oriented compact polyhedral pseudomanifolds $X$, stratified without
strata of odd codimension, Goresky and MacPherson employed their
intersection homology to obtain a bordism invariant signature
and thus, using Thom's method, $L$-classes
$L_i (X)\in H_i (X;\rat)$ in (ordinary) homology, \cite{gmih1}.
It turned out later that one can move well beyond spaces with
only strata of even codimension:
A first step was the work \cite{siegel} of Siegel, yielding
$L$-classes for Witt spaces. A general treatment of
$L$-classes for arbitrary pseudomanifolds has been given 
in \cite{banagl-mem} and \cite{banagl-lcl},
where a local obstruction theory in terms of Lagrangian
structures along strata of odd codimension is described.

Due on the one hand to their close relation to the normal invariant
map on the set of homotopy smoothings of a Poincar\'e complex,
and on the other hand 
to their remarkable invariance under homeomorphisms, discovered by
Novikov, Hirzebruch's $L$-classes have come to occupy a central role
in high-dimensional manifold classification theory.
A particularly striking illustration is a classical result of
Browder and Novikov, which can readily be deduced from the
smooth surgery exact sequence:
a closed, smooth, simply connected manifold
of even dimension at least $5$ is determined, up to
finite ambiguity, by its homotopy type and its $L$-classes.
By work of Cappell and Weinberger (\cite{cw2}, \cite{weinberger}),
the Goresky-MacPherson $L$-class can be assigned a similar role
in the global classification of singular spaces, and it is still a topological invariant. 
But much less is known about the transformational laws that
govern its behavior, and this is reflected in the sparsity of
concrete calculations that have been carried out.
This is particularly true for complex algebraic varieties, where
the author is presently only aware of the formulae
obtained by Maxim and Schürmann in \cite{ms} for
simplicial, projective toric varieties. (Actually, such varieties are orbifolds,
and hence rational homology manifolds, so it is really the Thom-Milnor
$L$-class of a rational homology manifold that is being calculated, and
intersection homology is not required there.)
A well-known general result in the algebraic setting is the decomposition 
theorem \cite{bbd}, but determining the supports in the decomposition, as well as resolving
the twisting of $L$-classes by
local systems, often present challenges in concrete examples.

This lack of effective computations can be traced back to
two main issues: First, and most obviously, the
missing naturality and limited multiplicativity stemming from the fact
that characteristic classes of singular spaces live in homology and cannot
generally be lifted to cohomology.
Second, triviality of normal bundles, germane to the topological nature of 
Goresky-MacPherson-Thom's construction, is in fact not very natural 
in algebraic geometry. For example, it prevents one in practice to build
$L$-class calculations involving
transverse planar sections of singular projective varieties.
From this point of view, one should thus seek to incorporate nontrivial
normal geometry into the singular $L$-class picture, 
and this is what we do in the first half of the present paper.
An oriented normally nonsingular inclusion $g: Y\hookrightarrow X$ 
of real codimension $c$ has a \emph{Gysin homomorphism} 
$g^!: H_* (X) \to H_{*-c} (Y)$
on ordinary homology.
Our Theorem \ref{thm.lclassgysin} asserts:
\begin{thm*}
Let $g: Y \hookrightarrow X$ be a normally nonsingular inclusion of closed
oriented even-dimensional piecewise-linear Witt pseudomanifolds. Let $\nu$ be the 
topological normal bundle of $g$.
Then
\[ g^! L_* (X) = L^* (\nu) \cap L_* (Y). \]
\end{thm*}
For smooth manifolds this is a straightforward consequence of naturality,
the Whitney product formula and the fact that Gysin restrictions map
the fundamental class to the fundamental class.
For singular spaces, our guiding philosophy is the following:
Drop down to ordinary homology as late as possible from more elevated theories such 
as $\syml$-homology, or better yet, bordism. Then on bordism it is
possible to see the relation geometrically, using in particular
geometric descriptions of
cobordism due to Buoncristiano-Rourke-Sanderson in terms of mock bundles.
Implementing this philosophy requires a
wide-ranging portfolio of methods and results, including
Ranicki's symmetric algebraic L-theory, Siegel's Witt bordism,
natural transformations from such singular bordism theories  
to L-theory as
introduced recently by Banagl-Laures-McClure \cite{blm}, and various
unblocked and blocked bundle theories and Thom spectra, notably
the aforementioned theory of mock bundles \cite{buonrs}.

In the course of carrying out this program, we prove that
the Witt-bordism Gysin map sends the Witt-bordism fundamental class of $X$ to the 
Witt-bordism fundamental class of $Y$,
$g^! [X]_{\Witt} = [Y]_{\Witt}$
(Theorem \ref{thm.gysinpreserveswittfundclass}).
Using this, we prove that the $\syml$-homology Gysin restriction
sends the $\syml (\rat)$-homology fundamental class of $X$ to the 
$\syml (\rat)$-homology fundamental class of $Y$,
$g^! [X]_\mbl = [Y]_\mbl$
(Theorem \ref{thm.gysinpreserveslhomfundclass}).
Finally, one arrives at the above theorem on $L$-classes essentially by
localizing at zero.

We apply the above $L$-class Gysin Theorem in proving two rigidity results,
Theorem \ref{thm.rigid} and Proposition \ref{cor.rigid}, for projective 
algebraic varieties.
We find that the signature of generic complex $2$-dimensional planar sections of
(possibly singular) projective varieties is invariant under homeomorphisms with
respect to the complex topology,
which are homologically compatible with the projective embeddings.
Calculations of the (real) codegree-$4$ Goresky-MacPherson $L$-class, pushed forward into the
homology of projective space, are obtained for projective varieties homeomorphic
to varieties regular in codimension $2$.\\

If $\xi$ is a complex vector bundle over a base space $B$ 
with Chern roots $a_i$, Hirzebruch had also defined a
generalized Todd class
$T^*_y (\xi) \in H^* (B)\otimes \rat [y],$ whose
specialization to $y=1$ is the $L$-class, $T^*_1 = L^*.$
Let $X$ be a possibly singular complex algebraic variety of pure dimension,
let $MHM (X)$ denote the abelian category of Morihiko Saito's algebraic mixed
Hodge modules on $X$ and
$K_0 (MHM (X))$ the associated Grothendieck group.
A motivic Hirzebruch class transformation 
\[ MHT_{y*} :K_0 (MHM(X)) \to H^\BM_{2*} (X) \otimes 
    \rat [y^{\pm 1}, (1+y)^{-1}] \] 
to Borel-Moore homology
has been defined by Brasselet, Sch\"urmann and Yokura
in \cite{bsy}, based on insights of Totaro \cite{totaro}.
Applying this to the mixed Hodge object $\rat^H_X$, one gets a homological characteristic class
$T_{y*} (X)= MHT_{y*} [\rat^H_X]$ such that
for $X$ smooth and $y=1$, $T_{1*} (X) = L_* (X)$.
For this reason, $T_{1*} (X)$ has been called the Hodge $L$-class of $X$.
However, examples of singular curves show that 
generally $T_{1*} (X) \not= L_* (X).$
This suggests applying $MHT_{y*}$
to the intersection Hodge module $IC^H_X$, which yields an
intersection generalized Todd class 
$IT_{y*} (X) = MHT_{y*} [IC^H_X [-\dim_\cplx X]].$
If $X$ is an algebraic rational homology manifold, then
$\rat^H_X [\dim_\cplx X] \cong IC^H_X$, so $IT_{y*} (X) = T_{y*} (X)$.

In the second half of the present paper, we prove that $IT_{1*}$
transforms under Gysin restrictions associated to suitably normally nonsingular
closed algebraic regular embeddings
in the same manner as the Goresky-MacPherson $L$-class
in the above Theorem. In the algebraic setting, one uses the algebraic normal
bundle of a regular embedding.
Since such a bundle need not
correctly reflect the topology near the subvariety, 
the Gysin result requires a tightness assumption (Definition \ref{def.tightemb}), which
holds automatically in transverse situations.
We introduce a condition called \emph{upward normal nonsingularity}
(Definition \ref{def.upwardlynns}), which requires for a tight regular embedding that
the exceptional divisor in the
blow-up relevant to deformation to the normal cone be normally nonsingular.
This holds in suitably transverse situations and is related to the
clean blow-ups of Cheeger, Goresky and MacPherson. 
Our Algebraic Gysin Theorem \ref{thm.it1classgysin} is:
\begin{thm*}
Let $X,Y$ be pure-dimensional compact complex algebraic varieties and
let $g: Y \hookrightarrow X$ 
be an upwardly normally nonsingular embedding.
Let $N = N_Y X$ be the algebraic normal bundle of $g$
and let $\nu$ denote the topological normal bundle of the
topologically normally nonsingular inclusion underlying $g$.
Then
\[ g^! IT_{1*} (X) = L^* (N) \cap IT_{1*} (Y)
      = L^* (\nu) \cap IT_{1*} (Y). \]
\end{thm*}
Viewed in conjunction, our Gysin theorems may be interpreted as
further evidence towards a conjectural equality $IT_{1*} = L_*$ (\cite[Remark 5.4]{bsy})
for pure-dimensional compact complex algebraic varieties.
Sections \ref{sec.lclassintro} -- \ref{sec.rigiditytheorems} deal with the topological L-class of 
piecewise-linear (PL) pseudomanifolds, whereas the remaining Sections
\ref{sec.hodgeclasses} and \ref{sec.hodgeundernnsincl} are concerned 
with the Hodge-theoretic class $IT_{1*}$.
These two parts can be read independently.

The behavior of the $L$-class for singular spaces under Gysin
transfers associated to finite degree covers is already
completely understood.
In \cite{banaglcovertransfer}, we showed that for a
closed oriented Whitney stratified pseudomanifold $X$ admitting Lagrangian structures
along strata of odd codimension (e.g. $X$ Witt)
and $p: X' \to X$ an orientation preserving topological covering map of finite degree,
the $L$-class of $X$ transfers to the $L$-class of the cover, i.e.
\[ p_! L_* (X) = L_* (X'), \]
where $p_!: H_* (X;\rat) \to H_* (X';\rat)$ is the transfer induced by $p$.
This enabled us, for example, to establish the above conjecture 
for normal connected complex projective $3$-folds $X$
that have at worst canonical singularities, trivial canonical divisor, and  
$\dim H^1 (X;\Oo_X)>0$. (Note that such varieties are rational homology manifolds.)
In the complex algebraic setting, results concerning the multiplicativity
of the $\chi_y$-genus (which in the smooth compact context
corresponds to the signature for $y=1$)
under finite covers were obtained by A. Libgober and L. Maxim in
\cite[Lemma 2.3]{maximlibgober}. J. Schürmann discusses going up-and-down techniques for the
behavior of the motivic Chern class transformation $MHC_y$
under \'etale morphisms in \cite[Cor. 5.11, Cor. 5.12]{schuermannmsri}.
Let $\sigma (X)$ denote the signature of a compact Witt space $X$.
If $X$ is a complex projective algebraic variety, then by Saito's 
intersection cohomology Hodge index theorem
(\cite{saito88}, \cite[Section 3.6]{mss}), $IT_{1,0} (X)=\sigma (X)=L_0 (X)$, that is,
the conjecture is known to hold in degree $0$. 
Furthermore, Cappell, Maxim, Schürmann and Shaneson \cite[Cor. 1.2]{cmssequivcharsing}
have shown that the conjecture holds for orbit spaces $X=Y/G$,
where $Y$ is a projective $G$-manifold and $G$ a finite group of algebraic
automorphisms.
The conjecture holds for simplicial projective toric varieties
\cite[Corollary 1.2(iii)]{ms} and for certain complex hypersurfaces with
isolated singularities \cite[Theorem 4.3]{cmss}.

Following the overall strategy introduced in the present paper, formulae
describing the behavior of the Goresky-MacPherson $L$-class and of
$IT_{1*}$ under transfer homomorphisms associated to fiber bundles with
nonsingular positive dimensional fiber can also be obtained, but 
deserve treatment in a separate paper.\\

\textbf{Acknowledgements.}
We express our gratitude to J\"org Sch\"urmann, whose 
thoughtful comments on an earlier version of this
paper lead to numerous improvements,
and to Laurentiu Maxim for
providing helpful information on certain aspects of mixed Hodge modules.

\section{The $L$-Class of a Pseudomanifold}
\label{sec.lclassintro}

If $\xi$ is a real vector bundle over a topological space $B$, let
\[ L^* (\xi) = L^0 (\xi) + L^1 (\xi) + L^2 (\xi) +\cdots \in H^{4*} (B;\rat),~
  L^0 (\xi)=1, \]
denote its cohomological Hirzebruch $L$-class with
$L^i (\xi) \in H^{4i} (B;\rat)$.
For a closed oriented smooth manifold $M$ of real dimension $n$,
\[ L_* (M^n) = L_n (M) + L_{n-4} (M) + L_{n-8} (M) + \cdots \]
denotes the Poincar\'e dual of the Hirzebruch $L$-class 
$L^* (M) = L^* (TM)$ 
of the tangent bundle $\xi = TM$ of $M$. Thus
\[ L_i (M)\in H_i (M;\rat),~ L_{n-4i} (M) = L^i (M)\cap [M]. \]
We have
\[ L_n (M) = L^0 (M)\cap [M] = 1\cap [M] = [M], \]
and if $M$ has real dimension $n=4k$, then
\[ \epsilon_* L_0 (M)= \epsilon_* (L^k (M)\cap [M]) = \sigma (M), \]
where $\sigma (M)$ denotes the signature of $M$.

Let $X$ be a compact oriented piecewise-linear (PL) pseudomanifold of dimension $n$.
Such a pseudomanifold can be equipped with a choice of PL stratification, and there is
a PL-intrinsic such stratification. 
Siegel called $X$ a \emph{Witt space} if the middle degree, lower middle perversity
rational intersection homology of even-dimensional links of strata vanishes, \cite{siegel}.
This condition turns out to be independent of the choice of PL stratification,
\cite[Section 2.4]{gmih2}.
A pure-dimensional complex algebraic variety can be Whitney stratified, and thus
PL stratified, without strata of odd dimension and is thus a Witt space.
Compact Witt spaces $X$ have homological $L$-classes
\[ L_i (X) \in H_i (X;\rat) \cong \Hom (H^i (X;\rat), \rat), \]
on which a cohomology class $\xi \in H^i (X;\rat),$ 
stably represented as $\xi = f^* (u),$ $f:X\to S^i$ transverse with regular value
$p\in S^i$,
evaluates to $\langle \xi, L_i (X) \rangle = \sigma (f^{-1} (p))$,
where $u\in H^i (S^i)$ is the appropriate generator and $\sigma$ denotes the signature.
Note that the transverse preimage $f^{-1} (p)$ is again a Witt space.
Using $L^2$-forms on the top stratum with respect to conical Riemannian metrics,
Cheeger gave a local formula for $L_* (X)$ in terms of eta-invariants of links, \cite{cheeger}.
Again $\epsilon_* L_0 (X) = \sigma (X)$ and if $X=M$ is a smooth manifold, then
$L_i (X)$ agrees with the above Poincar\'e duals $L_i (M)$ of Hirzebruch's class.

\section{Behavior of the $L$-Class Under Normally Nonsingular Inclusions}

Let $g:Y\hookrightarrow X$ be an inclusion of compact oriented stratified pseudomanifolds.
If the inclusion is normally nonsingular with trivial normal bundle, then,
by the very definition of the $L$-class, there is a clear relationship
between the $L$-classes of $X$ and $Y$.
In the projective algebraic situation, the triviality assumption on the
normal bundle is not very natural, and it is important to understand the
relationship of these characteristic classes for arbitrary normal bundles.
This will be accomplished in the present section by establishing a precise
formula (one of the main results of this paper)
involving the Gysin transfer associated to the normally nonsingular 
embedding $g$. The formula is motivated by the special case of a smooth
embedding of manifolds, where it is easily established (see below).

\begin{defn}
A \emph{topological stratification} of a topological space $X$
is a filtration 
\[ X=X_n \supset X_{n-1} \supset \cdots \supset X_1 \supset X_0 
      \supset X_{-1}=\varnothing \]
by closed subsets $X_i$ such that the difference sets $X_i - X_{i-1}$
are topological manifolds of pure dimension $i$. The connected components $X_\alpha$
of these difference sets are called the \emph{strata}.
We will often write stratifications as $\Xa = \{ X_\alpha \}$.
\end{defn}
The following definition is due to Siebenmann
\cite{siebenmann}; see also Schürmann \cite[Def. 4.2.1, p. 232]{schuermanntsscs}.
\begin{defn}
A topological stratification $\{ X_i \}$ of $X$
is called \emph{locally cone-like} if for all $x\in X_i - X_{i-1}$
there is an open neighborhood $U$ of $x$ in $X$,
a compact topological space $L$ with filtration
\[ L=L_{n-i-1} \supset L_{n-i-2} \supset \cdots \supset L_0 
      \supset L_{-1}=\varnothing, \]
and a filtration preserving homeomorphism
$U \cong \real^i \times \operatorname{cone}^\circ (L),$
where $\operatorname{cone}^\circ (L)$ denotes the open cone on $L$.
\end{defn}
Locally cone-like topological stratifications are also called \emph{cs-stratifications}.
We understand an \emph{algebraic stratification} of a complex algebraic variety $X$ to
be a locally cone-like topological stratification $\{ X_{2i} \}$ of $X$ such that
all subspaces $X_{2i}$ are closed algebraic subsets of $X$.
Complex algebraic Whitney stratifications are algebraic stratifications in this sense.

\begin{defn} \label{def.snns}
Let $X$ be a topological space with locally cone-like topological stratification
$\Xa = \{ X_\alpha \}$ and let $Y$ be any topological space.
An embedding $g:Y\hookrightarrow X$ is called
\emph{normally nonsingular} (with respect to $\Xa$), if
\begin{enumerate}
\item $\Ya := \{ Y_\alpha := X_\alpha \cap Y \}$ is a locally
 cone-like topological stratification of $Y$,
\item there exists a topological vector bundle $\pi: E\to Y$ and
\item there exists a (topological) embedding $j:E \to X$ such that
 \begin{enumerate}
  \item $j(E)$ is open in $X$,
  \item $j|_Y =g,$ and 
  \item the homeomorphism $j:E\stackrel{\cong}{\longrightarrow} j(E)$
    is stratum preserving, where the open set $j(E)$ is endowed with
    the stratification $\{ X_\alpha \cap j(E) \}$ and $E$ is endowed
    with the stratification $\Ea = \{ \pi^{-1} Y_\alpha \}$.
 \end{enumerate}
\end{enumerate}
\end{defn}
Note that the above stratification $\Ea$ of the total space $E$
is automatically topologically locally cone-like.

\begin{defn} \label{def.compstrat}
If $X$ and $Y$ are complex algebraic varieties and
$g:Y\hookrightarrow X$ a closed algebraic embedding whose underlying
topological embedding $g(\cplx)$ in the complex topology is normally
nonsingular, then we will call $g$ and $g(\cplx)$ \emph{compatibly stratifiable}
if there exists an algebraic stratification $\Xa$ of $X$ such that
$g(\cplx)$ is normally nonsingular with respect to $\Xa$ and the induced
stratification $\Ya$ is an algebraic stratification of $Y$.
\end{defn}

An oriented normally nonsingular inclusion $g: Y\hookrightarrow X$ 
of real codimension $c$ has a \emph{Gysin map} 
\[ g^!: H_* (X) \longrightarrow H_{*-c} (Y) \]
on ordinary singular homology, given as follows:
Let $u\in H^{c} (E,E_0)$ denote the Thom class in ordinary cohomology 
of the rank $c$ vector bundle $\pi: E\rightarrow Y$, where
$E_0 \subset E$ denotes the complement of the zero section in $E$.
Then $g^!$ is the composition
\[ H_k (X) \rightarrow H_k (X,X-Y)
  \overset{e_\ast}{\underset{\cong}{\leftarrow}} H_k (E,E_0)
  \overset{u\cap -}{\underset{\cong}{\rightarrow}} H_{k-c} (E)
  \overset{\pi_\ast}{\underset{\cong}{\rightarrow}} H_{k-c} (Y), \]
where we use the embedding $j:E\to X$ in defining the excision
isomorphism $e_*$.
For classes $x\in H^p (X),$ $y\in H_* (X),$ the formula
\begin{equation} \label{equ.gysinofcap} 
g^! (x\cap y) = g^\ast x \cap g^! y 
\end{equation}
holds, provided either $p$ or the real codimension $c$ is even
(\cite[Lemma 5, p. 613]{banaglcappshan}, \cite[Ch. V, \S 6.2]{boardman}).
In the special case of a smooth embedding $g: N\hookrightarrow M$ of
closed oriented even-dimensional smooth manifolds, the Gysin transfer
maps the fundamental class $[M]$ of $M$ to the fundamental class $[N]$ of $N$. 
Thus in this case, using naturality and the Whitney sum formula, and
with $\nu$ the normal bundle $E\to N$ of $N$ in $M$,
\[ 
g^! L_* (M)
   = g^* L^* (TM) \cap g^! [M] 
 = L^* (g^* TM) \cap [N] 
   = L^* (\nu) \cap L_* (N).
\]
(All involved classes lie in even degrees and hence no signs enter.)
In this section, we shall show that this relation continues to hold
for normally nonsingular inclusions of singular spaces.
Note that when the normal bundle is trivial, the formula becomes
$g^! L_* (M) = L_* (N),$ as it should be.
In fact, for trivial normal bundle, the relation $g^! L_* (X) = L_* (Y)$
was already known to Cappell and Shaneson \cite{cappshanstratifmaps} 
in the singular context, even for general Verdier self-dual complexes of sheaves.

Complex algebraic pure-dimensional varieties are Witt spaces in the
sense of Siegel \cite{siegel}. 
Bordism of Witt spaces, denoted by $\omwt_* (-)$, is a generalized
homology theory represented by a spectrum $\MWITT$.
For a (real) codimension $c$ normally nonsingular inclusion
$g: Y^{n-c} \hookrightarrow X^n$ of (compact, oriented) Witt spaces,
we will define a Gysin map
\[ g^!: \omwt_k (X) \longrightarrow \omwt_{k-c} (Y), \]
and we shall prove that it sends the Witt-orientation of $X$, represented by
the identity map, to
the Witt orientation of $Y$. This will then be applied in proving
the analogous statement for the $\syml (\rat)$-homology orientations,
using the full force of the machinery of Banagl-Laures-McClure \cite{blm}.
We write $\syml = \syml (\intg) = \syml \langle 0 \rangle (\intg)$ for Ranicki's
connected symmetric algebraic $L$-spectrum
with homotopy groups $\pi_n (\syml)=L^n (\intg),$ the symmetric
$L$-groups of the ring of integers. 
Localization $\intg \to \rat$ induces a map
$\syml (\intg) \to \syml (\rat)$ and $\pi_n (\syml (\rat))=L^n (\rat)$ with
\[ L^n (\rat) \cong \begin{cases}
 \intg \oplus (\intg/_2)^\infty \oplus (\intg/_4)^\infty,& n\equiv 0 (4) \\
   0,& n\not\equiv 0 (4).
\end{cases} \]
As far as cobordism is concerned, the idea is to
employ the framework of 
Buoncristiano-Rourke-Sanderson \cite{buonrs},
which provides a geometric description of 
\emph{co}bordism in terms of mock bundles, as well
as geometric descriptions of Thom classes in cobordism, and
cap products between cobordism and bordism.

\subsection{Thom Classes in Cobordism}
\label{sec.thomclso}

Our approach requires uniform notions of Thom spaces and Thom classes 
in cobordism for various types of bundle theories and cobordism theories.
This will now be set up.

The term \emph{fibration} will always mean Hurewicz fibration.
A \emph{sectioned fibration} is a pair $(\xi,s)$, where
$\xi$ is a fibration $p: E\to B$, $s:B\to E$ is a section of $p$,
and the inclusion of the image of $s$ in $E$ is a fiberwise cofibration
over $B$. Let $(S^n,*)$ be a pointed $n$-sphere.
An \emph{$(S^n,*)$-fibration} is a sectioned $S^n$-fibration $(\xi,s)$
such that $(p^{-1}(b), s(b))$ is pointed homotopy equivalent
to $(S^n,*)$ for every $b\in B$.
Such $(S^n,*)$-fibrations are classified by maps into a classifying
space $\BF_n$. In particular, over $\BF_n$, there is a universal 
$(S^n,*)$-fibration $\gamma^F_n$.
\begin{defn}
The \emph{Thom space} of an $(S^n,*)$-fibration $\alpha = (\xi,s)$ is defined to be
\[ \Th (\alpha) := E / s(B). \]
\end{defn}
(See Rudyak \cite{rudyak}.)
Let $(\xi,s), (\xi', s')$ be $(S^n,*)$-fibrations with $\xi, \xi'$
given by $p: E\to B,$ $p':E' \to B',$ respectively.
A \emph{morphism of $(S^n,*)$-fibrations}
$\phi: (\xi,s) \to (\xi', s')$ is a pair $\phi = (g,f)$, where
$f:B\to B'$ and $g: E \to E'$ are maps such that $p'\circ g = f\circ p$,
\[ g|: (p^{-1} (b), s(b)) \longrightarrow (p'^{-1} (f(b)), s' (f(b))) \]
is a pointed homotopy equivalence for all $b\in B$, and
$\phi$ respects the sections, i.e. 
$g\circ s = s' \circ f$. The composition of two morphisms  of
$(S^n,*)$-fibrations is again an $(S^n,*)$-fibration and the identity 
is a morphism of $(S^n,*)$-fibrations. Thus $(S^n,*)$-fibrations form a category.
A morphism $\phi: \alpha = (\xi,s) \to (\xi',s')=\alpha'$ 
of $(S^n,*)$-fibrations induces a map
\[ \Th (\phi):\Th (\alpha) = E/s(B) \longrightarrow E'/s'(B') = \Th (\alpha'). \]
In this way, $\Th (\cdot)$ becomes a functor on the category of $(S^n,*)$-fibrations.
Let $\theta = \theta^F_1$ denote the trivial (product) $(S^1,*)$-fibration over a point.
Then, using fiberwise homotopy smash product $\wedge^h$,
$\gamma^F_n \wedge^h \theta$ is an $(S^{n+1},*)$-fibration over $\BF_n$, and hence
has a classifying morphism $\phi_n: \gamma^F_n \wedge^h \theta \to \gamma^F_{n+1}$
of $(S^{n+1},*)$-fibrations. This yields in particular maps
$f_n: \BF_n \to \BF_{n+1}$ and we denote the stable classifying space by $\BF$.
In addition to $\BF_n$, the following classifying spaces will be relevant:
\begin{itemize}
\item $\BSO_n$, classifying oriented real vector bundles,
\item $\BSPL_n$, classifying oriented PL $(\real^n, 0)$-bundles
  (and oriented PL microbundles),
\item $\BSTOP_n$, classifying oriented topological $(\real^n, 0)$-bundles
  (and oriented topological microbundles),
\item $\BBSPL_n$, classifying oriented PL closed disc block bundles,
\item $\BG_n$, classifying spherical fibrations with fiber $S^{n-1}$.
\end{itemize}  
The unoriented versions of these spaces will be denoted by omitting the `S'.  
For the theory of block bundles, due to Rourke and Sanderson, 
we ask the reader to consult
\cite{rosaannounce}, \cite{rosablockbundles1}, \cite{rosablockbundles3}, and
\cite{rosatopblockbundles}; the definition of a block bundle will be briefly
reviewed further below.
There is a homotopy commutative diagram
\[
\xymatrix{
\BSO_n \ar[r]^{LR} & \BSPL_n \ar[r]^{\operatorname{forget}} \ar[d] & \BSTOP_n \ar[d] &  \\
  & \BBSPL_n \ar[r] & \BG_n  \ar[r] & \BF_n,
} \]
whose philosophy here is that we can flush Thom space issues down to the level of $\BF_n$.
Thus,
a vector bundle has an underlying microbundle, \cite[p. 55, Example (2)]{milnormicro}.
The leftmost horizontal arrow is due to Lashof and Rothenberg \cite{lashofrothen},
who showed that $\operatorname{O}_n$-vector bundles can be triangulated.
The left vertical arrow is due to Rourke and Sanderson:
A PL microbundle gives rise to a unique equivalence class of PL block bundles,
\cite{rosaannounce}.
A PL block bundle determines a unique spherical fibration with fiber
$S^{n-1}$, \cite[Cor. 5.9, p. 23]{rosablockbundles1}.
(Also cf. Casson \cite{casson}.)
Of course, given an (oriented) topological $(\real^n,0)$-bundle, one can
delete the zero-section to obtain an $S^{n-1}$-fibration, which describes the
composition $\BSTOP_n \to \BG_n$.
Consider $S^0 = \{ -1,+1 \}$ as the trivial $S^0$-bundle $\theta_0$ over 
a point. Given an $S^{n-1}$-fibration $\xi$, there is a canonical
$(S^n,*)$-fibration $\xi^\bullet$ associated to it, namely
$\xi^\bullet := \xi * \theta_0$ (fiberwise unreduced suspension).
Note that the fiberwise unreduced suspension $\xi^\bullet$ can
be given a canonical section, by consistently taking north poles (say).
This describes the map $\BG_n \to \BF_n.$

To fix notation, let $\xi$ be a rank $n$ oriented vector bundle over 
the polyhedron $X=|K|$ of a finite
simplicial complex $K$. 
Then $\xi$ has a classifying map
$\xi: X \longrightarrow \BSO_n.$
(We denote classifying maps and the bundle they classify by the same letter.)
Composing with the Lashof-Rothenberg map $LR$, we get a classifying map
\[ \xi_\PL: X \longrightarrow \BSPL_n, \]
which determines an underlying oriented PL $(\real^n,0)$-bundle (or PL microbundle)
over $X$.
We compose further with the map
$\BSPL_n \longrightarrow \BBSPL_n$ and get a classifying map
\[ \xi_\PLB: X \longrightarrow \BBSPL_n, \]
which determines an underlying oriented PL block bundle $\xi_\PLB$ over $X$.
On the other hand, we may compose $\xi_\PL$ with the forget map
to obtain a classifying map
\[ \xi_\TOP: X \longrightarrow \BSTOP_n, \]
which determines an underlying oriented topological $(\real^n, 0)$-bundle
(or topological microbundle) $\xi_\TOP$ over $X$.
Composing with the map $\BSTOP_n \to \BG_n$, we receive a
classifying map
\[ \xi_{\operatorname{G}}: X \longrightarrow \BG_n, \]
which determines an underlying $S^{n-1}$-fibration $\xi_{\operatorname{G}}$
over $X$, which in turn has an underlying $(S^n,*)$-fibration 
$\xi^\bullet = \xi_{\operatorname{G}}^\bullet$.
\begin{defn} \label{def.thomspaceoftopbundle}
Let $\xi$ be a real vector bundle, or PL/topological $(\real^n,0)$-bundle, or
PL closed disc block bundle, or $S^{n-1}$-fibration. 
Then the \emph{Thom space} $\Th (\xi)$ of $\xi$
is defined to be the Thom space of its underlying $(S^n,*)$-fibration,
\[ \Th (\xi) := \Th (\xi^\bullet). \]
\end{defn}
In particular for an oriented vector bundle $\xi$,
\[ \Th (\xi) = 
 \Th (\xi_\PL) = \Th (\xi_\PLB) = \Th (\xi_\TOP) = \Th (\xi^\bullet). \]

Uniform constructions of Thom spectra can be given via the notion of
Thom spectrum of a map $f$.
Let $X$ be a CW complex and $f:X\to \BF$ a continuous map.
The Thom spaces $\Th (f^*_n \gamma^F_n)$ of the pullbacks 
under $f_n:X_n \to \BF_n$ of the universal $(S^n, *)$-fibrations 
form a spectrum $\Th (f)$, whose structure maps are induced on Thom spaces
by the morphisms $f^*_n \gamma^F_n \oplus \theta \to f^*_{n+1} \gamma^F_{n+1}$.
Here, $f_n$ is the restriction of $f$ to an increasing and
exhaustive CW-filtration $\{ X_n \}$ of $X$ such that $f(X_n)\subset \BF_n$.
The spectrum $\Th (f)$ is called the \emph{Thom spectrum of the map $f$}.
This construction applies to the map $f: \BSTOP \to \BF$, filtered by
$f_n: \BSTOP_n \to \BF_n$, and yields
the Thom spectrum $\MSTOP = \Th (\BSTOP \to \BF)$.
Note that $f^*_n \gamma^F_n$ has classifying map $f_n:\BSTOP_n \to \BF_n$,
but so does the underlying $(S^n,*)$-fibration $(\gamma^\STOP_n)^\bullet$
of the universal oriented topological $(\real^n,0)$-bundle $\gamma^\STOP_n$ over
$\BSTOP_n$. Hence $f^*_n \gamma^F_n$ and $(\gamma^\STOP_n)^\bullet$ are
equivalent $(S^n,*)$-fibrations and so have homotopy equivalent
Thom spaces.
Similarly, we obtain Thom spectra
$\MSPL = \Th (\BSPL \to \BF)$ and 
$\MSO = \Th (\BSO \to \BF)$.
These spectra $\MSO, \MSPL, \MSTOP$ are commutative ring spectra,
Rudyak \cite[Cor. IV.5.22, p. 261]{rudyak}.

Let $\Omega^\STOP_n (-),$
$\Omega^\SPL_n (-),$ and $\Omega^\SO_n (-)$
denote bordism of oriented topological, or PL, or smooth manifolds.
The Pontrjagin-Thom theorem provides natural isomorphisms
\[ \Omega^\STOP_n (X)\cong \MSTOP_n (X),~ \Omega^\SPL_n (X)\cong \MSPL_n (X),~
 \Omega^\SO_n (X)\cong \MSO_n (X). \]
(In the TOP case, this requires Kirby-Siebenmann topological transversality in high
dimensions, and the work of Freedman and Quinn in dimension $4$.)

We shall next construct maps between Thom spectra.
This can be achieved using the following general principle:
Let $X',X$ be CW complexes with CW filtrations $\{ X'_n \},$ $\{ X_n \}$,
respectively.
Let $g:X' \to X$ be a map with $g(X'_n)\subset X_n$.
Let $f:X\to \BF$ be a map as above so that $\Th (f)$ is defined. 
Then composition gives a map $f' =fg: X'\to \BF$
such that the Thom spectrum $\Th (f')$ is defined as well.
The map $g$ induces a map of spectra
\[ \Th (f') \longrightarrow \Th (f). \]
Applying this principle to $g: \BSPL =X' \to X=\BSTOP$, with
$f:X=\BSTOP \to \BF$ as in the above definition of $\MSTOP,$
yields a map of spectra
\[  \phi_F: \MSPL = \Th (f') \longrightarrow \Th (f) = \MSTOP. \]
Similarly, we get $\phi_{LR}: \MSO \to \MSPL$
using the Lashof-Rothenberg map.

We turn to uniform constructions of Thom classes in cobordism theory.
First, say, for topological bundles:
Let $\xi$ be an oriented topological $(\real^n, 0)$-bundle.
Then $\xi$ is classified by a map $t:X\to \BSTOP_n$
and has an underlying $(S^n,*)$-fibration $\xi^\bullet$
with classifying map the composition
\[ X \stackrel{t}{\longrightarrow} \BSTOP_n
       \stackrel{f_n}{\longrightarrow} \BF_n. \]
Let $\zeta^T_n$ be the $(S^n,*)$-fibration such that
$\MSTOP_n = \Th (\zeta^T_n),$
i.e. $\zeta^T_n = f_n^* \gamma^F_n$.
(This is nothing but $(\gamma^\STOP_n)^\bullet$.)
Then 
\[ t^* \zeta^T_n = t^* f^*_n \gamma^F_n = \xi^\bullet, \]
with corresponding morphism $\psi: \xi^\bullet \to \zeta^T_n$ of $(S^n,*)$-fibrations.
This morphism induces on Thom spaces a map
\[ \Th (\psi): \Th (\xi^\bullet) \longrightarrow \Th (\zeta^T_n) = \MSTOP_n. \]
By Definition \ref{def.thomspaceoftopbundle}, $\Th (\xi^\bullet) = \Th (\xi)$.
So we may write $\Th (\psi)$ as
\[ \Th (\psi): \Th (\xi) \longrightarrow \Th (\zeta^T_n) = \MSTOP_n. \]
Suspension and composition with the structure maps of $\MSTOP$ gives a
map of spectra
\[ \Sigma^\infty \Th (\xi) \longrightarrow \Sigma^n \MSTOP. \]
Here $\Sigma^\infty Y$ denotes the suspension spectrum of a space $Y$,
and $\Sigma^n E$ of a spectrum $E$ is the spectrum with
$(\Sigma^n E)_k = E_{n+k}.$
The map of spectra determines a homotopy class 
\[ u_\STOP (\xi) \in [\Sigma^\infty \Th (\xi), \Sigma^n \MSTOP]
  = \RMSTOP^n (\Th (\xi)). \]
This class $u_\STOP (\xi)$ is called the
\emph{Thom class} of $\xi$ in oriented topological cobordism and is
indeed an $\MSTOP$-orientation of $\xi^\bullet$ in the sense of Dold.

We proceed in a similar way to construct the Thom class of a PL bundle:
Let $\xi$ be an oriented PL $(\real^n, 0)$-bundle over a 
compact polyhedron $X$.
Then $\xi$ is classified by a map 
$h:X\to \BSPL_n$. 
Forgetting the PL structure, we have an underlying topological
$(\real^n,0)$-bundle $\xi_\TOP$, classified by the composition
\[ X \stackrel{h}{\longrightarrow} \BSPL_n 
      \stackrel{g_n}{\longrightarrow} \BSTOP_n. \]
This topological bundle in turn
has an underlying $(S^n,*)$-fibration $(\xi_\TOP)^\bullet$
with classifying map the composition
\[ X \stackrel{h}{\longrightarrow} \BSPL_n 
      \stackrel{g_n}{\longrightarrow} \BSTOP_n
       \stackrel{f_n}{\longrightarrow} \BF_n. \]
Of course $\xi$ itself has an underlying $(S^n,*)$-fibration
$\xi^\bullet$ and $\xi^\bullet = (\xi_\TOP)^\bullet$.
Let $\zeta^P_n$ be the $(S^n,*)$-fibration such that
$\MSPL_n = \Th (\zeta^P_n),$
i.e. $\zeta^P_n = (f_n g_n)^* \gamma^F_n$.
(This is nothing but $(\gamma^\SPL_n)^\bullet$.)
Then 
\[ h^* \zeta^P_n = h^* g^*_n f^*_n \gamma^F_n = \xi^\bullet, \]
with corresponding morphism $\phi: \xi^\bullet \to \zeta^P_n$ 
of $(S^n,*)$-fibrations.
This morphism induces on Thom spaces a map
\[ \Th (\phi): \Th (\xi^\bullet) \longrightarrow \Th (\zeta^P_n) = \MSPL_n. \]
By Definition \ref{def.thomspaceoftopbundle}, $\Th (\xi^\bullet) = \Th (\xi)$.
So we may write $\Th (\phi)$ as
\[ \Th (\phi): \Th (\xi) \longrightarrow \Th (\zeta^P_n) = \MSPL_n. \]
We arrive thus at a map of spectra
\[ \Sigma^\infty \Th (\xi) \longrightarrow \Sigma^n \MSPL, \]
which determines a homotopy class 
\[ u_\SPL (\xi) \in [\Sigma^\infty \Th (\xi), \Sigma^n \MSPL]
  = \RMSPL^n (\Th (\xi)). \]
This class $u_\SPL (\xi)$ is called the
\emph{Thom class} of $\xi$ in oriented PL cobordism.
As in the topological case, one verifies that this is an $\MSPL$-orientation of $\xi^\bullet$.
Earlier, we had constructed a map of Thom spectra
$\phi_F: \MSPL \longrightarrow \MSTOP.$
Recall that the underlying topological bundle $\xi_\TOP$ of a 
PL bundle $\xi_\PL$ and $\xi_\PL$ itself have the same Thom space,
\[ \Th (\xi_\PL) = \Th (\xi^\bullet) = \Th (\xi_\TOP). \]
\begin{lemma} \label{lem.usplustop}
Let $\xi_\PL$ be an oriented PL $(\real^n,0)$-bundle.
On cobordism groups, the induced map
\[ \phi_F: \RMSPL^n (\Th (\xi_\PL)) \longrightarrow
              \RMSTOP^n (\Th (\xi_\TOP)) \]
maps the Thom class of $\xi_\PL$ to the Thom class of the
underlying topological $(\real^n,0)$-bundle $\xi_\TOP$,
\[  \phi_F (u_\SPL (\xi_\PL) = u_\STOP (\xi_\TOP). \]              
\end{lemma}
The proof is a standard verification.

The cobordism Thom class of an oriented real vector bundle $\xi$ can be similarly
fit into this picture. 
If $n$ is the rank of $\xi$, then $\xi$ has a \emph{Thom class}
\[ u_\SO (\xi) \in [\Sigma^\infty \Th (\xi), \Sigma^n \MSO]
  = \RMSO^n (\Th (\xi)). \]
in smooth oriented cobordism.
Recall that we had earlier described a map $\phi_{LR}: \MSO \longrightarrow \MSPL$
based on the Lashof-Rothenberg map. The following compatibility result is
again standard (and readily verified).
\begin{lemma} 
Let $\xi$ be a rank $n$ oriented vector bundle over a compact polyhedron $X$.
On cobordism groups, the induced map
\[ \phi_{LR}: \RMSO^n (\Th (\xi)) \longrightarrow
              \RMSPL^n (\Th (\xi_\PL)) \]
maps the Thom class of $\xi$ to the Thom class of the
underlying oriented PL $(\real^n,0)$-bundle $\xi_\PL$,
\[  \phi_{LR} (u_\SO (\xi)) = u_\SPL (\xi_\PL). \]              
\end{lemma}

\subsection{Ranicki's Thom Class in $\syml$-Cohomology}
\label{sec.relbtwthomclasses}

We review Ranicki's definition of a Thom class
for topological $(\real^n,0)$-bundles (or microbundles) in symmetric 
$\mbl$-cohomology.
He constructs a map
\[ \sigma^*: \MSTOP \longrightarrow \syml, \]
see \cite[p. 290]{ranickitotsurgob}.
Let $X$ be the polyhedron of a finite simplicial complex and
$\xi: X\to \BSTOP_n$ a topological $(\real^n,0)$ bundle (or microbundle) over $X$.
Then, following \cite[pp. 290, 291]{ranickitotsurgob}, 
$\xi$ has a canonical $\syml$-cohomology orientation
\[ u_\mbl (\xi) \in {\redsyml}^n (\Th (\xi)), \]
which we shall also refer to as the 
\emph{$\syml$-cohomology Thom class} of $\xi$, defined by
\[ u_\mbl (\xi) := \sigma^* (u_\STOP (\xi)). \]
The morphism of spectra $\syml (\intg)\to \syml (\rat)$ coming from
localization induces a homomorphism
\[ \redsyml^n (\Th (\xi)) \longrightarrow \redsyml (\rat)^n (\Th (\xi)). \]
We denote the image of $u_\mbl (\xi)$ under this map again by
$u_\mbl (\xi) \in \redsyml (\rat)^n (\Th (\xi)).$

\subsection{Geometric Description of the $\MSO$-Thom class}
\label{sec.geomdescrmsothomcl}

Buoncristiano, Rourke and Sanderson \cite{buonrs}
give a geometric description of $\MSPL$-cobordism
and use it to obtain in particular a geometric description of the Thom
class $u_\SPL (\xi),$ which we reviewed homotopy theoretically
in Section \ref{sec.thomclso}.
The geometric cocycles are given by (oriented) mock bundles, whose
definition we recall here. The polyhedron of a ball complex $K$ is denoted
by $|K|$.
\begin{defn}
Let $K$ be a finite ball complex and $q$ an integer (possibly negative).
A \emph{$q$-mock bundle} $\xi^q/K$ with base $K$ and
total space $E(\xi)$ consists of a PL map
$p:E(\xi)\to |K|$ such that, for each $\sigma \in K$,
$p^{-1} (\sigma)$ is a compact PL manifold of dimension
$q + \dim (\sigma),$ with boundary $p^{-1} (\partial \sigma)$.
The preimage $\xi (\sigma) := p^{-1} (\sigma)$ is called the
\emph{block} over $\sigma$.
\end{defn}
The empty set is regarded as a manifold of
any dimension; thus $\xi (\sigma)$ may be empty for some
cells $\sigma \in K$.
Note that if $\sigma^0$ is a $0$-dimensional cell of $K$,
then $\partial \sigma^0 = \varnothing$ and thus
$p^{-1} (\partial \sigma)=\varnothing$.
Hence the blocks over $0$-dimensional cells are \emph{closed} manifolds.
For our purposes, we need \emph{oriented}
mock bundles, which are defined using incidence numbers of cells and blocks:
Suppose that $(M^n,\partial M)$ is an oriented PL manifold
and $(N^{n-1}, \partial N)$ is an oriented PL manifold
with $N \subset \partial M$. Then an incidence number
$\epsilon (N,M)=\pm 1$ is defined by comparing the orientation
of $N$ with that induced on $N$ from $M$ (the induced orientation
of $\partial M$ is defined by taking the inward normal last);
$\epsilon (N,M)=+1$ if these orientation agree and $-1$ if they
disagree. An \emph{oriented cell complex} $K$ is a cell complex
in which each cell is oriented. We then have the incidence
number $\epsilon (\tau, \sigma)$ defined for faces
$\tau^{n-1} < \sigma^n \in K$.
\begin{defn}
An \emph{oriented mock bundle} is a mock bundle $\xi/K$
over an oriented (finite) ball complex $K$ in which every block
is oriented (i.e. is an oriented PL manifold) such that
for each $\tau^{n-1} < \sigma^n \in K$,
$\epsilon (\xi (\tau), \xi (\sigma)) = \epsilon (\tau, \sigma)$.
\end{defn}
The following auxiliary result is an
analog of \cite[Lemma 1.2, p. 21]{buonrs}:
\begin{lemma} \label{lem.totalspaceofmockiswitt}
Let $(K,K_0)$ be a (finite) ball complex pair such that
$|K|$ is an $n$-dimensional (compact) Witt space with (possibly empty) boundary
$\partial |K|=|K_0|$.
Orient $K$ in such a way that the sum of oriented $n$-balls
is a cycle rel boundary. (This is possible since $|K|$,
being a Witt space, is an oriented
pseudomanifold-with-boundary.)
Let $\xi/K$ be an oriented $q$-mock bundle over $K$.
Then the total space $E(\xi)$ is an
$(n+q)$-dimensional (compact) Witt space with boundary 
$p^{-1} (\partial |K|).$
\end{lemma}
\begin{proof}
One merely has to modify the proof of \cite[Lemma 1.2]{buonrs}
for the Witt context, see also the proof of the IP-ad theorem
\cite[Theorem 4.4]{blm}.
First, choose a structuring of $K$ as a
\emph{structured cone complex} in the sense of
McCrory \cite{mccrory} by choosing points $\hat{\sigma}$ in the
interior of $\sigma$ for every cell $\sigma \in K$.
The associated first derived subdivision $\hat{K}$
is a simplicial complex and induces a concept of dual cells
$D(\sigma)$ for cells $\sigma \in K$.
Let $X=|K|$ be the underlying polyhedron of $K$.
Polyhedra have intrinsic PL stratifications, \cite{akin}.
In particular, points in $X$ have intrinsic links $L$ with respect
to this stratification.
The simplicial link of $\hat{\sigma}$ in $\hat{K}$ is a suspension 
of the intrinsic link $L$ at $\hat{\sigma}$.
Then the polyhedron of the dual complex of $\sigma$ 
can be written in terms of the intrinsic link $L$ as
\[ |D(\sigma)| \cong D^{j-k} \times \cone (L), \]
where $D^{j-k}$ denotes a closed disc of dimension $j-k$.

Now let $(X,\partial X)=(|K|,|K_0|)$ be an $n$-dimensional 
PL pseudomanifold-with-boundary, where
$(K,K_0)$ is a ball complex pair with $K$ structured as described above so that
dual blocks of balls are defined.
Assume that $X$ is Witt and $\xi$ is an oriented $q$-mock bundle over $K$.
By the arguments in the proof of \cite[Lemma 4.6]{blm},
the total space $E(\xi)$ is a PL pseudomanifold with collared boundary
$p^{-1} (|K_0|)$. (Those arguments do indeed cover the present case,
since they only require that the base $(|K|,|K_0|)$, as well as the blocks over
cells of that base, be PL pseudomanifolds with
collared boundary --- IP or Witt conditions are irrelevant for this argument.)
 
An orientation of $E(\xi)$ is
induced by the given orientation data as follows:
Triangulate $E(\xi)$ so that each block $\xi (\sigma)$ is a subcomplex.
Let $s$ be a top dimensional simplex of $E(\xi)$ in this triangulation.
Then there is a unique block $\xi (\sigma_s)$ that contains $s$.
This block is an oriented PL manifold with boundary (since $\xi$
is oriented as a mock bundle), and this orientation induces an
orientation of $s$. Note that $\sigma_s \in K$ is $n$-dimensional.
The sum of all $n$-dimensional oriented cells in $K$ is a cycle rel $|K_0|$,
since $|K|$ is a Witt space, and thus in particular oriented.
Then the preservation of incidence numbers between base cells
and blocks implies that
the sum of all $s$ is a cycle rel $p^{-1} (|K_0|)$.
Hence $E(\xi)$ is oriented as a pseudomanifold-with-boundary.

It remains to be shown that $E(\xi) - \partial E(\xi)$ satisfies
the Witt condition.
Let $x\in E(\xi) - \partial E(\xi)$ be a point in the interior of the total space.
There is a unique $\sigma \in K$ for which $x$ is in the
interior of the block $\xi (\sigma)$.
Note that then $p(x)$ lies in the interior of $\sigma$.
Let $d=\dim \sigma$.

By the arguments used to prove \cite[Lemma II.1.2]{buonrs} and 
\cite[Prop. 6.6]{lauresmcclure},
there exists (inductively, using collars)
a compact neighborhood $N$ of $x$ in $E(\xi)$,
a compact neighborhood $V\cong D^{q+d}$ of $x$ in 
the $(q+d)$-dimensional manifold $\xi (\sigma)$, and
a PL homeomorphism
\[ N \cong V \times |D(\sigma)|. \]
Since
\[ N \cong V \times |D(\sigma)| \cong D^{q+d} \times D^{j-k} \times \cone (L)
  \cong D^{q+d+j-k} \times \cone (L), \]
by a PL homeomorphism which sends $x$ to $(0,c)$, where $c\in \cone (L)$
denotes the cone vertex,
we conclude that the 
intrinsic link at $x$ in $E(\xi)$ is the intrinsic
link $L$ of $\sigma$ at $p(x)$ in $|K|$. If this link has even dimension
$2k$, then $IH^{\bar{m}}_{k} (L;\rat)=0$ since $|K|$ is a Witt space.
But then this condition is also satisfied for the intrinsic
link at $x$ in $E(\xi)$. Hence $E(\xi)-\partial E(\xi)$ is Witt.
\end{proof}
If $|K|$ is a compact Witt space with boundary $\partial |K| = |K_0|$ for a
subcomplex $K_0 \subset K$, and $\xi$ is an oriented mock bundle
over $K$ which is empty over $K_0$, then by 
Lemma \ref{lem.totalspaceofmockiswitt},
\[ \partial E(\xi) = p^{-1} (\partial |K|)= p^{-1} (|K_0|) = \varnothing, \]
i.e. $E(\xi)$ is a \emph{closed} Witt space.

Let $L\subset K$ be a subcomplex.
Oriented mock bundles $\xi_0$ and $\xi_1$ over $K$,
both empty over $L$,
are \emph{cobordant}, if there is an oriented mock bundle
$\eta$ over $K\times I$, empty over $L\times I$,
such that $\eta|_{K\times 0} \cong \xi_0$,
$\eta|_{K\times 1} \cong \xi_1$.
This is an equivalence relation and we set
\[ \Omega^q_\SPL (K,L) :=
 \{ [\xi^q/K] :~ \xi|_L = \varnothing \}, \]
where $[\xi^q/K]$ denotes the cobordism class of the 
oriented $q$-mock bundle $\xi^q/K$ over $K$.
Then the duality theorem \cite[Thm. II.3.3]{buonrs}
asserts that
$\Omega^{-*}_\SPL (-)$ is Spanier-Whitehead dual to oriented PL bordism
$\Omega^\SPL_* (-) \cong \MSPL_* (-)$; see also
\cite[Remark 3, top of p. 32]{buonrs}.
But so is $\MSPL^* (-)$. Hence Spanier-Whitehead duality 
provides an isomorphism
\begin{equation} \label{equ.identmsplcohbrstq} 
\beta: \Omega^{-q}_\SPL (K,L) \cong\MSPL^q (K,L) 
\end{equation}
for compact $|K|, |L|$, which is natural with respect to
inclusions $(K', L')\subset (K,L)$.
This is the geometric description of oriented
PL cobordism that we will use.
We shall now give an explicit description of the
isomorphism $\beta$ in (\ref{equ.identmsplcohbrstq}).
We write $X=|K|$ and $Y=|L|$ for the associated polyhedra, and assume
them to be compact. Embed $X$ into some sphere $S^N$ so that we have
inclusions $Y \subset X \subset S^N.$
We write $X^c, Y^c$ for the complements of $X,Y$ in the sphere.
We can regard $X^c$ and $Y^c$ also as compact polyhedra by removing
the interior of derived neighborhoods of $X$ and $Y$.
Then, according to \cite[Duality Theorem II.3.3]{buonrs},
there is a natural isomorphism
\[ \phi: \Omega^{-q}_\SPL (X,Y) \stackrel{\cong}{\longrightarrow}
  \Omega^\SPL_{N-q} (Y^c, X^c). \]
The Thom-Pontrjagin construction gives a natural isomorphism
\[ \tau: \Omega^\SPL_{N-q} (Y^c, X^c) \stackrel{\cong}{\longrightarrow}
  \MSPL_{N-q} (Y^c, X^c), \]
and Alexander duality provides an isomorphism
\[ \alpha: \MSPL_{N-q} (Y^c, X^c) \stackrel{\cong}{\longrightarrow}
  \MSPL^q (X,Y), \]
which is natural with respect to inclusions.  
On the technical level, we work with
$\alpha := (-1)^N \gamma_t$, where $\gamma_t$ is
Switzer's Alexander duality map \cite[Thm. 14.11, p. 313]{switzer}.
This choice of sign guarantees that for the $n$-ball,
\[ \alpha: \MSPL_0 (D^{\circ n}) = \MSPL_0 (Y^c, X^c)
   \longrightarrow \MSPL^n (D^n, \partial D^n)= \MSPL^n (X,Y) \]
sends the unit
$1\in \MSPL_0 (\pt) = \pi_0 (\MSPL) = \MSPL^0 (\pt)$
to the element
\[ \sigma^n \in \widetilde{\MSPL}^n (S^n) 
          = \MSPL^n (D^n, \partial D^n), \]
obtained by suspending the unit $n$ times.
Then $\beta$ in (\ref{equ.identmsplcohbrstq}) is the composition
\[ \Omega^{-q}_\SPL (X,Y) \stackrel{\phi}{\longrightarrow}
  \Omega^\SPL_{N-q} (Y^c, X^c) \stackrel{\tau}{\longrightarrow}
  \MSPL_{N-q} (Y^c, X^c) \stackrel{\alpha}{\longrightarrow}
  \MSPL^q (X,Y). \]
Let us describe $\phi$ in more detail, following \cite{buonrs}:
Let $N(X), N(Y)$ be derived neighborhoods of $X,Y$ in $S^N$.
Note that $N(X)$ and $N(Y)$ are manifolds with boundaries
$\partial N(X), \partial N(Y)$.
With $j: (X,Y)\hookrightarrow (NX,NY)$ the inclusion, 
pullback (restriction) of mock bundles defines a map
\[ j^*: \Omega^{-q}_\SPL (NX,NY) \longrightarrow \Omega^{-q}_\SPL (X,Y), \]
which is an isomorphism. Amalgamation defines a map
\[ \operatorname{amal}: \Omega^{-q}_\SPL (NX, NY) \longrightarrow 
  \Omega^\SPL_{n+q} (NX-NY, \partial NX - \partial NY), \]
which works as follows:
Given a mock bundle over $NX$, the amalgamation, i.e. the 
union of all its blocks, i.e. the total space, is a manifold,
since the blocks are manifolds and the base $NX$ is a 
manifold as well (this is \cite[Lemma 1.2, p. 21]{buonrs}).
The projection gives a map of the amalgamation to $NX$.
Furthermore, the boundary of the amalgamation is the
material lying over $\partial NX$. Moreover, if the
mock bundle is empty over $NY$, then the boundary of
the amalgamation will not map to $\partial NY$.
Thus we have a map as claimed.
Finally the inclusion 
\[ j: (NX-NY, \partial NX - \partial NY) \hookrightarrow (Y^c, X^c) \]
induces a map
\[ j_*: \Omega^\SPL_{N-q} (NX-NY, \partial NX - \partial NY) 
   \longrightarrow \Omega^\SPL_{N-q} (Y^c, X^c). \]
Then $\phi$ is the composition
\[ \xymatrix@C=80pt{
\Omega^{-q}_\SPL (X,Y) \ar[rdd]_\phi & 
   \Omega^{-q}_\SPL (NX,NY) \ar[l]_{j^*}^{\cong} \ar[d]^{\operatorname{amal}} \\
& \Omega^\SPL_{N-q} (NX-NY, \partial NX - \partial NY) \ar[d]^{j_*} \\
& \Omega^\SPL_{N-q} (Y^c, X^c).
} \]
The following example illustrates
the behavior of $\phi$ and will be used later.
\begin{example} \label{exple.phixdnysn1}
We consider the $n$-ball $X=D^n$ and its boundary sphere
$Y=\partial D^n$. Take $N=n$ and embed $D^n$ into $S^N = S^n$ as the
upper hemisphere so that $\partial D^n$ is embedded as the equatorial sphere.
Then $NY$ is a closed band containing the equator and $NX$ is the union of
this band with the upper hemisphere. The complement $X^c$ is the open
lower hemisphere and the complement $Y^c$ is the disjoint union of
open upper and lower hemisphere.
Note that $D^n$ may be regarded as the total space of a trivial
block bundle (\cite{rosablockbundles1}) over a point.
A block bundle always has a zero section, which for the trivial block
bundle over a point is the inclusion $i: \{ 0 \} \hookrightarrow D^n,$
where $D^n$ is triangulated so that its center $0$ is a vertex.
(Then $i$ is a simplicial inclusion.)
The BRS-Thom class $u_{BRS} (\xi)$ of a block bundle $\xi$ is explained further below,
in (\ref{equ.defubrs}).
For $\xi = \epsilon^n,$ the trivial $n$-block bundle over a point,
it is given by
\[ u_{BRS} (\epsilon^n) =
   [ \{ 0 \} \stackrel{i}{\hookrightarrow} D^n ]\in      
   \Omega^{-n}_\SPL (D^n, \partial D^n). \]
Here, we interpret the inclusion $\{ 0 \} \hookrightarrow D^n$ as the
projection of a $(-n)$-mock bundle over $D^n$ with block $\{ 0 \}$ over
the cell $D^n$ and empty blocks over all boundary cells of 
the polyhedron $D^n$.
We shall compute the image under $\phi$ of this element $u_{BRS} (\epsilon^n)$.
The center $0$ includes into $NX$, so we have
$\{ 0 \} \hookrightarrow NX.$
Again, we interpret this inclusion as a mock bundle over $NX$,
so it defines an element
$[\{ 0 \} \hookrightarrow NX] \in \Omega^{-n}_\SPL (NX,NY),$
as the blocks over the equatorial band $NY$ are all empty.
Induced mock bundles are given by pulling back under simplicial maps.
If $j$ is the (simplicial) inclusion $j:X\hookrightarrow NX$, then the 
pullback of the mock bundle $\{ 0 \} \hookrightarrow NX$
is given by
$j^* [\{ 0 \} \hookrightarrow NX] = [\{ 0 \} \hookrightarrow D^n].$
To compute the amalgamation of 
the mock bundle $\{ 0 \} \hookrightarrow NX$ over $NX$, we observe that
its total space consists of only one block
(namely $\{ 0 \}$), so there is nothing to amalgamate.
Thus
\[ \operatorname{amal} [\{ 0 \} \hookrightarrow NX]
   = [\{ 0 \} \hookrightarrow NX-NY] \in 
   \Omega^\SPL_0 (NX-NY, \partial NX - \partial NY). \]
(Note that $\{ 0 \} \not\in NY$.)
Now the boundary $\partial NY$ of the band $NY$ consists of two
disjoint circles, one in the upper hemisphere, the other in the lower
hemisphere. The circle in the lower hemisphere is $\partial NX$.
Therefore, $\partial NX - \partial NY = \varnothing$.
In particular,
\[ \Omega^\SPL_0 (NX-NY, \partial NX - \partial NY) =
   \Omega^\SPL_0 (NX-NY). \]
Since $Y^c$ is the disjoint union of two open discs, and
$X^c$ is the lower one of these discs, we have by excision
\[ \Omega^\SPL_0 (Y^c,X^c) = \Omega^\SPL_0 (D^{\circ n}) =
  \Omega^\SPL_0 (\{ 0 \}), \]
where $D^{\circ n}$ is the upper open disc, i.e. the one containing 
the point $0$. Under this identification,
\[ j_* [\{ 0 \} \hookrightarrow NX-NY] = 
   [\{ 0 \} \stackrel{\id}{\longrightarrow} \{ 0 \}] \in 
   \Omega^\SPL_0 (\pt). \]
We have shown that
\[ \phi (u_{BRS} (\epsilon^n)) =
 [ \{ 0 \} \hookrightarrow Y^c] \in \Omega^\SPL_0 (Y^c, X^c).  \]
\end{example}

Let $I$ denote the unit interval.
Recall that
a PL (closed disc) \emph{$q$-block bundle} $\xi^q/K$ consists of a PL total space
$E(\xi)$ and a ball complex $K$ covering a polyhedron $|K|$
such that $|K| \subset E(\xi)$,
for each $n$-ball $\sigma$ in $K$, there is a (closed)
PL $(n+q)$-ball $\beta (\sigma) \subset E(\xi)$ (called the \emph{block} over $\sigma$)
and a PL homeomorphism of pairs
\[ (\beta (\sigma), \sigma) \cong (I^{n+q}, I^n), \]
$E(\sigma)$ is the union of all blocks $\beta (\sigma),$ $\sigma \in K$,
the interiors of the blocks are disjoint, and
if $L$ is the complex covering the polyhedron 
$\sigma_1 \cap \sigma_2$, then $\beta (\sigma_1)\cap \beta (\sigma_2)$
is the union of the blocks over cells in $L$.
So a block bundle need not have a projection from the
total space to the base, but it always has a canonical zero section
$i: K\hookrightarrow E(\xi)$.
The trivial $q$-block bundle has total space
$E(\xi) = |K| \times I^q$ and blocks
$\beta (\sigma) = \sigma \times I^q$ for each $\sigma \in K$.
The cube $I^q$ has boundary $\Sigma^{q-1} = \partial I^q$ and
a block preserving PL homeomorphism 
$\Delta^n \times I^q
    \stackrel{\cong}{\to} \Delta^n \times I^q$
restricts to a block preserving PL homeomorphism    
$\Delta^n \times \Sigma^{q-1} 
    \stackrel{\cong}{\to} \Delta^n \times \Sigma^{q-1}$,
where $\Delta^n$ is the standard $n$-simplex.    
Hence there is a homomorphism
$\TPL_n (I) \rightarrow \TPL_n (\Sigma)$ of semi-simplicial groups
given by restriction, \cite[p. 436]{rosablockbundles3}.
On classifying spaces, this map induces
$\BBPL_n \to \BBPL_n (\Sigma)$.
Thus a closed disc block bundle $\xi$ has a well-defined \emph{sphere block bundle} 
$\dot{\xi}$, see
\cite[\S 5, p. 19f]{rosablockbundles1}, whose total space $\dot{E}$
is a PL subspace $\dot{E} \subset E(\xi)$ of the total space of $\xi$.

Now let $\xi:|K|\to \BBSPL_n$ be an oriented PL closed disc block bundle
of rank $n$ over a finite complex $K$. Then $\xi$ has a Thom class
as follows (cf. \cite[p. 26]{buonrs}):
Let $i: K \to E=E(\xi)$ be the zero section.
Endow $E$ with the ball complex structure given by taking the
blocks $\beta (\sigma)$ of the bundle $\xi$ as balls,
together with the balls of a suitable ball complex structure on the total
space $\dot{E}$ of the sphere block bundle $\dot{\xi}$.
Then $i: K\to E$ is the projection of an oriented 
$(-n)$-mock bundle, 
and thus determines an element
\begin{equation} \label{equ.defubrs} 
u_{BRS} (\xi) :=
   [i] \in \Omega^{-n}_\SPL (E,\dot{E}),  
\end{equation}   
which we shall call the
\emph{BRS-Thom class} of $\xi$.
Note that if $\sigma$ is a cell in $\dot{E}$, then 
$i^{-1} (\sigma) = \sigma \cap |K| = \varnothing$,
so $[i]$ defines indeed a class rel $\dot{E}$.
The BRS-Thom class is natural, \cite[p. 27]{buonrs}.

Let $\xi:|K|\to \BSPL_n$ be an oriented PL $(\real^n,0)$-bundle.
This bundle has a Thom class
\[ u_\SPL (\xi) \in \RMSPL^n (\Th (\xi)), \]
as discussed in Section \ref{sec.thomclso}.
Composing with the map
$\BSPL_n \to \BBSPL_n$, we get a map
$\xi_\PLB: |K|\longrightarrow \BBSPL_n,$
which is the classifying map of the underlying 
oriented PL block bundle $\xi_\PLB$ of $\xi$.
\begin{lemma} \label{lem.usplgoestoubrsfortriv}
For the trivial oriented PL $(\real^n,0)$-bundle $\epsilon^n$ 
over a point, the isomorphism (\ref{equ.identmsplcohbrstq}),
\[ \beta: \Omega^{-n}_\SPL (D^n, \partial D^n) \cong \MSPL^n (D^n, \partial D^n), \]
maps the BRS-Thom class
$u_{BRS} (\epsilon^n_\PLB)$ to the Thom class $u_\SPL (\epsilon^n)$.
\end{lemma}
\begin{proof}
The isomorphism $\beta$ is the composition
\[ \Omega^{-n}_\SPL (D^n,\partial D^n) \stackrel{\phi}{\longrightarrow}
  \Omega^\SPL_0 (Y^c, X^c) \stackrel{\tau}{\longrightarrow}
  \MSPL_{0} (Y^c, X^c) \stackrel{\alpha}{\longrightarrow}
  \MSPL^n (D^n,\partial D^n). \]
Note that the underlying PL block bundle $\epsilon^n_\PLB$ of $\epsilon^n$
is the trivial block bundle over a point.
Thus, by Example \ref{exple.phixdnysn1},
\[ \phi (u_{BRS} (\epsilon^n_\PLB)) =
 [ \{ 0 \} \hookrightarrow Y^c] \in \Omega^\SPL_0 (Y^c, X^c)  \]
and under the identification
$\Omega^\SPL_0 (Y^c,X^c) = \Omega^\SPL_0 (D^{\circ n}) = \Omega^\SPL_0 (\{ 0 \}),$
we have
\[ \phi (u_{BRS} (\epsilon^n_\PLB)) = 
   [\{ 0 \} \stackrel{\id}{\longrightarrow} \{ 0 \}] \in 
   \Omega^\SPL_0 (\{ 0 \}) = \Omega^\SPL_0 (\pt). \]
The Thom-Pontrjagin construction $\tau$ sends
$[\id_{\{ 0 \}}]$ to the unit $1\in \MSPL_0 (\pt)$.
Finally, the Alexander duality map $\alpha$
sends the unit $1\in \MSPL_0 (\pt)$ to
$\sigma^n \in \MSPL^n (D^n, \partial D^n).$
So
\[ \beta (u_{BRS} (\epsilon^n_\PLB))
 = \alpha \tau \phi (u_{BRS} (\epsilon^n_\PLB))
 = \alpha \tau [\id_{\{ 0 \}}] = \alpha (1) = \sigma^n. \]
Directly from the construction of $u_\SPL$ one sees that
$u_\SPL (\epsilon^n) = \sigma^n$ as well.
\end{proof}

\begin{lemma} \label{lem.usplgoestoubrs}
Let $\xi:|K|\to \BSPL_n$ be an oriented PL $(\real^n,0)$-bundle, $|K|$ compact.
Under the isomorphism $\beta$ in (\ref{equ.identmsplcohbrstq}),
the BRS-Thom class $u_{BRS} (\xi_\PLB)$ 
of the underlying oriented PL block bundle
gets mapped to the Thom class $u_\SPL (\xi)$ .
\end{lemma}
\begin{proof}
We write $X=|K|$ for the compact polyhedron of $K$.
Let $x\in X$ be a point. The bundle $\xi$ has a projection
$p: E\to X$ and we can speak of the fiber $E_x = p^{-1} (x) \cong \real^n$
over $x$. 
Let $E_0 \subset E$ be the complement of the zero section and let
$E_{0x} = E_x \cap E_0 \cong \real^n - \{ 0 \}$.
Let $E'$ denote the total space of the block bundle $\xi_\PLB$,
and $\dot{E}'$ the total space of the sphere block bundle of $\xi_\PLB$.
We may identify
$\MSPL^{n}_\SPL (E', \dot{E}') \cong
\MSPL^{n}_\SPL (E, E_0)$, since 
$E'/ \dot{E}'$ and $\Th (\xi_\PLB) = \Th (\xi^\bullet) = \Th (\xi)$
are naturally homotopy equivalent.
Let $\xi_\PLB|_{\{ x \}}$ denote the restriction of $\xi_\PLB$ to 
$\{ x \}$, where we subdivide $K$ so that $x$ becomes a vertex, if necessary.
Let $E'_x$ denote the total space of $\xi_\PLB|_{\{ x \}}$,
and $\dot{E}'_x$ the total space of the sphere block bundle of $\xi_\PLB|_{\{ x \}}$.
The inclusions
\[ (E'_x, \dot{E}'_x) 
      \hookrightarrow (E', \dot{E}'),~
 (E_x, E_{0x}) \hookrightarrow (E,E_0) \]
will be denoted by $j_x$.
By naturality of the isomorphism $\beta$
with respect to inclusions of pairs, the diagram
\[ \xymatrix{
\Omega^{-n}_\SPL (E', \dot{E}') \ar[r]^\cong_\beta \ar[d]_{j^*_x} &
  \MSPL^n (E', \dot{E}') \ar[d]^{j^*_x}  \ar@{=}[r]^\sim &
  \MSPL^n (E,E_0) \ar[d]^{j^*_x} \\
\Omega^{-n}_\SPL (E'_x, \dot{E}'_x) \ar[r]^\cong_\beta  &
  \MSPL^n (E'_x, \dot{E}'_x) \ar@{=}[r]^\sim &
  \MSPL^n (E_x, E_{x0})
} \]
commutes.
As $X$ is compact, it has finitely many path components $X_1,\ldots, X_m$.
For every $i=1,\ldots, m$, choose a point $x_i \in X_i$.
We shall compute the fiber restrictions of our two classes to these points.
Let $x\in \{ x_1,\ldots, x_m \}$.
Directly from the construction of $u_\SPL$, we have
$j^*_x u_\SPL (\xi) = \sigma^n \in \widetilde{\MSPL}^n (S^n)
\cong \widetilde{\MSPL}^0 (S^0)$.
In particular, $u_\SPL (\xi)$ is an orientation for $\xi$
(and $\xi^\bullet$) in Dold's sense.
For $\beta (u_{BRS} (\xi_\PLB))$ we have, using the above commutative diagram,
the naturality of both the BRS-Thom class
and $u_\SPL$, and Lemma \ref{lem.usplgoestoubrsfortriv},
\[ j^*_x (\beta (u_{BRS} (\xi_\PLB))) =
   \beta (j^*_x u_{BRS} (\xi_\PLB)) = \beta (u_{BRS} (\xi_\PLB|_{\{ x \}}) 
 = u_\SPL (\xi|_{\{ x \}}) = j^*_x u_\SPL (\xi). \]   
This shows that $\beta (u_{BRS} (\xi_\PLB))$ is also an orientation for
$\xi$.
Since $\MSPL$ is a connected spectrum, an orientation $u$ in
$\MSPL^n (E,E_0)$ for $\xi$ is uniquely determined by $j^*_x (u)$,
$x\in \{ x_1,\ldots, x_m \}$ (\cite[14.8, p. 311]{switzer}).
The above calculation shows that the $\MSPL$-orientations
$u_\SPL (\xi)$ and $\beta (u_{BRS} (\xi_\PLB))$ have the same restrictions
under the $j^*_x$ and thus 
$u_\SPL (\xi) = \beta (u_{BRS} (\xi_\PLB))$.
\end{proof}

\subsection{Witt Bordism and Cap Products}
\label{sec.geomcap}

Recall that we had the Lashof-Rothenberg map $\phi_{LR}: \MSO \to \MSPL$.
Let $\MWITT$ be the spectrum representing Witt-bordism $\omwt_* (-),$ 
considered explicitly first in \cite{curran}.
Curran proves in \cite[Thm. 3.6, p. 117]{curran} that
$\MWITT$ is an $\MSO$-module spectrum.
It is even an
$\MSPL$-module spectrum because the product of a Witt space
and an oriented PL manifold is again a Witt space.
(Further remarks on the structure of $\MWITT$ will be
made in Section \ref{ssec.relwittsymlgysin} below.)
Thus there is a cap product
\[ \cap: \MSPL^c (X,A) \otimes \MWITT_n (X,A) \longrightarrow \MWITT_{n-c} (X). \]
By Buoncristiano-Rourke-Sanderson, a geometric description of this cap product
is given as follows:
One uses the isomorphism (\ref{equ.identmsplcohbrstq}) to think
of the cap product as a product
\[ \cap: \Omega_\SPL^{-c} (K,L) \otimes \omwt_n (|K|,|L|) \longrightarrow
  \omwt_{n-c} (|K|) \]
for finite ball complexes $K$ with subcomplex $L\subset K$.
Let us first discuss the absolute case $L=\varnothing$, and then
return to the relative one.
If $f: Z\to |K|$ is a continuous map from an 
$n$-dimensional closed Witt space $Z$ to
$|K|$, and $\xi^q$ is a $q$-mock bundle over $K$ (with $q=-c$), 
then one defines (cf. \cite[p. 29]{buonrs})
\[ [\xi^q/K] \cap [f:Z\to |K|] :=
  [h: E(f^* \xi)\to |K|] \in \omwt_{n-c} (|K|),  \]
where $h$ is the diagonal arrow in the cartesian diagram
\[ \xymatrix{
E(f^* \xi) \ar[r] \ar[d] \ar[rd]^h & E(\xi) \ar[d]^p \\
Z \ar[r]^{f'} & K.
} \]
Here, we subdivide simplicially, homotope $f$ to a simplicial map $f',$
and use the fact (\cite[II.2, p. 23f]{buonrs}) that mock
bundles admit pullbacks under simplicial maps.
By Lemma \ref{lem.totalspaceofmockiswitt},
$E(f^* \xi)$ is a closed Witt space.
For the relative case, we observe that
if $(Z,\partial Z)$ is a compact Witt space with boundary,
$f:(Z,\partial Z)\to (|K|,|L|)$ maps the boundary into $|L|$,
and $\xi|_L = \varnothing$, then
$f^* \xi|_{\partial Z} = \varnothing$
and so $\partial E(f^* \xi) = \varnothing,$
i.e. the Witt space $E(f^* \xi)$ is \emph{closed}.
Hence it defines an absolute bordism class.

\subsection{The Gysin Map on Witt Bordism}
\label{sec.gysinwittbordism}

For a (real) codimension $c$ normally nonsingular inclusion
$g: Y^{n-c} \hookrightarrow X^n$ of closed oriented PL pseudomanifolds,
we define a Gysin map
\[ g^!: \omwt_k (X) \longrightarrow \omwt_{k-c} (Y), \]
and we shall prove that it sends the Witt orientation of $X$ to
the Witt orientation of $Y$, if $X$ and $Y$ are Witt spaces.
This will then be applied in proving
the analogous statement for the $\syml$-homology orientations.

Let $\nu$ be the normal bundle of the embedding $g$.
By definition of normal nonsingularity, $\nu$
is a vector bundle over $Y$, and it is canonically oriented
since $X$ and $Y$ are oriented. 
Thus $\nu$ is classified by a continuous map $\nu: Y\to \BSO_c$.
As explained in Section \ref{sec.thomclso}, $\nu$ determines
an oriented PL $(\real^c,0)$-bundle $\nu_\PL$,
an oriented PL (closed disc) block bundle $\nu_\PLB$, and
an oriented topological $(\real^c,0)$-bundle $\nu_\TOP$.
Let $E = E(\nu_\PLB)$ denote the total space
of the PL block bundle $\nu_\PLB$. Then $E$ is a compact PL pseudomanifold with
boundary $\partial E = \dot{E} = \dot{E}(\nu_\PLB)$.
(This uses that $Y$ is closed.)
The Thom space $\Th (\nu)$ of $\nu$ is homotopy equivalent
to the PL space $\Th' (\nu_\PLB) := E \cup_{\dot{E}} \cone \dot{E}$.
The standard map $j: X \to \Th' (\nu_\PLB)$ is the identity on
an open tubular neighborhood of $Y$ in $X$ and sends points
farther away from $Y$ to the cone point $\infty \in \Th' (\nu_\PLB)$.
As in Ranicki \cite[p. 186]{ranickialtm}, this map extends to a map
$j: X_+ \to \Th' (\nu_\PLB)$ by sending the additional disjoint point to $\infty$.
By Lashof-Rothenberg triangulation, we can and will assume that $j$ is simplicial.
This map induces a homomorphism
\[ j_*: \omwt_k (X) = \omwt_k (X_+,\pt)
\longrightarrow 
  \omwt_k (\Th' (\nu_\PLB),\infty) \cong \omwt_k (E,\dot{E}). \]
Recall from Section \ref{sec.geomcap} that we had a cap product
\[ \cap: \Omega^{-c}_\SPL (E,\dot{E}) \otimes \omwt_k (E,\dot{E}) \longrightarrow
  \omwt_{k-c} (E), \]
which we had described geometrically.
Capping with the BRS-Thom class
$u_{BRS} (\nu_\PLB) \in \Omega^{-c}_\SPL (E,\dot{E}),$
we get a map
\[ u_{BRS} (\nu_\PLB) \cap -: \omwt_k (E,\dot{E}) 
  \longrightarrow \omwt_{k-c} (E). \]
Composing this with the above map $j_*$, we get
the \emph{Witt-bordism Gysin map}
\[ g^! := (u_{BRS} (\nu_\PLB) \cap -)\circ j_*:
  \omwt_k (X) \longrightarrow \omwt_k (E,\dot{E})
  \longrightarrow \omwt_{k-c} (E)\cong \omwt_{k-c} (Y), \]
where the last isomorphism is the inverse of the isomorphism
induced by the zero section.  
A closed $n$-dimensional Witt space $X^n$ has a canonical 
\emph{Witt-bordism fundamental class}
\[ [X]_{\Witt} := [\id:X\to X] \in \omwt_n (X). \]

\begin{thm} \label{thm.gysinpreserveswittfundclass}
The Witt-bordism Gysin map $g^!$ of a
(real) codimension $c$ normally nonsingular inclusion
$g: Y^{n-c} \hookrightarrow X^n$ of closed (oriented) Witt spaces,
sends the Witt-bordism fundamental class of $X$ to the 
Witt-bordism fundamental class of $Y$:
\[ g^! [X]_{\Witt} = [Y]_{\Witt}. \]
\end{thm}
\begin{proof}
The image of $[\id: X\to X]$ under $j_*$ is
$[j: X \to \Th' (\nu_\PLB)] \in \omwt_n (\Th' (\nu_\PLB),\infty) 
\cong \omwt_n (E,\dot{E})$.
The BRS-Thom class of $\nu_\PLB$ is given by the class $[i:Y\to E]$
of the zero-section. Under the identification
$\Omega^{-c}_\SPL (E,\dot{E}) \cong \Omega^{-c}_\SPL (\Th' (\nu_\PLB),\infty),$
it is represented by composing $i$ with the inclusion $E \to \Th' (\nu_\PLB)$.
We call the resulting map again $i: Y \to \Th' (\nu_\PLB)$; it is a $(-c)$-mock
bundle projection, where $\Th' (\nu_\PLB)$ is equipped with a ball complex
structure which contains $\infty$ as a zero dimensional ball.
Since $Y$ does not touch $\infty$, this mock bundle is empty over the ball
$\infty$.
The cap product
\[  [i:Y\to \Th' (\nu_\PLB)] \cap [j:X \to \Th' (\nu_\PLB)], \]
is given by $[h]$, where $h$ is the diagonal arrow
in the cartesian diagram
\[ \xymatrix{
E(j^* (i)) \ar[r] \ar[d] \ar[rd]^h & Y \ar[d]^i \\
X \ar[r]_j & \Th' (\nu_\PLB).
} \]
The pullback $E(j^* (i))$ is just $Y$ and the above 
diagram is
\[ \xymatrix{
Y \ar[r]^{\id} \ar[d]_g \ar[rd]^h & Y \ar[d]^i \\
X \ar[r]_j & \Th' (\nu_\PLB).
} \]
(Recall that $j$ is the identity in a tubular neighborhood
of $Y$; the points of $X$ that are mapped under $j$ to the zero section
are precisely the points of $Y$.)
So
\[ [i:Y\to \Th' (\nu_\PLB)] \cap [j:X \to \Th' (\nu_\PLB)] = [h] = [i]. \]
Now under the isomorphism 
\[ i_*: \omwt_{n-c} (Y) \stackrel{\cong}{\longrightarrow} 
     \omwt_{n-c} (E), \]
the Witt-bordism fundamental class $[\id:Y\to Y]$ is sent to $[i]$.
\end{proof}

\subsection{The Gysin Map on $\syml$-Homology}

We continue in the context of Section \ref{sec.gysinwittbordism}.
Thus $g: Y^{n-c} \hookrightarrow X^n$ is a normally nonsingular
inclusion of closed Witt spaces with
normal vector bundle $\nu$.
The canonical map $j: X_+ \to \Th (\nu)$ induces a homomorphism
\[ j_*: \syml (\rat)_* (X) \longrightarrow 
  \redsyml (\rat)_* (\Th (\nu)). \]
As discussed in Section \ref{sec.relbtwthomclasses},
the oriented topological $(\real^c,0)$-bundle
$\nu_\TOP$ determined by $\nu$ has an $\syml$-cohomology Thom class
\[ u_\mbl (\nu_\TOP) \in \redsyml (\rat)^c (\Th (\nu_\TOP)), \]
defined by
\[ u_\mbl (\nu_\TOP) = \epsilon_\rat \sigma^* (u_\STOP (\nu_\TOP)), \]
where $\epsilon_\rat$ is induced by
$\syml (\intg)\to \syml (\rat)$.
Capping with this class, we receive a map
\[ u_\mbl (\nu_\TOP) \cap -: \redsyml (\rat)_k (\Th (\nu)) 
  \longrightarrow \syml (\rat)_{k-c} (Y). \]
Composing this with the above map $j_*$ on $\syml (\rat)$-homology, we get
the \emph{$\syml$-homology Gysin map}
\[ g^! := (u_\mbl (\nu_\TOP) \cap -)\circ j_*:
  \syml (\rat)_k (X) \longrightarrow \syml (\rat)_k (\Th (\nu))
  \longrightarrow \syml (\rat)_{k-c} (Y). \]
(Of course this map can be defined over $\syml$, but we only need it
over $\syml (\rat)$.)

\subsection{Relation between Witt and $\syml$-Gysin Maps}
\label{ssec.relwittsymlgysin}

The spectra $\syml (\intg)$ and $\syml (\rat)$ are ring spectra.
The product of
two $\rat$-Witt spaces is again a $\rat$-Witt space. This implies essentially that
$\MWITT$ is a ring spectrum; for more details see \cite{blm}.
There, we constructed a map
\[ \tau: \MWITT \longrightarrow \syml (\rat). \]
(Actually, we even constructed an integral map
$\operatorname{MIP} \to \syml,$ where $\operatorname{MIP}$ represents
bordism of integral intersection homology Poincar\'e spaces studied in
\cite{gorsie} and \cite{pardon}, 
but everything works in the same manner for Witt, if one
uses the $\syml$-spectrum with rational coefficients.)
This map is multiplicative, i.e. a ring map, as shown
in \cite[Section 12]{blm}. 
Using this map $\tau$, a closed Witt space $X^n$ has a canonical 
\emph{$\syml (\rat)$-homology fundamental class}
\[ [X]_\mbl \in \syml (\rat)_n (X), \]   
which is by definition the image of $[X]_{\Witt}$ under
the map
\[ \tau_*: \omwt_n (X) = \MWITT_n (X) \longrightarrow
   \syml (\rat)_n (X), \]
i.e.
\[ [X]_\mbl := \tau_* ([X]_{\Witt}). \]   
Every oriented PL manifold is a Witt space. Hence there is a map
\[ \phi_W: \MSPL \longrightarrow \MWITT, \]
which, using the methods of ad-theories and Quinn spectra
employed in \cite{blm}, can be constructed to be multiplicative, i.e. a map of ring spectra.
By the construction of $\tau$ in \cite{blm}, the diagram
\begin{equation} \label{equ.esigmaphifistauphiw}
\xymatrix{
& \MSTOP \ar[r]^{\sigma^*} & \syml (\intg) \ar[dd]^{\epsilon_\rat} \\
\MSPL \ar[ru]^{\phi_F} \ar[rd]_{\phi_W} & & \\
& \MWITT \ar[r]_\tau & \syml (\rat)
} \end{equation}
homotopy commutes.
In the proof of Theorem \ref{thm.gysinpreserveslhomfundclass} below,
we shall use the following standard fact:
\begin{lemma} \label{lem.caponemodmor}
If $E$ is a ring spectrum, $F,F'$ module spectra over $E$ and
$\phi: F\to F'$ an $E$-module morphism, then the diagram
\[ \xymatrix{
E^c (X,A) \otimes F_n (X,A) \ar[r]^>>>>>\cap  \ar[d]_{\id \otimes \phi_*}
  & F_{n-c} (X) \ar[d]^{\phi_*} \\
E^c (X,A) \otimes F'_n (X,A) \ar[r]^>>>>>\cap & F'_{n-c} (X)
} \]
commutes: if $u\in E^c (X,A),$ and $a\in F_n (X,A)$, then
\[ \phi_* (u\cap a) = u\cap \phi_* (a). \]
\end{lemma} 

\begin{thm} \label{thm.gysinpreserveslhomfundclass}
The $\syml$-homology Gysin map $g^!$ of a
(real) codimension $c$ normally nonsingular inclusion
$g: Y^{n-c} \hookrightarrow X^n$ of closed (oriented) Witt spaces
sends the $\syml (\rat)$-homology fundamental class of $X$ to the 
$\syml (\rat)$-homology fundamental class of $Y$:
\[ g^! [X]_\mbl = [Y]_\mbl. \]
\end{thm}
\begin{proof}
Let $\nu$ be the topological normal vector bundle of $g$.
The diagram
\[ \xymatrix{
\omwt_n (X) \ar[r]^{\tau_*} \ar[d]_{j_*} & \syml (\rat)_n (X) \ar[d]^{j_*} \\
\omwt_n (\Th (\nu),\infty) \ar[r]^{\tau_*}  
& \syml (\rat)_n (\Th (\nu),\infty) 
} \]
commutes, since $\tau_*$ is a natural transformation of homology theories.
We shall prove next that the diagram
\begin{equation} \label{equ.tauandu} 
\xymatrix{
\omwt_n (\Th (\nu),\infty) \ar[r]^{\tau_*} \ar[d]_{u_{BRS} (\nu_\PLB) \cap -} 
& \syml (\rat)_n (\Th (\nu),\infty) \ar[d]^{u_\mbl (\nu_\TOP) \cap -}  \\
\omwt_{n-c} (Y) \ar[r]^{\tau_*} & \syml (\rat)_{n-c} (Y), 
} 
\end{equation}
commutes as well.
Let $a\in \omwt_n (\Th (\nu),\infty)$ be an element.
According to the definition of the $\syml$-cohomology Thom class, we have
\[ u_\mbl (\nu_\TOP) \cap \tau_* (a) 
  = \epsilon_\rat \sigma^* (u_\STOP (\nu_\TOP)) \cap \tau_* (a). \]
By Lemma \ref{lem.usplustop},
\[ \epsilon_\rat \sigma^* (u_\STOP (\nu_\TOP)) \cap \tau_* (a)
 = \epsilon_\rat \sigma^* \phi_F (u_\SPL (\nu_\PL)) \cap \tau_* (a). \]
Using diagram (\ref{equ.esigmaphifistauphiw}),
\[ \epsilon_\rat \sigma^* \phi_F (u_\SPL (\nu_\PL)) \cap \tau_* (a)
  = \tau \phi_W (u_\SPL (\nu_\PL)) \cap \tau_* (a). \]
In the above formulae, the symbol $\cap$ denotes the cap-product
on $\syml$-(co)homology.
Using the ring map $\phi_W: \MSPL \to \MWITT$, the spectrum $\MWITT$
becomes an $\MSPL$-module with action map
\[ \MSPL \wedge \MWITT \longrightarrow \MWITT \]
given by the composition
\[ \MSPL \wedge \MWITT \stackrel{\phi_W \wedge \id}{\longrightarrow}
   \MWITT \wedge \MWITT \longrightarrow \MWITT. \]
Using the ring map $\tau \phi_W: \MSPL \to \syml (\rat)$, 
the spectrum $\syml (\rat)$
becomes an $\MSPL$-module with action map
\[ \MSPL \wedge \syml (\rat) \longrightarrow \syml (\rat) \]
given by the composition
\[ \MSPL \wedge \syml (\rat) \stackrel{(\tau \phi_W) \wedge \id}{\longrightarrow}
   \syml (\rat) \wedge \syml (\rat) \longrightarrow \syml (\rat). \]
Hence
\[ \tau \phi_W (u_\SPL (\nu_\PL)) \cap \tau_* (a)
   =  u_\SPL (\nu_\PL)) \cap \tau_* (a), \]
where $\cap$ on the left hand side denotes the $\syml$-internal cap-product,
whereas $\cap$ on the right hand side denotes the cap-product 
coming from the above structure of $\syml (\rat)$ as an $\MSPL$-module.
The homotopy commutative diagram
\[ \xymatrix{
\MSPL \wedge \MWITT \ar[r]^{\id \wedge \tau} \ar[d]_{\phi_W \wedge \id} & 
   \MSPL \wedge \syml (\rat) \ar[d]^{(\tau \phi_W)\wedge \id} \\
\MWITT \wedge \MWITT \ar[r]^{\tau \wedge \tau} \ar[d] 
   & \syml (\rat) \wedge \syml (\rat) \ar[d] \\
\MWITT \ar[r]^\tau & \syml (\rat)
} \]
shows that $\tau: \MWITT \to \syml (\rat)$ is an
$\MSPL$-module morphism.
Thus by Lemma \ref{lem.caponemodmor},
\[ \xymatrix{
\MSPL^c (\Th (\nu),\infty) \otimes \MWITT_n (\Th (\nu),\infty) 
    \ar[r]^>>>>>>\cap  \ar[d]_{\id \otimes \tau_*}
  & \MWITT_{n-c} (Y) \ar[d]^{\tau_*} \\
\MSPL^c (\Th (\nu),\infty) \otimes \syml (\rat)_n (\Th (\nu),\infty) 
  \ar[r]^>>>>>>>\cap & \syml (\rat)_{n-c} (Y)
} \]
commutes, so that
\[ u_\SPL (\nu_\PL) \cap \tau_* (a)
    = \tau_* (u_\SPL (\nu_\PL) \cap a). \]
By Lemma \ref{lem.usplgoestoubrs},
the canonical isomorphism (\ref{equ.identmsplcohbrstq})
identifies the Thom class $u_\SPL (\nu_\PL)$ with the BRS-Thom class
$u_{BRS} (\nu_\PLB)$. Therefore,
\[ \tau_* (u_\SPL (\nu_\PL) \cap a) 
    = \tau_* (u_{BRS} (\nu_\PLB) \cap a). \]
Altogether then,
\[ u_\mbl (\nu_\TOP) \cap \tau_* (a) 
 = \tau_* (u_{BRS} (\nu_\PLB) \cap a), \]
which shows that the diagram (\ref{equ.tauandu}) commutes as claimed.
We have shown that the diagram
\[ \xymatrix{
\omwt_n (X) \ar[r]^{\tau_*} \ar[d]_{j_*} & \syml (\rat)_n (X) \ar[d]^{j_*} \\
\omwt_n (\Th (\nu),\infty) \ar[r]^{\tau_*} \ar[d]_{u_{BRS} (\nu_\PLB) \cap -} 
& \syml (\rat)_n (\Th (\nu),\infty) \ar[d]^{u_\mbl (\nu_\TOP) \cap -}  \\
\omwt_{n-c} (Y) \ar[r]^{\tau_*} & \syml (\rat)_{n-c} (Y), 
} \]
commutes.
Thus the diagram of Gysin maps
\[ \xymatrix{
\omwt_n (X) \ar[r]^{\tau_*} \ar[d]_{g^!} & \syml (\rat)_n (X) \ar[d]^{g^!} \\
\omwt_{n-c} (Y) \ar[r]^{\tau_*} & \syml (\rat)_{n-c} (Y)
} \]
commutes.
Using Theorem \ref{thm.gysinpreserveswittfundclass}, it follows that
\[ g^! [X]_\mbl = g^! \tau_* [X]_{\Witt} =
  \tau_* g^! [X]_{\Witt} = \tau_* [Y]_{\Witt} = [Y]_\mbl. \]
\end{proof}

\begin{thm}  \label{thm.lclassgysin}
Let $g: Y \hookrightarrow X$ be a normally nonsingular inclusion of closed
oriented even-dimensional PL Witt pseudomanifolds. Let $\nu$ be the 
topological normal bundle of $g$.
Then
\[ g^! L_* (X) = L^* (\nu) \cap L_* (Y). \]
\end{thm}
\begin{proof}
By Theorem \ref{thm.gysinpreserveslhomfundclass},
the $\syml$-homology Gysin map $g^!$ of $g$
sends the $\syml (\rat)$-homology fundamental class of $X$ to the 
$\syml (\rat)$-homology fundamental class of $Y$:
$g^! [X]_\mbl = [Y]_\mbl.$
It remains to analyze what this equation means after we tensor with
$\rat$. By \cite[Lemma 11.1]{blm},
\[ [X]_\mbl \otimes \rat = L_* (X),~
   [Y]_\mbl \otimes \rat = L_* (Y). \]
Furthermore, according to \cite[Remark 16.2, p. 176]{ranickialtm},
$u_\mbl (\nu) \otimes \rat = L^* (\nu)^{-1} \cup u_\rat,$
where $u_\rat \in \redh^c (\Th (\nu);\rat)$ is the Thom class of $\nu$
in ordinary rational cohomology.   
(Note that Ranicki omits cupping with $u_\rat$ in his notation.)
Thus
\begin{align*}
L_* (Y)
&= [Y]_\mbl \otimes \rat 
    = (g^! [X]_\mbl) \otimes \rat 
    = (j_* [X]_\mbl \cap u_\mbl (\nu)) \otimes \rat \\
&= j_* ([X]_\mbl \otimes \rat) \cap (u_\mbl (\nu) \otimes \rat) 
= j_* L_* (X) \cap (L^* (\nu)^{-1} \cup u_\rat) \\
&= j_* L_* (X) \cap (u_\rat \cup L^* (\nu)^{-1}) 
= (j_* L_* (X) \cap u_\rat) \cap L^* (\nu)^{-1} \\
&= (g^! L_* (X)) \cap L^* (\nu)^{-1}.
\end{align*}
(Note that all involved classes lie in even degrees and hence no signs
come in.)
\end{proof}

\begin{example}
For the top $L$-class, Theorem \ref{thm.lclassgysin}
implies ($n=\dim X,$ $m=\dim Y$)
\begin{align*}
g^! [X]
&= g^! L_n (X) = (L^* (\nu) \cap L_* (Y))_m \\
&= ((1 + L^1 (\nu) + \cdots)\cap (L_m (Y) + L_{m-4} (Y) + \cdots))_m \\
&= ((1 + L^1 (\nu) + \cdots)\cap ([Y] + L_{m-4} (Y) + \cdots))_m \\
&= 1\cap [Y] = [Y],
\end{align*}
i.e. Gysin maps fundamental classes to fundamental classes.
\end{example}

\section{Rigidity Theorems}
\label{sec.rigiditytheorems}

As an application of the $L$-class Gysin Theorem \ref{thm.lclassgysin}, we 
prove results that 
exhibit rigid algebraic properties
of topologically similar projective varieties.
In their general form, these theorems make no restrictions
on the nature of the singularities to which they are applicable. 
The idea is to exploit, in addition to the Gysin formulae, the topological
invariance of the Goresky-MacPherson $L$-class, \cite{csw}, \cite[12.3.]{weinberger}.

The naturality of Thom classes, together with \cite[Part I, Thm. 1.11]{gmsmt}, implies:
\begin{lemma} \label{lem.homgysincommsqu}
Let $M$ be an oriented smooth manifold and $g:N\hookrightarrow M$ the inclusion
of an oriented smooth submanifold $N$ with 
(compatibly oriented) normal bundle $\nu_N$. Let $f: X \hookrightarrow M$
be the inclusion of a closed Whitney stratified subset, which is
assumed to be an oriented pseudomanifold.
If each stratum of $X$ is transverse to $N$, then
$Y = X \cap N$ is an oriented pseudomanifold Whitney stratified by its 
intersection with the strata of $X$,
the inclusion $g_Y:Y \hookrightarrow X$ is normally nonsingular
with oriented normal bundle isomorphic to the restriction of $\nu_N$, and
the diagram
  \[ \xymatrix{
  H_{k-c} (Y) \ar[r]^{f_{Y*}} & H_{k-c} (N) \\
  H_k (X) \ar[u]^{g^!_Y} \ar[r]_{f_*} & H_k (M) \ar[u]_{g^!}
  } \]
 commutes, where $f_Y: Y \hookrightarrow N$ is the inclusion and
 $c$ is the (real) codimension of $N$ in $M$. 
\end{lemma}

\begin{thm} \label{thm.rigid}
Let $X, X' \subset \pr^N$ be complex, purely $n$-dimensional, 
arbitrarily singular, closed algebraic 
varieties in projective space,
both equipped with algebraic Whitney stratifications, $n\geq 2$.
Let $P, P' \subset \pr^N$ be $(N+2-n)$-dimensional planes transverse to the
Whitney strata of $X,X',$ respectively.
If $h: X\to X'$ is any orientation preserving topological homeomorphism 
(complex topology) whose composition with the
embedding $X' \subset \pr^N$ agrees homologically with the embedding $X\subset \pr^N$, then
the signatures of the planar sections $X\cap P$ and $X' \cap P'$ agree,
\[ \sigma (X\cap P) = \sigma (X' \cap P'). \]
\end{thm}
\begin{proof}
Since $X,X'$ are pure-dimensional, they are oriented pseudomanifolds.
Call the planar sections $S=X\cap P$ and $S' = X'\cap P'$.
By transversality to the Whitney strata, the inclusions
$g: S\hookrightarrow X$ and $g':S' \hookrightarrow X'$ are both normally nonsingular
of real codimension $c=2(n-2)$,
with topological normal bundles $\nu, \nu'$ given by restricting the 
normal bundles $\nu_P, \nu_{P'}$ of $P,P'$ to $S,S'$,
by Lemma \ref{lem.homgysincommsqu}.
Note that $S,S'$ are complex surfaces, generally singular of course, and they are
pseudomanifolds.
Let $\ell: P \hookrightarrow \pr^N$ denote the linear inclusion.
By Lemma \ref{lem.homgysincommsqu}, there is a commutative diagram
\[ \xymatrix{
  H_0 (S) \ar[r]^{f_{S*}} & H_0 (P) \\
  H_c (X) \ar[u]^{g^!} \ar[r]_{f_*} & H_c (\pr^N), \ar[u]_{\ell^!}
  } \]
where $f_*$ is induced by the inclusion $f:X\hookrightarrow \pr^N$.
The group $H_c (\pr^N)$ is generated by the fundamental class
$[\pr^{n-2}]$, which
under Gysin maps to $\ell^! [\pr^{n-2}] = [\pt] \in H_0 (P)$.
For any topological space $Z$, let $\epsilon_*: H_0 (Z)\to \rat$
denote the augmentation homomorphism.
We consider the $L$-class $L_c (X) \in H_c (X;\rat)$.
By topological invariance of the Goresky-MacPherson $L$-class
\cite{csw}, \cite[12.3.]{weinberger},
$h_* L_c (X) = L_c (X')$. Using the diagram,
\[ f_* L_c (X) = (\epsilon_* g^! L_c (X)) [\pr^{n-2}]. \]
Since on homology
$f_* = f'_* h_*$, where $f'$ is the embedding $f':X'\hookrightarrow \pr^N,$ we have
$f_* L_c (X) = f'_* L_c (X'),$ i.e. 
$(\epsilon_* g^! L_c (X)) [\pr^{n-2}] = (\epsilon_* g'^! L_c (X')) [\pr^{n-2}].$
By our $L$-class Gysin Theorem \ref{thm.lclassgysin},
\[
\epsilon_* g^! L_c (X) 
 = \epsilon_* (L^* (\nu) \cap L_* (S)) 
 = \sigma (S) + \langle L^1 (\nu), [S] \rangle
\]
(similarly for $S'$), from which we deduce that
$\sigma (S) + \langle L^1 (\nu), [S] \rangle =
   \sigma (S') + \langle L^1 (\nu'), [S'] \rangle.$
By naturality of the cohomological $L$-class,
$\langle L^1 (\nu), [S] \rangle = \langle L^1 (\nu_P), f_{S*} [S] \rangle.$
Let us write $d$ for the degree of $X$ in $\pr^N$,
let $[\pr^2] \in H_4 (\pr^N)=\intg$ be the preferred generator, and let
\[ p_2 = \ell^{-1}_* [\pr^2] \in H_4 (P),~
   p'_2 = \ell'^{-1}_* [\pr^2] \in H_4 (P') \]
be the corresponding generators in the homology of $P,P'$.
(Note that $N\geq 2$ and
$\ell_*: H_4 (P)\to H_4 (\pr^N),$
$\ell'_*: H_4 (P')\to H_4 (\pr^N)$
are isomorphisms.)
Using transversality, $d$ is also the degree of $S$ in $P$ and
thus $f_{S*} [S] = dp_2 \in H_4 (P).$
The first $L$-class of the normal bundle $\nu_P$ of $P$ does not
depend on the choice of linear embedding $P\hookrightarrow \pr^N$:
There exists an automorphism $T:\pr^N \to \pr^N$ which restricts to
$t:= T|:P\to P'$.
By naturality of the $L$-class,
$L^1 (\nu_P) = t^* L^1 (\nu_{P'}).$
The induced isomorphism $T_*$ is the
identity on homology and $t_* (p_2) = p'_2.$
Therefore, 
\[
\langle L^1 (\nu_P), f_{S*} [S] \rangle
= \langle t^* L^1 (\nu_{P'}), dp_2 \rangle 
= \langle L^1 (\nu_{P'}), d p'_2) \rangle 
   = \langle L^1 (\nu_{P'}), f_{S'*} [S'] \rangle,
\]
which implies
$\langle L^1 (\nu), [S] \rangle =
   \langle L^1 (\nu'), [S'] \rangle,$
and hence
$\sigma (S)=\sigma (S')$.
\end{proof}

\begin{example}
We illustrate the assumptions on the homeomorphism $h$ in the Rigidity Theorem \ref{thm.rigid}
by a simple curve example.
Let $C\subset \pr^2$ be the cuspidal plane cubic curve
$y^2 z - x^3 =0$, $\deg C = 3$.
The singular set of $C$ consists of the point
$(0:0:1)$, which is an $A_2$ singularity.
The rational parametrization $h: \pr^1 \to C,$
$h (u:v) = (u^2 v : u^3 : v^3),$
is a homeomorphism in the complex topology.
It is not an isomorphism of varieties, since $\pr^1$ is nonsingular
and $C$ is singular. Considering $\pr^2 = \{ (x:y:z:0) \}$ as a hyperplane
of $\pr^3$, we may consider $C$ as a curve in $\pr^3$ with inclusion $j$.
The degree of $C$ in $\pr^3$ is still $3$.
Now consider the Veronese embedding
$\nu: \pr^1 \hookrightarrow \pr^3,$
$\nu (u:v) = (u^3 : u^2 v : u v^2 : v^3),$
whose image $\nu (\pr^1)$ is the rational normal curve in $\pr^3$
(the twisted cubic).
The degree of $\nu (\pr^1)$ in $\pr^3$ is $3$.  
The diagram
\[ \xymatrix@R=5pt{
\pr^1 \ar@{^{(}->}[rd]^\nu \ar[dd]_h & \\
& \pr^3 \\
C~ \ar@{^{(}->}[ur]_j 
} \]
does not commute, but induces a commutative diagram on homology.
In fact, in this example, the diagram does commute up to homotopy,
since by the Hurewicz theorem, the homotopy invariants 
$\pi_2 (\pr^3) \cong H_2 (\pr^3)= \intg$ are already completely given
by the degree.
\end{example}

The following proposition, building on the methods used to prove the
Rigidity Theorem, makes no explicit assumption
on the singular set of $X'$. Its dimension may well
increase under the topological homeomorphism $h$.
\begin{prop} \label{cor.rigid}
Let $X, X' \subset \pr^{n+2}$ be complex purely $n$-dimensional, 
possibly singular, closed algebraic varieties in projective space, $n\geq 2$.
Suppose that the singular set of $X$ has codimension at least $3$.
Let $h: X\to X'$ be any orientation preserving topological homeomorphism 
(complex topology) whose composition with the
embedding $X' \subset \pr^{n+2}$ agrees homologically with the embedding $X\subset \pr^{n+2}$.
Then, depending on the degree $d$ of $X$, the Goresky-MacPherson $L$-class 
$L_{2(n-2)} (X')$ pushes forward to
\[ f'_* L_{2(n-2)} (X') 
  \begin{cases}
    = \frac{1}{3}(n+1)[\pr^{n-2}],& \text{ if } d=1, \\
    = \frac{2}{3}(n-2)[\pr^{n-2}],& \text{ if } d=2, \\
    \in \{ (n-7)[\pr^{n-2}], (n-2)[\pr^{n-2}] \},& \text{ if } d=3, \\
    \in \left( \{ -16, -4, 1 \} + \frac{4(n-2)}{3} \right) [\pr^{n-2}],& \text{ if } d=4.
   \end{cases}  \]
\end{prop}
\begin{proof}
We use notation as in the proof of the Rigidity Theorem, $c=2(n-2)$.
Endow $X$ with an algebraic Whitney stratification and
let $X^\circ = X - \Sing_X$ denote the regular part of $X$.
Let $P_0 \subset \pr^{n+2}$ be a plane of dimension $N+2-n = 4$.
Using Kleiman-Bertini type theorems, one finds a transformation
$g\in \operatorname{PGL}_{n+3} (\cplx)$
such that with $P=gP_0,$
$P\cap X^\circ$ is nonsingular and $P$ is transverse to the
algebraic Whitney stratification of $X$.
Since the singular set $\Sigma =\Sing_X$ of $X$ 
has dimension at most $n-3$, transversality implies that
$\Sigma$ does not intersect $P$, whence $P \cap X^\circ = P\cap X=S$.
Thus $S$ is a nonsingular algebraic surface in $P=\pr^4$.  

Let $g\in H^2 (\pr^{n+2})\cong \intg$ denote the generator Poincar\'e dual
to the hyperplane class$[\pr^{n+1}]$.
The restriction $b=\ell^* (g) \in H^2 (P^4)$ is the generator Poincar\'e dual
to the hyperplane class $[\pr^3] \in H_6 (P^4)$.
Using the Whitney product formula,
\[
L^1 (\nu_P) 
 = \ell^* L^1 (T\pr^{n+2}) - L^1 (TP) 
 = \frac{1}{3} (\ell^* (n+3) g^2 - 5 \ell^* (g)^2) 
 = \frac{n-2}{3} b^2.
\]
Since $f_{S*} [S] = d[\pr^2],$ we have
$\langle L^1 (\nu_P), f_{S*} [S] \rangle = d(n-2)/3.$
  
While every smooth projective surface can be embedded into $\pr^5$,
embeddings into $\pr^4$ impose significant restrictions on the surface $S$:
If the degree of $S$ in $P^4$ (which agrees with $d$) is $1$, then
the only possibility is the plane $S=\pr^2$, in which case $\sigma (S)=1$.
If $d=2$, then $S=\pr^1 \times \pr^1$ with signature $\sigma (S)=0$.
If $d=3$, then $S$
is either a complete intersection $S_{1,3}$ with signature
$\sigma (S_{1,3})=-5$, or the Hirzebruch surface $\mathbb{F}_1$
with $\sigma (\mathbb{F}_1)=0$.
If $d=4$, then $S$ is the complete intersection
$S_{1,4}$, or the Del Pezzo surface $S_{2,2}$, or else
the Veronese surface $V$.
The corresponding signatures are
$\sigma (S_{1,4})=-16$, $\sigma (S_{2,2})=-4$, $\sigma (V)=1$.
In the proof of the Rigidity Theorem we have seen that
$f_* L_c (X) = (\epsilon_* g^! L_c (X)) [\pr^{n-2}]$
with
\[ \epsilon_* g^! L_c (X) 
   = \sigma (S) + \langle L^1 (\nu_P), f_{S*} [S] \rangle =
     \sigma (S) + \frac{d(n-2)}{3}. \]
\end{proof}
One may continue the list of the proposition somewhat beyond degree $d=4$;
the results become less and less definite, of course.

\section{Hodge-Theoretic Characteristic Classes}
\label{sec.hodgeclasses}

For an algebraic variety $X$, let $K^\alg_0 (X)$ denote the Grothendieck
group of the abelian category of coherent sheaves of $\Oo_X$-modules. 
When there is no danger of
confusion with other $K$-homology groups, we shall also write $K_0 (X) = K^\alg_0 (X)$.
Let $K^0 (X) = K_\alg^0 (X)$ denote the Grothendieck group of 
the exact category of algebraic vector
bundles over $X$.
The tensor product $\otimes_{\Oo_X}$ induces a 
\emph{cap product}
\[ \cap: K^0 (X) \otimes K_0 (X) \longrightarrow K_0 (X),~
  [E] \cap [\Fa] = [E \otimes_{\Oo_X} \Fa]. \]
Thus,
\begin{equation} \label{equ.kthtokhomalg} 
-\cap [\Oo_X]: K^0 (X)\longrightarrow K_0 (X) 
\end{equation}
sends a vector bundle $[E]$ to its associated (locally free) sheaf of germs of local
sections $[E\otimes \Oo_X]$.
If $X$ is smooth, then $-\cap [\Oo_X]$ is an isomorphism.

Let $X$ be a complex algebraic variety and $E$ an algebraic vector bundle
over $X$. For a nonnegative integer $p$, let $\Lambda^p (E)$ denote the
$p$-th exterior power of $E$.
The \emph{total $\lambda$-class} of $E$ is by definition
\[ \lambda_y (E) = \sum_{p\geq 0} \Lambda^p (E)\cdot y^p, \]
where $y$ is an indeterminate functioning as a bookkeeping device.
This construction induces a homomorphism
$\lambda_y (-): K^0_\alg (X) \longrightarrow K^0_\alg (X)[y]$
from the additive group of $K^0 (X)$ to the multiplicative monoid of
the polynomial ring $K^0 (X)[y]$.
Now let $X$ be a smooth variety, let $TX$ denote its tangent bundle and
$T^* X$ its cotangent bundle.
Then $\Lambda^p (T^* X)$ is the vector bundle ``of $p$-forms on $X$''.
Its associated sheaf of sections is denoted by $\Omega^p_X$. Thus
\[ [\Lambda^p (T^* X)] \cap [\Oo_X] = [\Omega^p_X] \]
and hence
\[ \lambda_y (T^* X) \cap [\Oo_X] = \sum_{p=0}^{\dim X} [\Omega^p_X] y^p. \]
Let $X$ be a complex algebraic variety and let
$MHM (X)$ denote the abelian category of M. Saito's algebraic mixed
Hodge modules on $X$. 
Totaro observed in \cite{totaro} that Saito's construction of a pure
Hodge structure on intersection homology implicitly contains a
definition of certain characteristic homology classes for singular
algebraic varieties. 
The following definition is based on this observation and due to 
Brasselet, Schürmann and Yokura, \cite{bsy}, see also the expository paper
\cite{schuermannmsri}.
\begin{defn}
The \emph{motivic Hodge Chern class transformation}
\[ MHC_y: K_0 (MHM(X)) \to K^\alg_0 (X) \otimes \intg [y^{\pm 1}] \]
is defined by
\[ MHC_y [M]
= \sum_{i,p} (-1)^i [\Ha^i (Gr^F_{-p} DR[M])] (-y)^p. \]
\end{defn}
A flat morphism $f:X\to Y$ gives rise to a flat pullback
$f^*: \operatorname{Coh} (Y)\to \operatorname{Coh}(X)$
on coherent sheaves, which is exact and hence induces a flat pullback
$f^*_K: K^\alg_0 (Y)\to K^\alg_0 (X)$. This applies in particular
to smooth morphisms and is then often called smooth pullback.
An arbitrary algebraic morphism $f:X \to Y$ 
(not necessarily flat) induces a homomorphism
\[ f^*: K_0 (MHM(Y)) \longrightarrow K_0 (MHM (X)) \]
which corresponds under the functor
$\qrat: D^b MHM (-)\to D^b_c (-;\rat)$
to $f^{-1}$ on constructible complexes of sheaves.
We record Schürmann's \cite[Cor. 5.11, p. 459]{schuermannmsri}:
\begin{prop} \label{prop.mhcvrrsmpullb}
(Verdier-Riemann-Roch for smooth pullbacks.)
For a smooth morphism $f:X\to Y$ of complex algebraic varieties,
the Verdier Riemann-Roch formula
\[ \lambda_y (T^*_{X/Y}) \cap f^*_K MHC_y [M] =
  MHC_y (f^* [M]) = MHC_y [f^* M] \]
holds for $M \in D^b MHM (Y),$ where $T^*_{X/Y}$ denotes
the relative cotangent bundle of $f$.
\end{prop}

Let $E$ be a complex vector bundle and let $a_i$ denote the Chern roots of $E$.
In \cite{hirzebruch},
Hirzebruch introduced a cohomological characteristic class
\[ T^*_y (E) = \prod_{i=1}^{\rk E} Q_y (a_i),
\]
where $y$ is an indeterminate, coming from the power series
\[ Q_y (a) = \frac{a (1+y)}{1-e^{-a (1+y)}} - a y \in \rat [y][[a]]. \]
If $R$ is an integral domain over $\rat$, then a power series
$Q(a)\in R[[a]]$ is called \emph{normalized} if it starts with $1$, i.e.
$Q(0)=1$.
With $R=\rat [y],$ we have $Q_y (0)=1$, so $Q_y (a)$ is normalized.
For $y=0,$
\begin{equation} \label{equ.t0istd} 
T^*_0 (E) = \prod_{i=1}^{\rk E} \frac{a_i}{1-e^{-a_i}} = \td^* (E)
\end{equation}
is the classical Todd class of $E$, while 
for $y=1,$
\begin{equation} \label{equ.todd1ishirzelcohom}
T^*_1 (E) = \prod_{i=1}^{\rk E} \frac{a_i}{\tanh a_i} = L^* (E)
\end{equation}
is the Hirzebruch $L$-class of the vector bundle $E$, as in Section \ref{sec.lclassintro}.
We shall also need a certain unnormalized version of $Q_y (a)$:
Let 
\[ \widetilde{Q}_y (a) = \frac{a (1+ye^{-a})}{1-e^{-a}} \in \rat [y][[a]] \]
and set
\[ \widetilde{T}^*_y (E) = \prod_{i=1}^{\rk E} \widetilde{Q}_y (a_i). \]
Note that $\widetilde{Q}_y (0) = 1+y \not= 1$, whence
$\widetilde{Q}_y (a)$ is unnormalized.
The relation
\[ (1+y) Q_y (a) = \widetilde{Q}_y ((1+y)a) \]
implies:
\begin{prop} \label{prop.relunnormtynormty}
If $E$ is a complex vector bundle of complex rank $r$, then for the degree $2i$ components:
\[ \widetilde{T}^i_y (E) = (1+y)^{r-i} T^i_y (E). \]
\end{prop}
More conceptually, we have the following formula for the
unnormalized class:
\begin{prop} \label{prop.unnormtytdtimeschlambda}
For any complex vector bundle $E$, we have
\[ \widetilde{T}^*_y (E) = \td^* (E) \cup \ch^* (\lambda_y (E^*)). \]
\end{prop}
Let $\tau_*: K_0 (X) \longrightarrow H^\BM_{2*} (X)\otimes \rat$
denote the Todd class transformation of Baum, Fulton, MacPherson.
We review, to some extent, construction and properties of this transformation.
Let 
\[ \alpha^*: K^0_\alg (X) \longrightarrow K^0_{\operatorname{top}} (X) \]
be the forget map which takes an algebraic vector bundle to its
underlying topological vector bundle.
Composing with the Chern character, one obtains a transformation
\[ \tau^* = \ch^* \circ \alpha^*: K^0_\alg (X) \longrightarrow H^{2*} (X;\rat), \]
see \cite[p. 180]{bfm2}.
Baum, Fulton and MacPherson construct a corresponding
homological version
\[ \alpha_*: K_0^\alg (X) \longrightarrow K_0^{\operatorname{top}} (X) \]
for quasi-projective varieties $X$.
Composing with the homological Chern character
\[ \ch_*: K^{\operatorname{top}}_0 (X) \longrightarrow
  H^\BM_{2*} (X;\rat), \]
where $H^\BM_*$ denotes Borel-Moore homology,  
they obtain a transformation
\[ \tau_* = \ch_* \circ \alpha_*: K_0^\alg (X) \longrightarrow H^\BM_{2*} (X;\rat). \]
This transformation is in fact available for any 
algebraic scheme over a field and generalizes the 
Grothendieck Riemann-Roch theorem to singular varieties.
\begin{remark} \label{rem.bfmtoddtochow}
Let $A_* (V)$ denote Chow homology of a variety $V$, i.e.
algebraic cycles in $V$ modulo rational equivalence.
Then there is a transformation
\[ \tau_*: K_0^\alg (X) \longrightarrow A_* (X)\otimes \rat \]
such that
\[ \xymatrix{
K_0^\alg (X) \ar[rd]^{\tau_*} \ar[d]_{\tau_*} & \\
A_* (X)\otimes \rat \ar[r]_{\cl} & H^\BM_{2*} (X;\rat)
} \]
commutes, where $\cl$ is the cycle map; see the first commutative
diagram on p. 106 of \cite[(0.8)]{bfm1}.
The construction of $\tau_*$ to Chow homology is described in
Fulton's book \cite[p. 349]{fultonintth}.
Thus Todd classes are
algebraic cycles that are well-defined up to rational equivalence over $\rat$.
\end{remark}
According to \cite[Theorem, p. 180]{bfm2}, $\tau_*$ and $\tau^*$ are compatible
with respect to cap products, i.e. the diagram
\[ \xymatrix@C=50pt{
K^0 (X) \otimes K_0 (X) \ar[r]^{\tau^* \otimes \tau_*} \ar[d]_\cap
& H^* (X;\rat) \otimes H^\BM_* (X;\rat) \ar[d]^\cap \\
K_0 (X) \ar[r]^{\tau_*} & H^\BM_* (X;\rat)
} \]
commutes. Thus, if $E$ is a vector bundle and $\Fa$ a coherent sheaf on $X$, then
\begin{equation} \label{equ.bfmtauofcap}
\tau_* ([E] \cap [\Fa]) = \ch^* (E) \cap \tau_* [\Fa].
\end{equation}
For smooth $X$, 
\[ \tau_* [\Oo_X] = \td^* (TX)\cap [X] = T^*_0 (TX)\cap [X]. \]
So if $E$ is a vector bundle on a smooth variety, then
\begin{equation} \label{taubfmecaposmoothx}
\tau_* ([E] \cap [\Oo_X]) = (\ch^* (E) \cup \td^* (TX))\cap [X].
\end{equation}
For locally complete intersection morphisms $f:X\to Y$,
Gysin maps 
\[ f^*_\BM: H^\BM_* (Y) \longrightarrow H^\BM_{*-2d} (X) \]
have been defined by Verdier \cite[\S 10]{verdierintcompl},
and Baum, Fulton and MacPherson \cite[Ch. IV, \S 4]{bfm1}, where
$d$ denotes the (complex) virtual codimension of $f$.
Thus for a regular closed embedding $g$, there is a Gysin
map $g^*_\BM$ on Borel-Moore homology, which we shall also 
write as $g^!$, and for a smooth morphism
$f$ of relative dimension $r$, there is a smooth pullback
$f^*_\BM: H^\BM_* (Y) \to H^\BM_{*+2r} (X)$.
Baum, Fulton and MacPherson show:
\begin{prop} \label{prop.bfmvrrsmpullb}
(Verdier-Riemann-Roch for smooth pullbacks.)
For a smooth morphism $f:X\to Y$ of complex algebraic varieties
and $[\Fa] \in K^\alg_0 (Y)$,
\[ \td^* (T_{X/Y}) \cap f^*_\BM \tau_* [\Fa] = \tau_* (f^*_K [\Fa]). \]
\end{prop}
Yokura \cite{yokura} twisted $\tau_*$ by a Hirzebruch-type variable $y$:
\begin{defn}
The \emph{twisted Todd transformation}
\[ \td_{1+y}: K_0 (X) \otimes \intg [y^{\pm 1}] \longrightarrow
                H^\BM_{2*} (X)\otimes \rat [y^{\pm 1}, (1+y)^{-1}] \]
is given by
\[ \td_{1+y} [\Fa] := \sum_{k\geq 0} \tau_k [\Fa]\cdot \frac{1}{(1+y)^k}, \]               
where the Baum-Fulton-MacPherson transformation
$\tau_*$ is extended linearly over $\intg [y^{\pm 1}]$, and
$\tau_k$ denotes the degree $2k$-component of $\tau_*$.
\end{defn}
\begin{remark} \label{rem.twistedtoddtochow}
Regarding the transformation $\tau_*$ as taking values in Chow
groups $A_* (-)\otimes \rat$ (cf. Remark \ref{rem.bfmtoddtochow}),
the above definition yields
a twisted Todd transformation
\[ \td_{1+y}: K_0 (X) \otimes \intg [y^{\pm 1}] \longrightarrow
                A_* (X)\otimes \rat [y^{\pm 1}, (1+y)^{-1}], \]
which commutes with the Borel-Moore twisted Todd transformation under the
cycle map.
\end{remark}
The definition of the motivic Hirzebruch class transformation below is due to 
Brasselet, Schürmann and Yokura \cite{bsy}, see also Sch\"urmann's expository paper
\cite{schuermannmsri}.
\begin{defn}
The \emph{motivic Hirzebruch class transformation} is
\[ MHT_{y*} := \td_{1+y} \circ MHC_y:
  K_0 (MHM(X)) \longrightarrow H^\BM_{2*} (X) \otimes 
    \rat [y^{\pm 1}, (1+y)^{-1}]. \] 
\end{defn}
For the intersection Hodge module $IC^H_X$ on a complex
purely $n$-dimensional variety $X$, we use the convention
\[ IC^H_X := j_{!*} (\rat^H_U [n]), \]
which agrees with \cite[p. 444]{schuermannmsri} and
\cite[p. 345]{peterssteenbrink}.
Here, $U \subset X$ is smooth, of pure dimension $n$,
Zariski-open and dense, and $j_{!*}$ denotes the
intermediate extension of mixed Hodge modules associated
to the open inclusion $j:U \hookrightarrow X$.
The underlying perverse sheaf is
$\qrat (IC^H_X) = IC_X,$ the intersection chain sheaf, where
$\qrat: MHM (X) \to \Per (X) = \Per (X;\rat)$ is the faithful and exact
functor that sends a mixed Hodge module to its underlying perverse sheaf.
Here, $\Per (X)$ denotes perverse sheaves on $X$ which are constructible
with respect to \emph{some} algebraic stratification of $X$.
This functor extends to a functor
$\qrat: D^b MHM (X) \to D^b_c (X) = D^b_c (X;\rat)$ between bounded derived
categories. For every object of $D^b_c (X)$ there exists \emph{some}
algebraic stratification with respect to which the object is constructible, and these
stratifications will generally vary with the object.
Recall that a functor $F$ is \emph{conservative}, if
for every morphism $\phi$ such that $F(\phi)$ is an isomorphism,
$\phi$ is already an isomorphism.
Faithful functors on balanced categories (such as abelian or
triangulated categories) are conservative.
According to \cite[p. 218, Remark (i)]{saitoextofmhm},
$\qrat: D^b MHM (X) \to D^b_c (X)$ is not faithful. But:
\begin{lemma} \label{lem.derratconservative}
The functor
$\qrat: D^b MHM (X) \to D^b_c (X)$ is conservative.
\end{lemma}
\begin{proof}
Let $\phi$ be a morphism in $D^b MHM (X)$ such that $\qrat (\phi)$ is an 
isomorphism in $D^b_c (X)$.
Applying the perverse cohomology functor
${}^p H^k: D^b_c (X)\to \Per (X),$
${}^p H^k (\qrat \phi)$ is an isomorphism in $\Per (X)$ for every $k$.
Now ${}^p H^k (\qrat \phi) = \qrat H^k (\phi)$, where $\qrat$ on the right
hand side is the faithful functor $\qrat: MHM (X) \to \Per (X)$.
It follows that $H^k (\phi)$ is an isomorphism in $MHM(X)$ for all $k$.
Thus $\phi$ is an isomorphism in $D^b MHM (X)$.
\end{proof}
The module $IC^H_X$ is the unique simple object in the
category $MHM (X)$ which restricts to $\rat_U [n]$ over $U$.
As $U$ is smooth and pure $n$-dimensional,
$\rat^H_U [n]$ is pure of weight $n$.
Since the intermediate extension $j_{!*}$ preserves weights, 
$IC^H_X$ is pure of weight $n$.
There is a duality isomorphism (polarization)
$\mathbb{D}^H_X IC^H_X \cong IC^H_X (n).$
Taking $\qrat$, this isomorphism induces a self-duality
isomorphism
\[  \mathbb{D}_X IC_X = \mathbb{D}_X \qrat IC^H_X \cong
\qrat \mathbb{D}^H_X IC^H_X \cong \qrat IC^H_X (n) \cong IC_X, \]
if an isomorphism $\rat_U (n) \cong \rat_U$ is chosen.
\begin{defn}
(\cite{bsy}, \cite{csgeneralattice}.)
The \emph{intersection generalized Todd class}
(or \emph{intersection Hirzebruch characteristic class}) is
\[ IT_{y*} (X) := MHT_{y*} [IC^H_X [-n]]
    \in H^\BM_{2*} (X) \otimes \rat [y^{\pm 1}, (1+y)^{-1}]. \]
\end{defn}
\begin{remark} \label{rem.itisalgebraic}
The intersection characteristic class $IT_{y*} (X)$ is represented by
an algebraic cycle by Remark \ref{rem.twistedtoddtochow}.
\end{remark}

\section{Behavior of the Hodge-Theoretic Classes 
  Under Normally Nonsingular Inclusions}
\label{sec.hodgeundernnsincl}

We embark on establishing a Verdier-Riemann-Roch type formula
$g^! IT_{1*} (X) = L^* (N) \cap IT_{1*} (Y)$ for appropriately
normally nonsingular regular algebraic embeddings 
$g: Y\hookrightarrow X$ of complex algebraic varieties.
Here, $g^!$ denotes Verdier's Gysin map on Borel-Moore homology 
for closed regular algebraic embeddings, and $N$ is the algebraic normal
bundle of $g$.
Following Verdier's construction of $g^!,$ one must first understand how
$IT_{1*} (X)$ behaves under specialization to homology of the algebraic
normal bundle. This then reduces the problem to establishing the desired formula in
the special case where $g$ is the zero section embedding into an algebraic
vector bundle. 
Philosophically, one may view the specialization map $\Sp_\BM$ as an
algebro-geometric substitute for
the simple topological operation of ``restricting a Borel-Moore cycle to an 
open tubular neighborhood of $Y$''. 
From this point of view, one expects that
$\Sp_\BM IT_{1*} (X) = IT_{1*} (N),$ and this is what we do indeed prove
(Proposition \ref{prop.spbmityisitn}).
That proof rests on three ideas:
First, in the context of deformation to the normal cone, the specialization map
to the central fiber can itself be expressed in terms of a \emph{hypersurface}
Gysin restriction.
Second, results of Cappell-Maxim-Sch\"urmann-Shaneson \cite{cmss} explain that a 
global hypersurface Gysin restriction applied to the motivic Hirzebruch class
transformation agrees with first taking Hodge nearby cycles, and then
executing the Hirzebruch transformation.
Third, we show that Saito's Hodge nearby cycle functor takes the intersection
Hodge module on the deformation space to the intersection Hodge module of the
special fiber $N$ (Proposition \ref{prop.psihpichzisichn}).
This requires in particular an analysis of the behavior of the Hodge intersection module
both under $g^!$ for topologically normally nonsingular closed algebraic embeddings
$g$ (Lemma \ref{lem.gshricyhcisicxh}), and under smooth pullbacks
(Lemma \ref{lem.smpullbic}). Vietoris-Begle techniques are being used after
careful premeditation of constructibility issues.
The remaining step is then to understand why the relation
$k^! IT_{1*} (N) = L^* (N) \cap IT_{1*} (Y)$ holds for the
zero section embedding $k:Y\hookrightarrow N$ of an algebraic vector bundle
$N\to Y$.
We achieve this in Proposition \ref{prop.ityrestrzerosect}.
In the case of a zero section embedding $k$, the Gysin restriction $k^!$ is, by the 
Thom isomorphism theorem, inverse to smooth pullback under the vector bundle
projection, and we find it easier to establish a relation for the latter
(Proposition \ref{prop.itysmpullb}). Sch\"urmann's $MHC_y$-Verdier-Riemann-Roch
theorem also enters.

Since algebraic normal bundles of regular algebraic embeddings need not
faithfully reflect the normal topology near the subvariety, 
the main result, Theorem \ref{thm.it1classgysin}, requires a tightness assumption, which
holds automatically in transverse situations (Proposition \ref{prop.topandalgtransvistight}).
Furthermore, our methods require that the exceptional divisor in the
blow-up of $X\times \cplx$ along $Y\times 0$ be normally nonsingular.
We do not know at present whether the latter condition, related to the
``clean blow-ups'' of Cheeger, Goresky and MacPherson, is necessary.
Again, it holds in transverse situations (Corollary \ref{cor.uptransvimplupnns}).\\

As regular algebraic embeddings need not be topologically normally
nonsingular, we define:
\begin{defn} \label{def.tightemb}
A closed regular algebraic embedding $Y\hookrightarrow X$
of complex algebraic varieties is called \emph{tight},
if its underlying topological embedding (in the complex topology)
is normally nonsingular 
and compatibly stratifiable (Definition \ref{def.compstrat}),
with topological normal bundle $\pi: E\to Y$
as in Definition \ref{def.snns},
and $E\to Y$ is isomorphic (as a topological vector bundle) to the 
underlying topological vector bundle
of the algebraic normal bundle $N_Y X$ of $Y$ in $X$.
\end{defn}
\begin{example}
A closed embedding $g: M\hookrightarrow W$ of smooth complex algebraic
varieties is tight because the normal bundles can be described in
terms of tangent bundles, and the smooth tubular neighborhood theorem
applies to provide normal nonsingularity (with respect to the intrinsic
stratification consisting of only the top stratum).
\end{example}

\begin{prop} \label{prop.topandalgtransvistight}
Let $M\hookrightarrow W$ be a
closed algebraic embedding
of smooth complex algebraic varieties. Let $X \subset W$
be a (possibly singular) algebraic subvariety, equipped
with an algebraic Whitney stratification
and set $Y = X \cap M$. If 
\begin{itemize}
\item each stratum of $X$ is transverse to $M$, and
\item $X$ and $M$ are Tor-independent in $W$,
\end{itemize}
then the embedding
$g:Y\hookrightarrow X$ is tight.
\end{prop}
\begin{proof}
The topological aspects of the proof proceed along the lines of 
\cite[p. 48, Proof of Thm. 1.11]{gmsmt}.
By smoothness, the closed embedding $M\hookrightarrow W$ is regular
with algebraic normal bundle $N_M W$.
The Tor-independence of $X$ and $M$ ensures that 
the closed embedding $Y\hookrightarrow X$ is also regular
(see \cite[Lemma 1.7]{illusietemkin}), and that the excess normal bundle vanishes, i.e.
the canonical closed embedding
$N_Y X \to j^* N_M W$ is an isomorphism of algebraic vector bundles,
where $j$ is the embedding $j: Y\hookrightarrow M$. The bundle $\pi: E\to Y$
in Definition \ref{def.tightemb} may then be taken to be the underlying topological vector bundle
of the restriction of $N_M W$ to $Y$.
\end{proof}

Let $V\hookrightarrow U$ be a closed regular embedding of complex varieties.
Then $\beta: \Bl_V U \to U$ will denote the blow-up of $U$ along $V$.
The exceptional divisor $E = \beta^{-1} (V) \subset \Bl_V U$ is the
projectivization $\pr (N)$ of the algebraic normal bundle $N$ of $V$ in $U$.
\begin{defn}
Let $X\hookrightarrow W \hookleftarrow M$ be closed algebraic embeddings
of algebraic varieties with $M,W$ smooth.
We say that these embeddings are
\emph{upwardly transverse}, if $X$ and $M$ 
are Tor-independent in $W$,
there exists an algebraic Whitney stratification of $X$ 
which is transverse to $M$ in $W$,
and there exists a (possibly non-algebraic) Whitney stratification on the 
strict transform of $X\times \cplx$ in $\Bl_{M\times 0} (W\times \cplx)$ 
which is transverse to the exceptional divisor.
\end{defn}

\begin{defn} \label{def.upwardlynns}
A tight embedding $Y\hookrightarrow X$ is called
\emph{upwardly normally nonsingular} if
the inclusion $E \subset \Bl_{Y\times 0} (X\times \cplx)$
of the exceptional divisor $E$
is topologically normally nonsingular.
\end{defn}
This notion is related to the \emph{clean blow-ups} of
Cheeger, Goresky and MacPherson \cite[p. 331]{cgm}.
A monoidal transformation $\pi: \widetilde{X} \to X$
with nonsingular center $Y\subset X$ is called a clean blow-up if
$E\to Y$ is a topological fibration, where $E=\pi^{-1} (Y)$ is the
exceptional divisor, and the inclusion
$E \subset \widetilde{X}$ is normally nonsingular.
A variety $X$ has \emph{algebraically conical singularities} if
it can be desingularized by a sequence of clean blow-ups.
\begin{prop} \label{prop.uptransvimplupnns}
Let $X\hookrightarrow W \hookleftarrow M$ be Tor-independent closed algebraic embeddings
of algebraic varieties with $M,W$ smooth.
If there exists an algebraic Whitney stratification of $X$ which is transverse to $M$,
and a Whitney stratification on the 
strict transform of $X$ in $\Bl_M W$ which is transverse to the
exceptional divisor, then 
the inclusion 
$Y=X\cap M \hookrightarrow X$ is tight and 
the inclusion $E' \subset \Bl_Y X$
of the exceptional divisor
is topologically normally nonsingular.
The corresponding topological normal vector bundle of
$E' \subset \Bl_Y X$ is then isomorphic to the restriction to $E'$ of
the tautological line bundle $\Oo_E (-1)$ over the exceptional divisor
$E \subset \Bl_M W$.
\end{prop}
\begin{proof}
Let $\beta: \Bl_M W \to W$ be the blow-up of $W$ along $M$.
As $M$ is a smoothly embedded smooth subvariety of the smooth variety $W$,
$\Bl_M W$ is smooth and the exceptional divisor
$E = \beta^{-1} (M) = \pr (N)$ is a smooth variety smoothly embedded in $\Bl_M W$.
By Lemma \ref{lem.homgysincommsqu} applied to
\[ \beta^{-1} (X) \hookrightarrow \Bl_M W \hookleftarrow E, \]
where the strict transform $\beta^{-1} (X)$ has been equipped with a
suitable Whitney stratification, the intersection
$\widetilde{Y} = \beta^{-1} (X) \cap E$ is an oriented pseudomanifold Whitney stratified by its 
intersection with the strata of $\beta^{-1} (X)$,
the inclusion $\widetilde{Y} \hookrightarrow \beta^{-1} (X)$ is normally nonsingular
with oriented topological normal bundle isomorphic to the restriction of 
the topological normal bundle $\nu_E$ of $E$ in $\Bl_M W$.
Now this normal bundle is the tautological bundle $\nu_E = \Oo_E (-1)$.

Since $X$ and $M$ are Tor-independent in $W$, and $M\hookrightarrow W$ is
a regular embedding,
the embedding $Y=X\cap M \hookrightarrow X$ is regular as well, and
the blow-up $\Bl_Y X$ of $X$ along $Y=X\cap M$ is given by
\[ \Bl_Y X = X \times_W \Bl_M W = \beta^{-1} (X), \]
\cite[Lemma 1.7]{illusietemkin}.
The exceptional divisor $E'$ of $\beta|: \Bl_Y X \to X$ is
\[ E' = \beta^{-1} (Y) = \beta^{-1} (X\cap M) = \beta^{-1} (X)\cap \beta^{-1} (M)
  = \beta^{-1} (X)\cap E = \widetilde{Y}. \]
The inclusion $Y\hookrightarrow X$ is tight by Proposition \ref{prop.topandalgtransvistight}.  
\end{proof}

\begin{cor} \label{cor.uptransvimplupnns}
If $X\hookrightarrow W \hookleftarrow M$ are upwardly transverse embeddings,
then the embedding $Y=X\cap M \hookrightarrow X$ is upwardly normally nonsingular.
The corresponding topological normal vector bundle of the exceptional divisor
$E' \subset \Bl_{Y \times 0} (X\times \cplx)$ 
is then isomorphic to the restriction to $E'$ of
the tautological line bundle $\Oo_E (-1)$ over the exceptional divisor
$E \subset \Bl_{M\times 0} (W\times \cplx)$.
\end{cor}
\begin{proof}
There exists an algebraic Whitney stratification $\Xa$
of $X$ such that
each stratum of $X$ is topologically transverse to $M$.
Since $X$ and $M$ are Tor-independent in $W$, the embedding
$Y=X\cap M \hookrightarrow X$ is regular, \cite[Lemma 1.7]{illusietemkin}.
Thus by Proposition \ref{prop.topandalgtransvistight},
the embedding $Y \hookrightarrow X$ is tight.
Equip $X' := X\times \cplx$ with the product stratification
$\Xa' = \Xa \times \cplx$. Then $\Xa'$ is an algebraic Whitney stratification
of $X'$ in $W' := W \times \cplx$.
Since $X$ is Whitney transverse to $M$ in $W$,
$X'$ is Whitney transverse to $M' := M\times 0$ in $W'$.
The Tor-independence of $X$ and $M$ in $W$ implies Tor-independence of
$X'$ and $M'$ in $W'$, since
\[ \Tor^{R[t]}_n (A[t],B) = \Tor^R_n (A, B) \]
for $R$-modules $A,B$; $R$ a $\cplx$-algebra.
Since $X\hookrightarrow W \hookleftarrow M$ are upwardly transverse embeddings,
there exists a Whitney stratification on the 
strict transform of $X'$ in $\Bl_{M'} (W')$ 
which is transverse to the exceptional divisor.
Hence, we may apply Proposition \ref{prop.uptransvimplupnns} to
$X'\hookrightarrow W' \hookleftarrow M'$.
It follows that the inclusion 
\[ Y'=X'\cap M' = (X\times \cplx)\cap (M\times 0) = (X\cap M)\times 0 = Y\times 0
  \hookrightarrow X'=X\times \cplx \]
is tight and 
the inclusion $E' \subset \Bl_{Y'} X'$
of the exceptional divisor
is topologically normally nonsingular.
In fact, the corresponding topological normal vector bundle of
$E' \subset \Bl_{Y'} X'$ is then isomorphic to the restriction to $E'$ of
the tautological line bundle $\Oo_E (-1)$ over the exceptional divisor
$E \subset \Bl_{M'} W'$.
\end{proof}

Let $Y\hookrightarrow X$ be a regular embedding with
normal bundle $N=N_Y X$. We recall briefly the technique of
\emph{deformation to the (algebraic) normal bundle}.
The embedding of $Y$ in $X$ gives rise to an embedding
$Y\times 0 \hookrightarrow X\times 0 \hookrightarrow X\times \cplx$
of $Y\times 0$ in $X\times \cplx$.
Let $Z = \Bl_{Y\times 0} (X\times \cplx)$
be the blow-up of $X\times \cplx$ along $Y\times 0,$ with
exceptional divisor $\pr (N\oplus 1)$.
The second factor projection $X\times \cplx \to \cplx$
induces a flat morphism $p_Z: Z\to \cplx$, whose
special fiber $p_Z^{-1} (0)$ is given by
\[ p_Z^{-1} (0) = \Bl_Y X \cup_{\pr (N)} \pr (N\oplus 1). \]
Let $Z^\circ = Z - \Bl_Y X.$
Then $p_Z$ restricts to a morphism
$p: Z^\circ \to \cplx,$
whose special fiber is
$p^{-1} (0) = N$
and whose general fiber is
$p^{-1} (t) \cong X\times \{ t \},~ t\in \cplx^*$.
\begin{prop} \label{prop.nbhdoftightembistautbndle}
Let $g:Y\hookrightarrow X$ be an upwardly normally nonsingular 
embedding of compact complex algebraic
varieties with associated deformation $p: Z^\circ \to \cplx$
to the (algebraic) normal bundle $N=N_Y X=p^{-1} (0)$.
Then there exists an open neighborhood 
(in the complex topology)
of $p^{-1} (0)$
in $Z^\circ$ which is homeomorphic to
$F|_N$, where $F\to \pr (N\oplus 1)$ is the topological
normal bundle to the exceptional divisor in $Z$.
\end{prop}
\begin{proof}
As $g$ is upwardly normally nonsingular,
$g$ is tight and the inclusion 
$\pr (N\oplus 1) \subset \Bl_{Y\times 0} (X\times \cplx)=Z$ of
the exceptional divisor is normally nonsingular.
In particular, $g$ is normally nonsingular and thus there is a
locally cone-like topological stratification
$\Xa = \{ X_\alpha \}$ of $X$ such that
$\Ya := \{ Y_\alpha := X_\alpha \cap Y \}$ is a locally
cone-like topological stratification of $Y$, and
there exists a topological vector bundle $\pi: E\to Y$ together with 
a topological embedding $j:E \to X$ such that
$j(E)$ is open in $X$, $j|_Y =g,$ and 
the homeomorphism $j:E\stackrel{\cong}{\longrightarrow} j(E)$
is stratum preserving, where the open set $j(E)$ is endowed with
the stratification $\{ X_\alpha \cap j(E) \}$ and $E$ is endowed
with the stratification $\Ea = \{ \pi^{-1} Y_\alpha \}$.

As the inclusion 
$\pr (N\oplus 1) \subset \Bl_{Y\times 0} (X\times \cplx)=Z$ 
is normally nonsingular,
there exists a topological vector bundle $\pi_F: F\to \pr (N\oplus 1)$ together with 
a topological embedding $J:F \to Z$ such that
$J(F)$ is open in $Z$ and $J|_{\pr (N\oplus 1)}$ is the inclusion
$\pr (N\oplus 1) \subset Z$.
As $N$ is open in $\pr (N\oplus 1)$,
the total space $F|_N$ is an open subset of $F$, and hence
$J(F|_N)$ is open in $J(F)$, which is open in $Z$.
Thus $J(F|_N)$ is open in $Z$.

Let $d$ be a metric on $Z$, whose metric topology agrees with the complex
topology on $Z$. Let
$r: \pr (N\oplus 1) \to \real_{\geq 0}$
be the continuous function defined by
$r(x) = \smlhf d(x, \Bl_Y X).$
If $x\in N\subset \pr (N\oplus 1)$, then $r(x)>0$ since $\Bl_Y X$ is compact.
If $x\in \pr (N) = \pr (N\oplus 1) - N$, then $r(x)=0$ since $\pr (N)\subset \Bl_Y X$.
Endow $F$ with the unique metric such that $J$ becomes an isometry.
Let $F_r \subset F$ be the open subset given by all vectors $v\in F$
of length $|v|:=d(0,v) <r (\pi_F (v))$.
Given a vector $v$ in $F_r \cap \pi^{-1}_F (N) = F_r$,
a triangle inequality argument shows that 
$v$ has positive distance from every point in $\Bl_Y X$,
from which we conclude that
$J(F_r|_N)$ and $\Bl_Y X$ are disjoint.
Hence $J(F_r|_N)\subset Z^\circ$.
As $F_r|_N$ is open in $F|_N$, we can find an open disc bundle
$F' \subset F_r|_N$ over $N$.
Then $J(F')$ is an open neighborhood of $N=p^{-1}(0)$ in $Z^\circ$
and the composition of $J$ with a fiber-preserving homeomorphism 
$F|_N \cong F'$ over $N$ yields a homeomorphism
$F|_N \cong F' \cong J(F')$.
\end{proof}
We shall next stratify $F|_N$ in a 
topologically locally cone-like fashion:
\begin{prop} \label{prop.lclstratnofoom1}
Assumptions and notation as in Proposition \ref{prop.nbhdoftightembistautbndle}.
Let $\widetilde{\pi}: F|_N \to N$ denote
the bundle projection. 
Let $\Ya = \{ Y_\alpha = X_\alpha \cap Y \}$ be the locally cone-like
topological stratification of $Y$
guaranteed by the normal nonsingularity of $Y$ in $X$.
Then the strata
\[ S_\alpha := \widetilde{\pi}^{-1} (\pi^{-1}_N Y_\alpha) \]
form a locally cone-like topological stratification
$\Sa = \{ S_\alpha \}$ of $F|_N$, where
$\pi_N$ denotes the vector bundle projection $\pi_N: N\to Y$.
\end{prop}
\begin{proof}
The strata $S_\alpha$ are topological
manifolds since the $Y_\alpha$ are topological manifolds and the total space of a 
(locally trivial) vector bundle over a topological manifold
is again a topological manifold.
Using local triviality of vector bundles and the fact that $\Ya$ is
locally cone-like, one constructs filtration preserving homeomorphisms
that show that $\Sa$ is locally cone-like.
\end{proof}
In order for Lemma \ref{lem.nearbycyclecplxconstructible},
concerning the constructibility of nearby cycle complexes,
to become applicable, we must refine the stratification
so that the central fiber becomes a union of strata:
\begin{lemma} \label{prop.refstratnofoom1}
The refinement of the locally cone-like topological stratification 
$\Sa$ of Proposition \ref{prop.lclstratnofoom1} given by
$\Sa' := \{ S_\alpha -N \} \cup \{ S_\alpha \cap N \}$
is again a locally cone-like topological stratification of
$F|_N$.
(Here we have identified $N$ with the zero section of
$F|_N$.)
\end{lemma}
\begin{proof}
Away from the zero section $N$, the strata of $\Sa'$ agree
with the strata of $\Sa$. So it suffices to prove that
points $v\in N$ on the zero section have cone-like neighborhoods
in a stratum preserving fashion.
According to Proposition \ref{prop.lclstratnofoom1}, $v$ has
an open neighborhood $W$ together with a homeomorphism
$W\cong \real^{2+2r+i} \times cL$, where $cL = \operatorname{cone}^\circ L$ denotes
the open cone on $L$.
This homeomorphism is stratum preserving if we endow $W$ with
the stratification induced from $\Sa$. It is, however, \emph{not} 
stratum preserving
if we endow $W$ with the stratification induced from $\Sa'$.
Let $L'$ be the join $L' = S^1 * L,$ a compact space.
Composing the homeomorphism
\begin{align*}
\real^{2+2r+i} \times cL
&= \real^{2r+i} \times (\real^2 \times cL) 
\cong \real^{2r+i} \times (cS^1 \times cL) \\
&\cong \real^{2r+i} \times c(S^1 * L) 
 \cong \real^{2r+i} \times cL'
\end{align*}
with above homeomorphism, we obtain a homeomorphism
$W \cong \real^{2r+i} \times cL'.$
We shall now stratify $L'$ in such a way that this homeomorphism
is stratum preserving if $W$ is equipped
with the stratification induced from $\Sa'$. This will finish the proof.
Let $A$ and $B$ be compact spaces with stratifications
$\Aa = \{ A_\alpha \},$ $\Ba = \{ B_\beta \},$ respectively.
The product stratification of $cA \times cB$ is given by
\begin{align*}
\Ca \Aa \times \Ca \Ba
&= \{ (0,1)\times A_\alpha \times (0,1) \times B_\beta \} 
 \cup \{ (0,1)\times A_\alpha \times \{ c_B \} \} \\
&\hspace{.5cm} \cup \{ \{ c_A \} \times (0,1) \times B_\beta \} 
 \cup \{ (c_A, c_B) \}.
\end{align*}
The join 
$A*B = cA \times B \cup_{A\times B} A \times cB$
is canonically stratified by
\[ \Ja =
   \{ (0,1)\times A_\alpha \times B_\beta \} \cup
   \{ \{ c_A \} \times B_\beta \} \cup
   \{ A_\alpha \times \{ c_B \} \}. \]
Therefore, the cone $c(A*B)$ has the canonical stratification
\begin{align*}
\Ca \Ja
&= \{ (0,1)\times (0,1) \times A_\alpha \times B_\beta \} 
 \cup \{ (0,1)\times A_\alpha \times \{ c_B \} \} \\
&\hspace{.5cm} \cup \{ (0,1)\times \{ c_A \} \times B_\beta \} 
 \cup \{ (c_A, c_B) \},
\end{align*}
which agrees with $\Ca \Aa \times \Ca \Ba$ under the homeomorphism
$cA \times cB \cong c(A*B)$. So this homeomorphism is stratum preserving
if we stratify as indicated.
We apply this with $A=S^1$, $B=L$, and $\Aa = \{ S^1 \}$ (one stratum).
Then, using the above method, the open disc $cA=cS^1=D^{\circ 2}$ receives
the stratification
\[ \Ca \Aa = \{ (0,1) \times S^1, c_A =0 \}, \]
where $0\in D^{\circ 2}$ denotes the center of the disc.
In the stratification $\Sa$ this disc is stratified with precisely
one stratum (the entire disc), while in the refined stratification
$\Sa'$, the disc must be stratified with two strata, namely the
center and its complement. As we have seen, this is achieved
automatically by the above cone stratification procedure.
Thus if we endow $L' = S^1 * L$ with the canonical join stratification
$\Ja$ described above, then the stratification $\Ca \Ja$
will agree with $\Ca \Aa \times \Ca \Ba$ under the homeomorphism
$cL' \cong cS^1 \times cL = D^{\circ 2} \times cL$ and
$D^{\circ 2} \times cL$ contains $\{ 0 \} \times cL$ as a union of strata,
as required by $\Sa'$.
\end{proof}

Via the homeomorphism of Proposition \ref{prop.nbhdoftightembistautbndle},
the locally cone-like topological stratification $\Sa'$ of
Lemma \ref{prop.refstratnofoom1} induces a
locally cone-like topological stratification $\Sa_U$ of a neighborhood
$U$ of $p^{-1} (0)=N$ in $Z^\circ$. In $\Sa_U$, the central fiber
$N$ is a union of strata. Hence, 
Lemma \ref{lem.nearbycyclecplxconstructible} below is applicable to the
stratification $\Sa_U$. We will apply the Lemma in this manner
in proving Proposition \ref{prop.psihpichzisichn} on the Hodge
nearby cycle functor applied to the intersection Hodge module
of the deformation space.
Saito's Hodge nearby cycle functor
$\psi^H_f$ is a functor
\[ \psi^H_f: MHM (V) \longrightarrow MHM (F), \]
where $f:V\to \cplx$ is an algebraic function with
central fiber $F=f^{-1} (0)$.
Deligne's nearby cycle functor $\psi_f$ 
does not preserve perverse sheaves, but the shifted functor
$\psi_f [-1]: \Per (V)\to \Per (F)$
does. Saito thus often uses the notation
${}^p \psi_f := \psi_f [-1]$ for the shifted functor.
Then the diagram
\[ \xymatrix{
MHM (V) \ar[r]^{\psi^H_f} \ar[d]_{\qrat} & MHM (F) \ar[d]^{\qrat} \\
\Per (V) \ar[r]^{\psi_f [-1]} & \Per (F)
} \]
commutes. It is also customary to write
$\psi'^H_f := \psi^H_f [1].$
Then one gets a commutative diagram
\begin{equation} \label{equ.nearbyandratshift} 
\xymatrix{
D^b MHM (V) \ar[r]^{\psi'^H_f} \ar[d]_{\qrat} & D^b MHM (F) \ar[d]^{\qrat} \\
D^b_c (V) \ar[r]^{\psi_f} & D^b_c (F)
} 
\end{equation}
So $\psi'^H_f$ lifts $\psi_f$ to the
derived category of mixed Hodge modules.
In proving Proposition \ref{prop.psihpichzisichn} below, we shall use
Schürmann's \cite[Lemma 4.2.1, p. 247]{schuermanntsscs}
in the following form:
\begin{lemma} \label{lem.nearbycyclecplxconstructible}
(Schürmann.)
Let $V$ be a topological space endowed with a locally cone-like
topological stratification and
let $p:V\to \cplx$ be a continuous function such that the subspace
$F = p^{-1} (0)$ is a union of strata.
If $\Fa$ is a constructible complex of sheaves on $V$,
then $\psi_p \Fa$ is constructible with respect to the induced
stratification of $F$, i.e. the restrictions of the cohomology
sheaves to strata are locally constant.
\end{lemma}

Let $g:Y\hookrightarrow X$ be a regular closed algebraic embedding with algebraic 
normal bundle
$N = N_Y X$. The associated deformation to the normal bundle
$p: Z^\circ \to \cplx$
comes with the commutative diagram
\[ \xymatrix{
& \{ 0 \} \ar@{^{(}->}[r] & \cplx & \cplx^* \ar@{_{(}->}[l] & \\
 & Y\times 0 \ar@{=}[ld] \ar[u] \ar@{^{(}->}[r]^{i_Y} \ar@{^{(}->}[d]^k 
   & Y\times \cplx \ar[u]_{p_Y} \ar@{^{(}->}[d]^G
    & Y\times \cplx^* \ar[u]_{\pro_2} \ar@{^{(}->}[d]^{g\times \id} \ar@{_{(}->}[l]_{j_Y} 
      \ar[r]^{\operatorname{pr}_1} & Y \ar@{^{(}->}[d]^g \\
Y & N \ar[l]^{\pi_N} \ar@{^{(}->}[r]^i \ar[d] & Z^\circ \ar[d]^p 
    & X\times \cplx^* \ar[d]^{\operatorname{pr}_2} \ar@{_{(}->}[l]_j 
     \ar[r]^{\operatorname{pr}_1} & X \\
& \{ 0 \} \ar@{^{(}->}[r] & \cplx & \cplx^*. \ar@{_{(}->}[l] &
} \]
The map $\pi_N:N\to Y$ is the bundle projection and
$k: Y\times 0 \hookrightarrow N$ its zero section.
The inclusions $i, i_Y$ are closed, while the embeddings
$j, j_Y$ are open. The map $G$ is the closed embedding of
$Y\times \cplx$ into the deformation space $Z^\circ$.

According to Saito \cite[p. 269]{saitomhmrims90}, the specialization
functor $\psi_p^H j_! (- \boxtimes \rat^H_{\cplx^*}[1])$
from mixed Hodge modules on $X$ to mixed Hodge modules on $N$
induces the identity on $MHM (Y)$, that is, the canonical morphism
\begin{equation} \label{equ.mhmrestrauxsommets}
k^* \psi^H_p [1] (j_! \pro^*_1 M) \longrightarrow
   g^* M
\end{equation}
for $M \in D^b MHM (X)$ is an isomorphism 
in $D^b MHM (Y)$.
Indeed, Verdier's property
``(SP5) Restrictions aux sommets''
(\cite[p. 353]{verdierspecfaiscmono}) asserts that
upon applying the functor $\qrat$ to (\ref{equ.mhmrestrauxsommets}),
the underlying morphism is an isomorphism in $D^b_c (Y)$.
Since the functor $\qrat$ is conservative by Lemma \ref{lem.derratconservative}, 
it follows that
(\ref{equ.mhmrestrauxsommets}) is an isomorphism in $D^b MHM (Y)$.

The behavior of the intersection Hodge module under normally nonsingular pullback
and normally nonsingular restriction is treated in the next lemmas.
We recall \cite[p. 443, Prop. 4.5]{schuermannmsri}:
\begin{lemma} \label{lem.pishrpidagvbproj}
Let $\pi:X\to Y$ be a morphism of algebraic varieties.
If $\pi$ is a smooth morphism of pure fiber dimension $r$, then
there is a natural isomorphism of functors
\[ \pi^! = \pi^* [2r](r): D^b MHM (Y) \longrightarrow D^b MHM (X). \]
\end{lemma}
The following result will be applied later in the case where the smooth
morphism is the projection of an algebraic vector bundle.

\begin{lemma} \label{lem.smpullbic}
Let $X$ and $Y$ be pure-dimensional complex algebraic varieties and let
$\pi: X\to Y$ be a smooth algebraic morphism 
of pure fiber dimension $r$.
Then
\[ \pi^* [IC^H_Y [r]] = [IC^H_X]  \]
under the smooth pullback 
$\pi^*: K_0 (MHM(Y)) \to K_0 (MHM(X)).$
\end{lemma}
\begin{proof}
According to Saito \cite[p. 257]{saitomhmrims90},
$MHM (-)$ is stable under smooth pullbacks. There is thus a functor
$\pi^* [r]: MHM (Y) \rightarrow MHM (X)$
and by Lemma \ref{lem.pishrpidagvbproj},
$\pi^! [-r] = \pi^* [r](r)$.
This functor is exact, which can be shown by an argument similar to the one
used to prove Lemma \ref{lem.gshrcrestrtomhm} below.
Let $V\subset Y$ be a Zariski-open, smooth, dense subset with inclusion 
$j: V\hookrightarrow Y$.
The preimage $j_U: U=\pi^{-1} (V) \hookrightarrow X$ is again
Zariski-open, smooth, and dense.
The restriction $\pi_U: U\to V$ of $\pi$ is again smooth of pure fiber dimension $r$,
so that in particular $\pi_U^! [-r] = \pi_U^* [r](r)$.
By \cite[p. 323, (4.4.3)]{saitomhmrims90}, the cartesian square
\[ \xymatrix{
U \ar[d]_{\pi_U} \ar[r]^{j_U} & X \ar[d]^\pi \\
V \ar[r]_j & Y
} \]
has associated base change natural isomorphisms 
$j_{U*} \pi_U^! \cong \pi^! j_*$ and
$j_{U!} \pi_U^* \cong \pi^* j_!$.
Using these, and the exactness of $\pi^* [r]$, we obtain isomorphisms
\[
\pi^* [r] (H^0 j_! \to H^0 j_*)
= H^0 (j_{U!} \pi_U^* [r]) \to H^0 (j_{U*} \pi_U^* [r]).
\]
Substitution of $\rat^H_V [n],$ where $n=\dim_\cplx Y$, gives
\begin{align*}
\pi^* [r] IC^H_Y
&= \pi^* [r] \im (H^0 j_! \rat^H_V [n] \to H^0 j_* \rat^H_V [n]) \\
&= \im \pi^* [r] (H^0 j_! \rat^H_V [n] \to H^0 j_* \rat^H_V [n]) \\
&= \im (H^0 (j_{U!} \pi_U^* [r] \rat^H_V [n]) 
        \to H^0 (j_{U*} \pi_U^* [r] \rat^H_V [n])) \\
&= \im (H^0 (j_{U!} \rat^H_U [n+r]) 
        \to H^0 (j_{U*} \rat^H_U [n+r])) \\
&= IC^H_X.
\end{align*}
\end{proof}

Given an algebraic stratification $\Sa$ of a complex algebraic variety $X$,
let $D^b_c (X,\Sa)$ denote the full subcategory of $D^b_c (X)$ consisting of all complexes
on $X$ which are constructible with respect to $\Sa$.
Similarly, we define $\Per (X,\Sa)$ to be
the full subcategory of $\Per (X)$ consisting of all perverse
sheaves on $X$ which are constructible with respect to $\Sa$.
The category $\Per (X,\Sa)$ is abelian and the inclusion functor
$\Per (X,\Sa) \to \Per (X)$ is exact. (For example, a kernel in $\Per (X)$ of a morphism
of $\Sa$-constructible perverse sheaves is itself $\Sa$-constructible and a kernel
in $\Per (X,\Sa)$.) Perverse truncation and cotruncation, and hence
perverse cohomology, restricts to $\Sa$-constructible objects:
\begin{equation} \label{equ.percohomrestricts}
\xymatrix{
D^b_c (X) \ar[r]^{{}^p H^k} & \Per (X) \\ 
D^b_c (X,\Sa) \ar[u] \ar[r]^{{}^p H^k} & \Per (X,\Sa). \ar[u] 
} \end{equation}
More generally, we may consider $D^b_c (X,\Sa)$ and $\Per (X,\Sa)$ on any space $X$
equipped with a locally cone-like topological stratification $\Sa$.
\begin{lemma} \label{lem.gshrcpervexact}
Let $X$ be an even-dimensional space equipped with a locally cone-like
topological stratification $\Sa$ whose strata are all even-dimensional.
Let $g:Y\hookrightarrow X$ be a
normally nonsingular topological
embedding of even real codimension $2c$ with respect to $\Sa$.
Let $\Ta$ be the locally cone-like stratification of $Y$ induced by $\Sa$.
Then the functor
$g^! [c]= g^* [-c]: D^b_c (X,\Sa) \to D^b_c (Y,\Ta)$ restricts to a functor
$g^! [c]= g^* [-c]: \Per (X,\Sa) \to \Per (Y,\Ta),$
which is exact.
\end{lemma}
\begin{proof}
First, as $g$ is normally nonsingular with respect to $\Sa$
(see Definition \ref{def.snns})
of real codimension $2c$, we have
$g^! = g^* [-2c]: D^b_c (X,\Sa) \to D^b_c (Y,\Ta),$
see e.g. \cite[p. 163, proof of Lemma 8.1.6]{banagltiss}.
Let $S_\alpha$ be the strata of $\Sa$ and
$s_\alpha: S_\alpha \hookrightarrow X$ the corresponding stratum inclusions.
By (1) of Definition \ref{def.snns}, a 
locally cone-like topological stratification of $Y$
is given by $\Ta = \{ T_\alpha \}$ with 
$T_\alpha = S_\alpha \cap Y$ and inclusions
$t_\alpha: T_\alpha \hookrightarrow Y$.
By (3) of Definition \ref{def.snns}, there is a topological vector bundle
$\pi: E\to Y$ and a topological embedding $j:E\hookrightarrow X$
onto an open subset $j(E)$ of $X$. The restricted homeomorphism $j:E\cong j(E)$
is stratum preserving, i.e. restricts further to a homeomorphism
$j|: \pi^{-1} (T_\alpha) \cong S_\alpha \cap j(E).$
Therefore,
\begin{equation} \label{equ.dimtadimsamc}
\dim_\real T_\alpha = \dim_\real S_\alpha -2c. 
\end{equation}
The category $\Per (X,\Sa)$ of $\Sa$-constructible perverse sheaves on $X$ is the
heart of $({}^p D^{\leq 0} (X,\Sa),$ ${}^p D^{\geq 0} (X,\Sa))$, the perverse t-structure
with respect to $\Sa$.
For $A^\bullet \in {}^p D^{\geq 0} (X,\Sa)$, one uses (\ref{equ.dimtadimsamc})
to verify that $g^! [c] A^\bullet \in {}^p D^{\geq 0} (Y,\Ta)$.
In particular, the functor
$g^! [c]= g^* [-c]: D^b_c (X,\Sa) \to D^b_c (Y,\Ta)$ is left-t-exact
with respect to the perverse t-structure.
Similarly, $A^\bullet \in {}^p D^{\leq 0} (X,\Sa)$ implies that
$g^! [c] A^\bullet \in {}^p D^{\leq 0} (Y,\Ta)$.
Hence, $g^! [c]=g^* [-c]$ is also
right-t-exact, and thus t-exact.
It follows that $g^! [c]= g^* [-c]: D^b_c (X,\Sa) \to D^b_c (Y,\Ta)$
preserves hearts. Moreover,
${}^p H^0 (g^! [c]) = g^! [c],$
and this functor is exact on the category of perverse sheaves, for example
by \cite[Prop. 7.1.15, p. 151]{banagltiss}.
\end{proof}

For a complex algebraic variety $X$ endowed with an algebraic stratification $\Sa$,
$MHM (X,\Sa)$ denotes the full subcategory of $MHM (X)$ whose objects are
those mixed Hodge modules $M$ on $X$ such that
$\qrat (M) \in \Ob \Per (X,\Sa).$
\begin{lemma}
The category $MHM (X,\Sa)$ is abelian and the inclusion functor
$MHM (X,\Sa) \to MHM (X)$ is exact.
\end{lemma}
\begin{proof}
We use the following general category-theoretic fact:
Let $F:A\to B$ be an exact functor between abelian categories.
Let $B'$ be a full subcategory of $B$ such that $B'$ is abelian and the inclusion 
functor $B' \to B$ is exact. Then the full subcategory $A'$ of $A$
given by 
\[ \Ob A' = \{ X\in \Ob A ~|~ \exists X' \in \Ob B' : F(X)\cong X' \} \]
is an abelian category and the inclusion functor
$A' \to A$ is exact.
In particular, if $B'$ is in addition isomorphism-closed in $B$, then
$A'$ with 
\[ \Ob A' = \{ X\in \Ob A ~|~ F(X)\in \Ob B' \} \]
is abelian with $A' \to A$ exact.
We apply this to the exact functor
$F=\qrat: MHM (X)\to \Per (X)$
and $B' = \Per (X,\Sa)$.
We noted earlier that $\Per (X,\Sa)$ is abelian and $\Per (X,\Sa)\to \Per (X)$
is exact.
Quasi-isomorphisms of complexes of sheaves preserve $\Sa$-constructibility.
Thus $\Per (X,\Sa)$ is isomorphism-closed in $\Per (X)$.
The statement of the lemma follows since $A'$ as described above agrees
in the application $F=\qrat$ with $MHM (X,\Sa)$.
\end{proof}
By definition, the functor $\qrat: MHM (X)\to \Per (X)$ restricts to
a functor $\qrat: MHM (X,\Sa)\to \Per (X,\Sa)$.
Since $\Per (X,\Sa)$ is isomorphism-closed in $\Per (X)$,
the subcategory $MHM (X,\Sa)$ is isomorphism-closed in $MHM (X)$.
The functor $\qrat: MHM (X,\Sa)\to \Per (X,\Sa)$ is exact and faithful.
Let $DM (X,\Sa)$ denote the full subcategory of $D^b MHM(X)$ whose objects
$M^\bullet$ satisfy $\qrat M^\bullet \in \Ob D^b_c (X,\Sa)$.
Thus by definition, $\qrat: D^b MHM (X)\to D^b_c (X)$ restricts to
\[ \qrat: DM (X,\Sa)\to D^b_c (X,\Sa), \]
which is still conservative.
We shall momentarily give an alternative description of $DM (X,\Sa)$ via cohomological
restrictions. We will use the following constructibility principle:
If $C^\bullet \in \Ob D^b_c (X)$ is a complex such that
${}^p H^k (C^\bullet)$ is $\Sa$-constructible for every $k$,
then $C^\bullet$ is $\Sa$-constructible.
\begin{lemma} \label{lem.altcharofdmxs}
The subcategory $DM (X,\Sa) \subset D^b MHM (X)$ equals the full subcategory
of $D^b MHM (X)$ whose objects $M^\bullet$ satisfy $H^k (M^\bullet) \in \Ob MHM (X,\Sa)$
for all $k$.
\end{lemma}
\begin{proof}
Let $M^\bullet$ be an object of $DM (X,\Sa)$.
Thus $\qrat M^\bullet$ is an object of $D^b_c (X,\Sa)$.
It follows that ${}^p H^k (\qrat M^\bullet) \in \Ob \Per (X,\Sa)$,
by (\ref{equ.percohomrestricts}).
Since ${}^p H^k (\qrat M^\bullet) = \qrat H^k (M^\bullet),$
the latter is an object of $\Per (X,\Sa)$. 
By the definition of $MHM (X,\Sa),$ $H^k (M^\bullet)$ is in $MHM (X,\Sa)$.

Conversely, let $M^\bullet$ be an object of $D^b MHM (X)$ such that
$H^k (M^\bullet) \in \Ob MHM (X,\Sa)$ for all $k$.
Then $\qrat H^k (M^\bullet) \in \Ob \Per (X,\Sa)$ for all $k$.
So ${}^p H^k (\qrat M^\bullet) \in \Ob \Per (X,\Sa)$ for all $k$, and this
implies, by the remark preceding the lemma, that
$\qrat M^\bullet \in \Ob D^b_c (X,\Sa)$.
Hence $M^\bullet \in \Ob DM (X,\Sa)$.
\end{proof}
Lemma \ref{lem.altcharofdmxs} implies:
\begin{lemma} 
The $MHM$-cohomology functor 
$H^k: D^b MHM (X) \to MHM (X)$ restricts to a functor
$H^k: DM (X,\Sa) \to MHM (X,\Sa)$,
\[ \xymatrix{
D^b MHM (X) \ar[r]^{H^k} & MHM (X) \\ 
DM (X,\Sa) \ar[u] \ar[r]^{H^k} & MHM (X,\Sa). \ar[u] 
} \]
\end{lemma}

The diagram
\[ \xymatrix{
D^b MHM (X) \ar[r]^{H^k} \ar[d]_{\qrat} & MHM (X) \ar[d]^\qrat \\
D^b_c (X) \ar[r]^{{}^p H^k} & \Per (X)
} \]
commutes (\cite[Lemma 14.5, p. 341]{peterssteenbrink}), whence the restricted diagram
\begin{equation} \label{equ.rathiispervhirat} 
\xymatrix{
DM (X,\Sa) \ar[r]^{H^k} \ar[d]_{\qrat} & MHM (X,\Sa) \ar[d]^\qrat \\
D^b_c (X,\Sa) \ar[r]^{{}^p H^k} & \Per (X,\Sa)
} \end{equation}
commutes as well.

\begin{lemma} \label{lem.gshrcrestrtomhm}
Let $X$ be a complex algebraic variety and
let $g:Y\hookrightarrow X$ be a closed algebraic embedding of
complex codimension $c$, whose underlying topological embedding is
normally nonsingular and compatibly stratifiable.
Let $\Sa$ be an algebraic stratification of $X$ compatible with the
normal nonsingularity of the embedding and such that the induced stratification
$\Ta$ on $Y$ is again algebraic.
Then the functor $g^* [-c]: D^b MHM (X) \to D^b MHM (Y)$ restricts to
a functor $g^* [-c]: MHM (X, \Sa) \to MHM (Y, \Ta),$ which is exact.
A similar statement applies to $g^! [c]$.
\end{lemma}
\begin{proof}
We start out by showing that $g^* [-c]: D^b MHM (X)\to D^b MHM (Y)$ restricts to a functor
$g^* [-c]: DM (X,\Sa)\to DM (Y,\Ta)$.
If $M^\bullet$ is an object of $DM (X,\Sa)$,
then $\qrat M^\bullet \in \Ob D^b_c (X,\Sa)$ and thus
$g^* [-c] (\qrat M^\bullet) \in \Ob D^b_c (Y,\Ta)$.
Now $g^* [-c] (\qrat M^\bullet) = \qrat (g^* [-c] M^\bullet)$, from which we
conclude that $g^* [-c] M^\bullet \in \Ob DM (Y,\Ta)$.

Let $P \in \Per (X, \Sa)$ be a perverse sheaf on $X$.
By Lemma \ref{lem.gshrcpervexact}, $g^* [-c] P \in \Per (Y, \Ta)$ and
hence
$g^* [-c] P = {}^p H^0 g^* [-c] P,$ while
${}^p H^k (g^* [-c] P)=0$ for $k\not= 0.$

The exact functor $\qrat: MHM (X)\to \Per (X)$
induces degreewise a functor $D^b MHM (X) \to D^b \Per (X).$
The ``realization'' functor
$\operatorname{real}: D^b \Per (X) \to D^b_c (X)$
of BBD \cite[p. 82, 3.1.9 and Prop. 3.1.10]{bbd} 
satisfies $\operatorname{real} \circ [-c] = [-c] \circ \operatorname{real},$
see \cite[p. 82, (3.1.9.3)]{bbd}.
Saito defines
$\qrat: D^b MHM (X) \to D^b_c (X)$
as the composition
\[ D^b MHM (X) \longrightarrow D^b \Per (X)
    \stackrel{\operatorname{real}}{\longrightarrow} D^b_c (X). \]
(See \cite[p. 222, Theorem 0.1]{saitomhmrims90}.)
Thus the diagram
\begin{equation} \label{equ.ratcommshift}
\xymatrix@R=15pt{
D^b MHM (X) \ar[r]^{[-c]} \ar[d]_{\qrat} & D^b MHM (X) \ar[d]^{\qrat} \\
D^b_c (X) \ar[r]^{[-c]} & D^b_c (X)
} 
\end{equation}
commutes.
Let $M \in MHM (X,\Sa)$ be a single mixed Hodge module, thought of as
an object in $DM (X,\Sa) \subset D^b MHM (X)$ concentrated in degree $0$.
Applying the functor $g^* [-c]$, we obtain an object
$g^* [-c] M \in DM (Y,\Ta)$.
By (\ref{equ.rathiispervhirat}),
$\qrat H^k (g^* [-c] M) = {}^p H^k (\qrat (g^* [-c] M)).$
Since $g^*$ on $D^b MHM$ lifts $g^*$ on $D^b_c$, 
we have
${}^p H^k (\qrat (g^* [-c] M)) = {}^p H^k (g^* \qrat ([-c] M)).$
By the commutativity of diagram (\ref{equ.ratcommshift}),
${}^p H^k (g^* \qrat ([-c] M)) =
   {}^p H^k (g^* [-c] \qrat M).$
Now $P = \qrat M$ is a perverse ($\Sa$-constructible) sheaf on $X$ and hence, as
observed above,
${}^p H^k (g^* [-c] \qrat M) =0$ for $k\not=0.$
We conclude that
\[ \qrat H^k (g^* [-c] M) =0 \text{ for } k\not=0. \]
Since $\qrat: MHM (Y) \to \Per (Y)$ is faithful,
$H^k (g^* [-c] M) =0$ for $k\not=0$.
So in $DM (Y,\Ta),$ there is a natural isomorphism
$H^0 (g^* [-c] M) = g^* [-c] M,$
given by composing the natural quasi-isomorphisms
\[ \tau_{\geq 0} \tau_{\leq 0} g^* [-c] M \longrightarrow
   \tau_{\geq 0} g^* [-c] M \longleftarrow
       g^* [-c] M. \]
This shows that $g^* [-c] M$ is canonically quasi-isomorphic to the
single mixed Hodge module $H^0 (g^* [-c] M) \in MHM (Y,\Ta)$.

Let 
\[ \xymatrix{
\Aa \ar[r]^A \ar[d]_F & \Aa' \ar[d]^{F'} \\
\Ba \ar[r]^B & \Ba'
} \]
be a commutative diagram of additive functors between abelian categories
with $F,F'$ exact and $F'$ faithful. If $B$ is exact, then
$A$ is exact.
Applying this to the commutative diagram of functors
\[ \xymatrix{
MHM (X,\Sa) \ar[r]^{g^* [-c]} \ar[d]_{\operatorname{rat}_X} 
   & MHM (Y,\Ta) \ar[d]^{\operatorname{rat}_Y} \\
\Per (X,\Sa) \ar[r]^{g^* [-c]} & \Per (Y,\Ta),
} \]
with $\operatorname{rat}_Y, \operatorname{rat}_X$
faithful and exact, and
$g^* [-c]$ on perverse sheaves exact, we conclude that
$g^* [-c]: MHM (X,\Sa) \to MHM (Y,\Ta)$ is exact.
The argument for $g^! [c]$ is entirely analogous.
\end{proof}

\begin{lemma} \label{lem.gshricyhcisicxh}
Let $X,Y$ be pure-dimensional complex algebraic varieties.
Let $g:Y \hookrightarrow X$ be a closed algebraic (not necessarily regular)
embedding of complex codimension $c$, whose underlying topological embedding
is normally nonsingular and compatibly stratifiable.
Then there is an isomorphism 
$g^* IC^H_X [-c] = IC^H_Y.$
\end{lemma}
\begin{proof}
By compatible stratifiability, there exists an algebraic stratification
$\Sa$ of $X$ such that $g$ is normally
nonsingular with respect to $\Sa$, and the induced stratification $\Ta$ of $Y$ 
is again algebraic.
Let $U\subset X$ be the top stratum of $\Sa$. Since $\Sa$ is algebraic
and $X$ is pure-dimensional,
$U$ is a Zariski-open, smooth, dense subset of $X$.
Let $j: U\hookrightarrow X$ be the corresponding inclusion.
The intersection $V=U\cap Y \hookrightarrow Y$ is 
the top stratum of $\Ta$ and hence also
Zariski-open, smooth (as a variety), and dense in $Y$.
Let $j_V: V\hookrightarrow Y$ be the corresponding inclusion.
By (3)(c) of Definition \ref{def.snns},
the restriction $g_V: V\to U$ of $g$ is again (algebraic and)
normally nonsingular (with respect to the intrinsic stratification
consisting of one stratum) of codimension $c$.
By \cite[p. 323, (4.4.3)]{saitomhmrims90}, the cartesian square
\[ \xymatrix{
V \ar[d]_{g_V} \ar[r]^{j_V} & Y \ar[d]^g \\
U \ar[r]_j & X
} \]
has associated base change natural isomorphisms 
$j_{V*} g_V^! \cong g^! j_*$ and
$j_{V!} g_V^* \cong g^* j_!$.
Let $m=\dim_\cplx X$ and $n=\dim_\cplx Y$ so that $c=m-n$.
The complexes
$j_! \rat_U [m]$ and $j_* \rat_U [m]$
are $\Sa$-constructible, e.g. by
\cite[Cor. 3.11.(iii), p. 79]{borel}.
Thus the objects 
$j_! \rat^H_U [m]$ and $j_* \rat^H_U [m]$ of $D^b MHM (X)$
belong in fact to $DM (X,\Sa)$.
Consequently, the canonical morphism
\[ H^0 j_! \rat^H_U [m] \longrightarrow H^0 j_* \rat^H_U [m] \]
is in the abelian category $MHM (X,\Sa)$.
Its image $IC^H_X \in MHM (X,\Sa)$ is the intersection Hodge module
on $X$. The exactness of the functor
$g^* [-c]: MHM (X,\Sa) \to MHM (Y,\Ta)$
provided by Lemma \ref{lem.gshrcrestrtomhm}
ensures that in $MHM (Y,\Ta),$
\begin{align*}
g^* [-c] IC^H_X 
&= \im g^* [-c] 
  (H^0 j_! \rat^H_U [m] \longrightarrow H^0 j_* \rat^H_U [m]) \\
&= \im 
  (H^0 g^* j_! \rat^H_U [m-c] \longrightarrow H^0 g^* j_* \rat^H_U [m-c]).
\end{align*}
We shall show that the normal nonsingularity of $g$ implies that the
natural morphism $g^* j_* \rat^H_U \to j_{V*} g^*_V \rat^H_U$ in
$DM(Y,\Ta)$ is an isomorphism.
As $\qrat: DM (Y,\Ta)\to D^b_c (Y,\Ta)$ is conservative, it suffices to prove
that the underlying morphism
$g^* j_* \rat_U \to j_{V*} g^*_V \rat_U$ is an isomorphism in $D^b_c (Y,\Ta)$.
As $g$ is normally nonsingular,
$g^! [c] = g^* [-c]$ on $D^b_c (X,\Sa)$ and, as
$g_V$ is normally nonsingular,
$g^!_V [c] = g^*_V [-c]$ on $D^b_c (U,\Sa \cap U)$.
Using the above base change isomorphism
$g^! j_* \cong j_{V*} g^!_V,$ we get a composition of isomorphisms
\[ g^* j_* \rat_U = g^! j_* \rat_U [2c] \cong j_{V*} g^!_V \rat_U [2c]  
         = j_{V*} g^*_V \rat_U \]
which factors 
$g^* j_* \rat_U \to j_{V*} g^*_V \rat_U$.
This establishes the claim. 
We deduce that the image above can be written as
\[ \im (H^0 j_{V!} g^*_V \rat^H_U [n] 
    \to H^0 j_{V*} g^*_V \rat^H_U [n]), \]
which, as $g^*_V \rat^H_U = \rat^H_V,$ is 
\[ 
\im (H^0 j_{V!} \rat^H_V [n] 
    \to H^0 j_{V*} \rat^H_V [n]) = IC^H_Y.
\]
\end{proof}

Let $g:Y\hookrightarrow X$ be a closed regular algebraic
embedding of pure-dimensional varieties
whose underlying topological embedding is
normally nonsingular
and compatibly stratifiable.
Take $M = IC^H_X [1]$ in the isomorphism
(\ref{equ.mhmrestrauxsommets}), shifted by $[-c]$,
to obtain an isomorphism
\begin{equation} \label{equ.knearbphicyisicx}
k^* \psi^H_p [1] (\pro^*_1 IC^H_X [1]) [-c] \cong
   g^* [-c] IC^H_X [1] \cong IC^H_Y [1], 
\end{equation}   
using Lemma \ref{lem.gshricyhcisicxh}.
\begin{prop} \label{prop.psihpichzisichn}
Let $X,Y$ be pure-dimensional
compact complex algebraic varieties.
Let $Y\hookrightarrow X$
be an upwardly normally nonsingular embedding
(Definition \ref{def.upwardlynns}) 
with algebraic normal bundle
$N = N_Y X$ and associated deformation to the normal bundle
$p: Z^\circ \to \cplx$. Then
\[ \psi^H_p IC^H_{Z^\circ} = IC^H_N,~ 
   \psi'^H_p IC^H_{Z^\circ} = IC^H_N [1]. \]
\end{prop}
\begin{proof}
By Lemma \ref{lem.smpullbic},  
$\pro^*_1 IC^H_X [1]= IC^H_{X\times \cplx^*}.$
Using the isomorphism (\ref{equ.knearbphicyisicx}), 
which is applicable here as $Y\hookrightarrow X$ is tight
(and $X,Y$ pure-dimensional), and
tight embeddings are regular and
topologically normally nonsingular
in a compatibly stratifiable manner,
we obtain an isomorphism
\[ k^* \psi^H_p (IC^H_{X\times \cplx^*}) [-c] \cong
   IC^H_Y. \]
In $D^b MHM (N),$ we have the adjoint relation
\[ \Hom_{D^b MHM (N)} (\pi^*_N M_1, M_2) = 
    \Hom_{D^b MHM (X)} (M_1, \pi_{N*} M_2), \]
see \cite[p. 441, Thm. 4.1]{schuermannmsri}.
Thus there is an adjunction morphism
\begin{equation} \label{equ.piadjpsi} 
\pi^*_N \pi_{N*} \psi^H_p (IC^H_{X \times \cplx^*})
 \stackrel{\operatorname{adj}}{\longrightarrow}
   \psi^H_p (IC^H_{X \times \cplx^*}) 
\end{equation}
in $D^b MHM (N)$.
Taking $\qrat$, one obtains the adjunction morphism
\begin{equation} \label{equ.ratpiadjpsi} 
\pi^*_N \pi_{N*} \psi_p [-1] (IC_{X \times \cplx^*})
 \stackrel{\operatorname{adj}}{\longrightarrow}
   \psi_p [-1] (IC_{X \times \cplx^*}) 
\end{equation}   
in $D^b_c (N)$.
As $Y\hookrightarrow X$ is upwardly normally nonsingular,
and $X,Y$ compact,
Propositions \ref{prop.nbhdoftightembistautbndle},
\ref{prop.lclstratnofoom1}, and Lemma
\ref{prop.refstratnofoom1} all apply.
We obtain an open neighborhood $U$ of $N=p^{-1} (0)$ in $Z^\circ$
together with a locally cone-like topological
stratification $\Sa_U$ of $U$ such that
the central fiber $N\subset U$ is a union of strata,
and those strata are given by
$\Sa_\alpha \cap N = \pi^{-1}_N Y_\alpha$.
Taking nearby cycles is a local operation:
if $p': U\to \cplx$ denotes the restriction of $p$ to $U$,
then $\psi_p (-) = \psi_{p'} (-|_U)$.
In particular, 
$\psi_p [-1] (IC_{X \times \cplx^*}) =\psi_{p'} [-1] (IC_U)$.
 The complex $IC_U$ is constructible
with respect to the locally cone-like topological stratification $\Sa_U$, 
by topological invariance of intersection homology, see also 
\cite[V, Cor. 4.18, p. 95]{borel}.
Thus by Lemma \ref{lem.nearbycyclecplxconstructible},
$\psi_{p'} [-1] (IC_U)$ (and hence also $\psi_p [-1] (IC_{X \times \cplx^*})$)
is constructible with respect to the strata
$\Sa_\alpha \cap N = \pi^{-1}_N Y_\alpha$.
In particular, $\psi_p [-1] (IC_{X \times \cplx^*})$ is
cohomologically locally constant with respect to the 
strata $\pi^{-1}_N Y_\alpha$ so that Vietoris-Begle implies
that (\ref{equ.ratpiadjpsi}) is an isomorphism,
\cite[p. 164, Lemma 10.14(i)]{borel}.
Since $\qrat$ is conservative on $D^b MHM (N)$ by Lemma \ref{lem.derratconservative}, 
the adjunction morphism (\ref{equ.piadjpsi}) is an isomorphism
\[ \pi^*_N \pi_{N*} \psi^H_p (IC^H_{X \times \cplx^*})
 \cong \psi^H_p (IC^H_{X \times \cplx^*}). \]
Applying $k^* [-c]$ yields isomorphisms
\[
k^* \pi^*_N \pi_{N*} \psi^H_p (IC^H_{X \times \cplx^*}) [-c]
 \cong k^* \psi^H_p (IC^H_{X \times \cplx^*}) [-c] \cong IC^H_Y. \]
Since $k^* \pi^*_N$ is the identity, this is an isomorphism
$\pi_{N*} \psi^H_p (IC^H_{X \times \cplx^*}) 
 \cong IC^H_Y [c].$
Applying $\pi^*_N$, we get an isomorphism
\[
\pi^*_N \pi_{N*} \psi^H_p (IC^H_{X \times \cplx^*})
 \cong \pi^*_N IC^H_Y [c]. \]
By Lemma \ref{lem.smpullbic}, $\pi^*_N IC^H_Y [c] = IC^H_N$.
Hence
$\psi^H_p (IC^H_{X \times \cplx^*}) \cong IC^H_N.$
\end{proof}

Recall that a flat morphism $f:X\to Y$ gives rise to a flat pullback
$f^*_K: K^\alg_0 (Y)\to K^\alg_0 (X)$. 
\begin{prop} \label{prop.twistedtdsmpullb}
Let $Y$ be a complex algebraic variety and
$\pi: N\to Y$ an algebraic vector bundle over $Y$.
For any coherent sheaf $\Fa$ on $Y$,
\[ T^*_y (T_\pi) \cap \pi^*_\BM \td_{1+y} [\Fa] =
   \td_{1+y} (\lambda_y (T^*_\pi) \cap \pi^*_K [\Fa]). \]
\end{prop}
\begin{proof}
We define an Adams-type operation $\psi^j$, which operates on a
cohomology class $\xi$ of degree $2j$ by
$\psi^j (\xi) = (1+y)^j \cdot \xi$. Similarly, a homological 
Adams-type operation is given by
$\psi_k (x) = (1+y)^{-k} \cdot x$ on a degree-$2k$ homology class $x$.
The behavior of these operations in a cap product of a degree-$2(j-k)$ class $\xi$
and a degree $2j$-class $x$ is described by the formula
\[ \psi_k (\xi \cap x) = \psi^{j-k} (\xi) \cap \psi_j (x). \]
Let $r$ be the complex rank of $N$.
Note that if $x$ has degree $2(j-r)$, then $\pi^*_\BM (x)$ has degree $2j$.
Under smooth pullback,  one then has 
\[ (1+y)^r \psi_j \pi^*_\BM (x) = \pi^*_\BM (\psi_{j-r} x). \]
By the definition of $\td_{1+y}$ and (\ref{equ.bfmtauofcap}),
\begin{align*}
\td_{1+y} (\lambda_y (T^*_\pi) \cap \pi^*_K [\Fa])
&= \sum_{k\geq 0} \psi_k
  (\tau_* (\lambda_y (T^*_\pi) \cap \pi^*_K [\Fa]))_k \\
&= \sum_{k\geq 0} \psi_k
  (\ch^* (\lambda_y (T^*_\pi)) \cap \tau_* \pi^*_K [\Fa])_k.
\end{align*}
By BFM-VRR for smooth pullbacks (Proposition \ref{prop.bfmvrrsmpullb}), 
this equals
\[  \sum_{k\geq 0} \psi_k
  (\ch^* (\lambda_y (T^*_\pi)) \cup
    \td^* (T_\pi) \cap \pi^*_\BM \tau_* [\Fa])_k, \]
which by Proposition \ref{prop.unnormtytdtimeschlambda} is
\[  \sum_{k\geq 0} \psi_k
  (\widetilde{T}^*_y (T_\pi) \cap \pi^*_\BM \tau_* [\Fa])_k. \]
Computing the degree-$2k$ component in this expression, we get
\[  \sum_{k,j\geq 0} 
   \psi_k \left( \widetilde{T}^{j-k}_y (T_\pi) \cap (\pi^*_\BM \tau_* [\Fa])_j \right)
   = \sum_{k,j\geq 0} 
   \psi^{j-k} \widetilde{T}^{j-k}_y (T_\pi) \cap \psi_j (\pi^*_\BM \tau_* [\Fa])_j. \]
According to Proposition \ref{prop.relunnormtynormty}, this can be
written in terms of $T^*_y$ as
\begin{eqnarray*}
\lefteqn{\sum_{k,j\geq 0} \psi^{j-k} (1+y)^{r-(j-k)}   
   T^{j-k}_y (T_\pi) \cap \psi_j (\pi^*_\BM \tau_* [\Fa])_j} & & \\
& & = \sum_{k,j\geq 0} (1+y)^r   
   T^{j-k}_y (T_\pi) \cap \psi_j (\pi^*_\BM \tau_* [\Fa])_j \\
& & = \sum_{k,j\geq 0}
   T^{j-k}_y (T_\pi) \cap (1+y)^r \psi_j (\pi^*_\BM \tau_* [\Fa])_j \\
& & = \sum_{k,j\geq 0}
   T^{j-k}_y (T_\pi) \cap \pi^*_\BM (\psi_{j-r} \tau_{j-r} [\Fa]) \\   
& & = \sum_{i\geq 0}
   T^{i}_y (T_\pi) \cap \pi^*_\BM \sum_{k\geq 0} \psi_k \tau_k [\Fa] \\ 
& & = T^*_y (T_\pi) \cap \pi^*_\BM \td_{1+y} [\Fa].
\end{eqnarray*}
\end{proof}

\begin{thm} \label{thm.mhtsmpullb}
Let $Y$ be a complex algebraic variety and
$\pi: N\to Y$ an algebraic vector bundle over $Y$.
For $M \in D^b MHM (Y)$,
\[ T^*_y (T_\pi) \cap \pi^*_\BM MHT_{y*} [M] =
  MHT_{y*} (\pi^*_\MHM M). \]
\end{thm}
\begin{proof}
This follows readily from Proposition \ref{prop.twistedtdsmpullb}
together with Schürmann's $MHC_y$-Verdier-Riemann-Roch
(Proposition \ref{prop.mhcvrrsmpullb}):
\begin{align*}
MHT_{y*} (\pi^*_\MHM M)
&= \td_{1+y} MHC_y (\pi^*_\MHM M) 
  = \td_{1+y} (\lambda_y (T^*_\pi) \cap \pi^*_K MHC_y [M]) \\
&= T^*_y (T_\pi) \cap \pi^*_\BM \td_{1+y} MHC_y [M] 
  = T^*_y (T_\pi) \cap \pi^*_\BM MHT_{y*} [M].
\end{align*}
\end{proof}

\begin{prop} \label{prop.itysmpullb}
If $Y$ is a pure-dimensional complex algebraic variety and
$\pi: N\to Y$ an algebraic vector bundle over $Y$, then
\[ T^*_y (T_\pi) \cap \pi^*_\BM IT_{y*} (Y) = IT_{y*} (N). \]
\end{prop}
\begin{proof}
Using Theorem \ref{thm.mhtsmpullb} and
Lemma \ref{lem.smpullbic},
\begin{align*}
T^*_y (T_\pi) \cap \pi^*_\BM IT_{y*} (Y)
&= T^*_y (T_\pi) \cap \pi^*_\BM MHT_{y*} [IC^H_Y [-n]] 
 = MHT_{y*} (\pi^*_\MHM [IC^H_Y [-n]]) \\ 
&= MHT_{y*} [IC^H_N [-m]] 
  = IT_{y*} (N).
\end{align*}
\end{proof}

Given a closed algebraic embedding $Y\hookrightarrow X$,
a specialization map 
\begin{equation} \label{equ.spbmschuer}
\Sp_\BM: H^\BM_* (X) \longrightarrow H^\BM_* (N_Y X), 
\end{equation}
on Borel-Moore homology, where $N = N_Y X$ is the normal cone of $Y$ in $X$,
has been constructed by Verdier in \cite[\S 8]{verdierintcompl}.
As before, let $p:Z^\circ \to \cplx$ be the deformation to the
normal cone, obtained by restricting $p_Z: Z\to \cplx$.
It will be convenient to embed
the family $Z$ as a Zariski open dense subset into the
following family $W$:
The embedding of $Y$ in $X$ gives rise to an embedding
$Y\times 0 \hookrightarrow X\times 0 \hookrightarrow X\times \pr^1$.
Let $W = \Bl_{Y\times 0} (X\times \pr^1)$.
There is a flat morphism $p_W: W\to \pr^1$, whose
special fiber is given by
\[ p_W^{-1} (0) = p_Z^{-1} (0) = \Bl_Y X \cup_{\pr (N)} \pr (N\oplus 1). \]
Let $W^\circ = W - \Bl_Y X.$
Then $p_W$ restricts to a morphism
$p: W^\circ \to \pr^1,$
whose special fiber is
$p^{-1} (0) = N$.
The open complement $\pr^1 - \{ 0 \} \cong \cplx$ has preimage
$p^{-1} (\cplx) \cong X\times \cplx$.
As blow-ups are determined locally, 
the open dense embedding $X\times \cplx \subset X\times \pr^1$
induces an open dense embedding
$Z \subset W$ and an open dense embedding $Z^\circ \subset W^\circ$.
The advantage of $W$ over $Z$ is that the open complement
$\cplx$ of $\{ 0 \}$ in $\pr^1$ is contractible and has the structure
of a complex vector space, while neither is true for $\cplx^*$.
The factor projection
$\operatorname{pr}_1: X\times \cplx \to X$ induces a 
smooth pullback 
$\operatorname{pr}^*_{1,\BM}: H^\BM_* (X) \to
  H^\BM_{*+2} (X\times \cplx)$
on Borel-Moore homology.
(We continue to use real, not complex, indexing for Borel-Moore
homology.)
By the Thom isomorphism theorem, 
this suspension map is an isomorphism. Thus we may invert it and define
\begin{equation} \label{equ.deflimsto0}
\lim_{t\to 0} := \Sp_\BM \circ (\operatorname{pr}^*_{1,\BM})^{-1}: 
   H^\BM_{*+2} (X\times \cplx) 
  \longrightarrow H^\BM_* (N).  
\end{equation}

The closed embedding $i: N\hookrightarrow W^\circ$ is regular (with trivial
algebraic normal bundle pulled up from the trivial normal bundle
of $\{ 0 \}$ in $\pr^1$).
Thus there is a Gysin homomorphism
\[ i^! = i^*_\BM: H^\BM_{*+2} (W^\circ) \longrightarrow
   H^\BM_* (N). \]
As $N$ is a
hypersurface in $W^\circ$ defined globally as the zero set
$N = \{ p = 0 \},$ Theorem 1.5 of
Cappell-Maxim-Schürmann-Shaneson \cite{cmss}, applies and asserts:
\begin{prop}(Cappell, Maxim, Schürmann, Shaneson.) \label{prop.thm15cmss}
Let $Y\hookrightarrow X$ be a closed algebraic embedding with 
normal cone $N=N_Y X$ and associated deformation to the normal cone
$p: W^\circ \to \pr^1$.
Then the diagram 
\[ \xymatrix{
K_0 MHM (W^\circ) \ar[r]^{\psi'^H_p} \ar[d]_{MHT_{1*}} &
  K_0 MHM (N) \ar[d]^{MHT_{1*}} \\
H^\BM_{*+2} (W^\circ;\rat) \ar[r]^{i^!} & H^\BM_* (N;\rat)
} \]
commutes. (Actually, this holds more generally for $MHT_{y*}$, but we use
it only for $y=1$.)
\end{prop}
For a complex variety $V$, $A_k (V)$ denotes the Chow group of
algebraic $k$-cycles in $V$ modulo rational equivalence.
\begin{lemma} \label{lem.restropenlims0isishriek}
Let $Y\hookrightarrow X$ be a closed algebraic embedding with 
normal cone $N=N_Y X$ and associated deformation
$p: W^\circ \to \pr^1$ to the normal cone.
Then the diagram
\[ \xymatrix{
H^\BM_{*+2} (W^\circ) \ar[rr]^{i^!} \ar[rd]_{j^*_\BM} & &
   H^\BM_* (N) \\
 & H^\BM_{*+2} (X\times \cplx) \ar[ru]_{\lim_{t\to 0}}
} \]
commutes on algebraic cycles, where $j^*_\BM$ denotes restriction of a Borel-Moore
cycle to an open subset, i.e. the diagram commutes on the image
of the cycle map $\cl: A_{*+1} (W^\circ) \to H^\BM_{2*+2} (W^\circ)$.
\end{lemma}
\begin{proof}
There is a short exact sequence
\[ A_{*+1} (N) \stackrel{i_*}{\longrightarrow} 
   A_{*+1} (W^\circ) \stackrel{j^*_A}{\longrightarrow} 
   A_{*+1} (X\times \cplx) \longrightarrow 0, \]
where the map $i_*$ is proper pushforward under the proper map
$i: N\hookrightarrow W^\circ$, and $j^*_A$ is restriction to an open subset.
Let $i_A^!: A_{*+1} (W^\circ) \to A_* (N)$ denote the
Gysin map for divisors. Then the composition
$i_A^! \circ i_*$ is zero, since the algebraic normal bundle of $N$ in $W^\circ$
is trivial.
(Intuitively, the triviality of the normal bundle implies
that any cycle in $N$ can be pushed off of $N$ in $W^\circ$ and thus its
transverse intersection with $N$ is zero.)
By exactness, we may identify $A_{*+1} (X\times \cplx)$ with the
cokernel of $i_*$. Then $\im i_* \subset \ker i_A^!$ implies that
$i_A^!$ induces uniquely a map
\[ \lim_{t\to 0}^A: A_{*+1} (X\times \cplx) \longrightarrow
  A_* (N) \]
such that
\begin{equation}
\xymatrix{
A_{*+1} (W^\circ) \ar[rr]^{i_A^!} \ar[rd]_{j^*_A} & &
   A_* (N) \\
 & A_{*+1} (X\times \cplx) \ar[ru]_{\lim^A_{t\to 0}}
} 
\end{equation}
commutes.
Note that this is the diagram in the statement of the lemma, 
only on Chow instead of Borel-Moore.
To finish the proof, one uses that
the Gysin map of a regular embedding, as well as smooth pullback,
commute with the cycle map from Chow to Borel-Moore.
The Chow level specialization map
$\Sp_A: A_* (X) \longrightarrow A_* (N)$
is defined to be the composition
\[ A_* (X) \stackrel{\operatorname{pr}^*_{1,A}}{\longrightarrow}
   A_{*+1} (X\times \cplx) \stackrel{\lim_{t\to 0}^A}{\longrightarrow}
   A_* (N), \]
see \cite[p. 89, Proof of Prop. 5.2]{fultonintth} or
\cite[p. 15, (25)]{schuermannvrrcsmacp}, or
\cite[p. 198]{verdierintcompl}, and is known to
commute with the cycle map.
\end{proof}

\begin{prop} \label{prop.spbmityisitn}
Let $X,Y$ be pure-dimensional
compact complex algebraic varieties.
If $g:Y\hookrightarrow X$ is an upwardly normally nonsingular embedding
with algebraic normal bundle $N_Y X$, then
\[ \Sp_\BM IT_{1*} (X) = IT_{1*} (N_Y X). \]
\end{prop}
\begin{proof}
By Definition (\ref{equ.deflimsto0}),
$\Sp_\BM IT_{1*} (X) 
  = \lim_{t\to 0}~ \operatorname{pr}^*_{1,\BM} IT_{1*} (X).$
We regard $\pi := \operatorname{pr}_1: X\times \cplx \to X$ as the projection of
the trivial line bundle $1_X$ over $X$.
Then $T_\pi = \pi^* (1_X) = 1_{X\times \cplx}$ and
hence, using (\ref{equ.todd1ishirzelcohom}),
$T_1^* (T_\pi) = T_1^* (1_{X\times \cplx}) =
  L^* (1_{X\times \cplx}) = 1.$
By Proposition \ref{prop.itysmpullb},
\begin{align*}
\pi^*_\BM IT_{1*} (X)
= 1 \cap \pi^*_\BM IT_{1*} (X) 
= T_1^* (T_\pi) \cap \pi^*_\BM IT_{1*} (X) 
= IT_{1*} (X\times \cplx).
\end{align*}
With $n=\dim_\cplx X$, we thus have
\[ \Sp_\BM IT_{1*} (X) = 
 \lim_{t\to 0} IT_{1*} (X\times \cplx) =
  \lim_{t\to 0}
 MHT_{1*} (IC^H_{X\times \cplx} [-n-1]). \]
Let $j$ denote the open embedding
$j: X\times \cplx \hookrightarrow W^\circ$ associated to the
deformation to the normal bundle;
$j^{-1} IC^H_{W^\circ} = IC^H_{X\times \cplx}.$
Since the transformation $MHT_{y*}$ commutes with restriction
to open subsets,
\[ \Sp_\BM IT_{1*} (X) =
 \lim_{t\to 0} j^*_\BM MHT_{1*} (IC^H_{W^\circ} [-n-1]). \]
By Remark \ref{rem.itisalgebraic}, 
the class $MHT_{1*} (IC^H_{W^\circ} [-n-1]) = IT_{1*} (W^\circ)$
is algebraic.
Hence, by Lemma \ref{lem.restropenlims0isishriek},
\[ \lim_{t\to 0} j^*_\BM MHT_{1*} (IC^H_{W^\circ} [-n-1])
 = i^! MHT_{1*} (IC^H_{W^\circ} [-n-1]). \]
By the CMSS Proposition \ref{prop.thm15cmss},
\[ i^! MHT_{1*} (IC^H_{W^\circ} [-n-1]) =
  MHT_{1*} \psi'^H_p (IC^H_{W^\circ} [-n-1]). \]
Finally, by Proposition \ref{prop.psihpichzisichn} (which
requires upward normal nonsingularity of the embedding, 
pure-dimensionality and compactness),
\[ \psi'^H_p (IC^H_{W^\circ} [-n-1]) 
= \psi^H_p [1] (IC^H_{W^\circ} [-n-1])
= \psi^H_p (IC^H_{W^\circ} [-n])
= IC^H_N [-n]. \]
We conclude that
$\Sp_\BM IT_{1*} (X) = MHT_{1*} (IC^H_N [-n]) = IT_{1*} (N),$
since $\dim_\cplx N =n$.
\end{proof}
The following cap product formula for homological Gysin maps is standard, see e.g.
\cite[Ch. V, \S 6.2 (c), p. 35]{boardman}) or
\cite[Lemma 5, p. 613]{banaglcappshan}.
\begin{lemma} \label{lem.smpullbbmcap}
Let $Y$ be a complex algebraic variety and
let $\pi:N\to Y$ be an
algebraic vector bundle projection.
If $\eta \in H^* (Y)$ and $a\in H^\BM_* (Y)$
are classes in even degrees, then
\[ \pi^*_\BM (\eta \cap a) = \pi^* (\eta) \cap \pi^*_\BM (a). \]
\end{lemma}

\begin{prop} \label{prop.ityrestrzerosect}
Let $Y$ be a pure-dimensional complex algebraic variety and
let $k: Y \hookrightarrow N$ be the zero section of an
algebraic vector bundle projection $\pi:N\to Y$. Then
\[ k^! IT_{y*} (N) = T^*_y (N) \cap IT_{y*} (Y). \]
\end{prop}
\begin{proof}
By the Thom isomorphism theorem, 
the Gysin pullback $k^! = k^!_\BM$ and the smooth pullback $\pi^*_\BM$
are inverse isomorphisms on Borel-Moore homology, see
Chriss-Ginzburg \cite[Prop. 2.6.43, p. 107]{chrissginzburg}.
The relative tangent bundle of $\pi$
is given by $T_\pi = \pi^* N$. Since $T^*_y$ is a natural characteristic
class in cohomology, 
$T^*_y (T_\pi) = \pi^* T^*_y (N).$
Thus, using Proposition \ref{prop.itysmpullb} and
Lemma \ref{lem.smpullbbmcap}, we get
\begin{align*}
k^!_\BM IT_{y*} (N)
&= k^!_\BM (T^*_y (T_\pi) \cap \pi^*_\BM IT_{y*} (Y)) 
 = k^!_\BM (\pi^* T^*_y (N) \cap \pi^*_\BM IT_{y*} (Y)) \\
&= k^!_\BM \pi^*_\BM (T^*_y (N) \cap IT_{y*} (Y)) 
 = T^*_y (N) \cap IT_{y*} (Y).
\end{align*}
\end{proof}

\begin{lemma} \label{lem.gshriekisrestr0spbm}
Let $g:Y\hookrightarrow X$ be a closed regular embedding of possibly
singular varieties.
Let $N=N_Y X$ denote the algebraic normal bundle and
let $c$ be the complex codimension of $Y$ in $X$.
The Gysin map $g^!: H^\BM_* (X) \to H^\BM_{*-2c} (Y)$ factors
as
\[ \xymatrix{
H^\BM_{*} (X) \ar[rr]^{g^!} \ar[rd]_{\Sp_\BM} & &
   H^\BM_{*-2c} (Y) \\
 & H^\BM_* (N) \ar[ru]_{k^!},
} \]
where $k^!$ is the Gysin restriction to the zero section
and $\Sp_\BM$ is Verdier's Borel-Moore specialization map
(\ref{equ.spbmschuer}).
\end{lemma}
\begin{proof}
This is simply Verdier's description of the Gysin map as given
in \cite[p. 222]{verdierintcompl},
observing that $(\pi^*_\BM)^{-1} = k^!_\BM$ according to the
Thom isomorphism theorem on Borel-Moore homology.
\end{proof}

\begin{thm} \label{thm.it1classgysin}
Let $X,Y$ be pure-dimensional
compact complex algebraic varieties and
let $g: Y \hookrightarrow X$  
be an upwardly normally nonsingular embedding 
(Definition \ref{def.upwardlynns}). 
Let $N = N_Y X$ be the algebraic normal bundle of $g$
and let $\nu$ denote the topological normal bundle of the
topologically normally nonsingular inclusion underlying $g$.
Then
\[ g^! IT_{1,*} (X) = L^* (N) \cap IT_{1,*} (Y)
      = L^* (\nu) \cap IT_{1,*} (Y). \]
\end{thm}
\begin{proof}
By Lemma \ref{lem.gshriekisrestr0spbm},
$g^! IT_{1,*} (X) = k^! \Sp_\BM IT_{1,*} (X).$
Proposition \ref{prop.spbmityisitn}, which
requires upward normal nonsingularity of the embedding 
(as well as pure-dimensionality),
yields
$k^! \Sp_\BM IT_{1,*} (X) = k^! IT_{1,*} (N),$
while by Proposition \ref{prop.ityrestrzerosect},
$k^! IT_{1,*} (N) = T^*_1 (N) \cap IT_{1,*} (Y).$
Finally, we recall that $T^*_1 (N) = L^* (N)$, (\ref{equ.todd1ishirzelcohom}).
Since $g$ is tight, there is a bundle isomorphism $N\cong \nu$ of topological
vector bundles. Hence $L^* (N)=L^* (\nu)$.
\end{proof}

\end{document}